\documentclass[12pt]{article}
\usepackage[a4paper,margin=1.2in]{geometry}

\usepackage{enumitem}

\usepackage[T1]{fontenc}
\usepackage[utf8]{inputenc}

\newcommand{\Syn}{\mathsf{Syn}}
\newcommand{\HD}[2]{H^{#1}\mathbf D(#2)} 
\usepackage{amsmath}
\DeclareMathOperator{\Mech}{Mech}
\usepackage{lmodern}
\newcommand{\tors}{\operatorname{tors}}
\DeclareMathOperator{\Div}{Div}
\DeclareMathOperator{\Prin}{Prin}
\DeclareMathOperator{\Jac}{Jac}
\usepackage{amsmath,amsthm,amssymb,mathtools}
\DeclareMathOperator{\Aut}{Aut}
\DeclareMathOperator{\Map}{Map}
\DeclareMathOperator{\Sp}{Sp}
\DeclareMathOperator{\Tor}{Tor}
\DeclareMathOperator{\rank}{rank}
\DeclareMathOperator{\Prob}{Prob}
\DeclareMathOperator{\Fix}{Fix}
\newcommand{\Id}{\mathrm{Id}}
\DeclareMathOperator{\im}{im}
\DeclareMathOperator{\coker}{coker}

\newcommand{\idim}{\operatorname{idim}}
\usepackage{tikz-cd}
\usepackage{hyperref}
\newcommand{\Alg}{\mathrm{Alg}}
\usepackage{microtype}
\usepackage{cmll}
\usepackage{algorithm}
\newcommand{\FinVect}{\mathbf{FinVect}}
\newcommand{\RR}{\mathbb{R}}
\usepackage{algorithmic}

\newcommand{\ob}{\mathrm{ob}}
\newcommand{\E}{\mathcal{E}}
\newcommand{\F}{\mathcal{F}}
\newcommand{\Ab}{\mathsf{Ab}}
\newcommand{\Mod}{\mathsf{Mod}}
\newcommand{\ZZ}{\mathbb{Z}}
\newcommand{\QQ}{\mathbb{Q}}
\newcommand{\DD}{\mathbb{D}}
\newcommand{\Hom}{\operatorname{Hom}}
\newcommand{\End}{\operatorname{End}}
\newcommand{\Ext}{\operatorname{Ext}}
\newcommand{\id}{\mathrm{id}}
\newcommand{\op}{\mathrm{op}}
\newcommand{\RHom}{\mathbf{R}\!\operatorname{Hom}}
\newcommand{\Dplus}{D^{+}}
\newcommand{\coev}{\mathrm{coev}}
\newcommand{\ev}{\mathrm{ev}}

\theoremstyle{plain}
\newtheorem{theorem}{Theorem}[section]
\newtheorem{lemma}[theorem]{Lemma}
\newtheorem{proposition}[theorem]{Proposition}
\newtheorem{corollary}[theorem]{Corollary}

\theoremstyle{definition}
\newtheorem{definition}[theorem]{Definition}

\theoremstyle{remark}
\newtheorem{remark}[theorem]{Remark}

\numberwithin{equation}{section}

\title{Ethic Duality: A Homological Framework for Primal-Dual Problems}
\author{Dmitry Pasechnyuk-Vilensky $^{1, 2}$ and Martin Tak\'a\v{c} $^{1}$\\
$^{1}$ \quad MBZUAI, United Arab Emirates\\
$^{2}$ \quad ISP RAS, Russia}

\begin{document}
\maketitle

\begin{abstract}
We develop a homological duality framework based on a contravariant functor
$D=\operatorname{Hom}_E(-,R)$ with dualizing object $R$. A morphism is called
\emph{ethic} when it satisfies the canonical double-dual compatibility
$D^2(f)\eta=\eta f$. In the derived setting, the functor
$\mathrm{RHom}_E(-,R)$ produces a graded family of Ext-groups that measure all
failures of this compatibility. The first layer $\operatorname{Ext}^1$ identifies
primal--dual gaps, while higher $\operatorname{Ext}^k$ provide a systematic
hierarchy of derived obstructions to exactness.

This formulation specializes uniformly across several classical domains.  
In linear and conic optimization, Farkas- and Slater-type exactness criteria
correspond to the vanishing of $\operatorname{Ext}^1$, and integer duality gaps
coincide with torsion Ext-classes.  
In graph theory, Kirchhoff- and Baker--Norine-type dualities arise as instances of
ethic exactness. In dynamical systems, the higher derived layers encode
nonvanishing persistence phenomena. Additional examples include social-choice
configurations, categorical factorization in scattering formalisms, coding-theoretic
duality, and Bellman-type recurrences, all appearing as concrete instances of
Ext-controlled exactness.

All resulting invariants are stable under derived Morita equivalence and depend
only on the dualizing pair $(E,R)$. The framework therefore provides a
substrate-independent criterion for primal--dual exactness and a uniform
homological description of its obstructions.
\end{abstract}




\tableofcontents

\section*{Introduction}

Duality appears throughout mathematics in many unrelated guises: topological
duals of Banach and locally convex spaces, Verdier and Grothendieck duality in
geometry, Fenchel conjugation in convex analysis, and primal--dual formulations
in optimization.  What these settings share is not a specific construction but a
structural question: under what conditions does a dual assignment recover the
information of the original object, and how can one measure the extent to which
it fails to do so?

In this paper we develop a homological framework that formulates this question
inside an abelian category $\E$ equipped with a dualizing object $R$.  We study
the contravariant functor
\[
\mathbf D = \Hom_\E(-,R)
\]
together with its canonical comparison map $\eta : \Id \Rightarrow \mathbf D^2$.
A morphism $f$ is called \emph{ethic} if it satisfies
$\mathbf D^2(f)\,\eta=\eta\,f$, which expresses a strict form of compatibility
with respect to dualization.  All deviations from this condition are detected by
the derived dual $\RHom_\E(-,R)$: the cohomology groups
\[
H^k\mathbf D(A)=\Ext^k_\E(A,R)
\]
form a graded obstruction tower, with $\Ext^1$ representing the primary
obstruction to dual exactness and the higher terms describing the remaining
layers of derived inconsistency.  In the derived category $D^+(\E)$ these
obstructions give rise to an ethic $t$--structure whose heart consists of the
objects completely determined by their duals.

This homological formulation accommodates a variety of classical duality
phenomena within a single mechanism.  In convex and conic optimization,
vanishing of $\Ext^1$ recovers strong duality, and torsion in
$\Ext^1_\ZZ(M,\ZZ)$ captures algebraic duality gaps in integer programs.
Kirchhoff-type formulas and the Baker--Norine Riemann--Roch theorem in graph
theory arise as instances of ethic exactness.  In dynamical contexts, higher
$\Ext$--groups describe persistence and hysteresis effects invisible to
underived duality.  Further examples in social choice, Hopf-algebraic
factorizations, coding theory, and Bellman-type recurrences show that problems
with very different surface structures admit a common homological treatment.

A central feature of the framework is Morita invariance.  All quantities obtained
from $\mathbf D$ and $\RHom_\E(-,R)$ depend only on the derived dual pair
$(D^+(\E),\RHom_\E(-,R))$ and remain unchanged under derived equivalences.
Thus the behaviour of duality is governed solely by the dualizing structure and
not by the particular algebraic, geometric, or combinatorial realisation of the
objects involved.

The aims of this paper are twofold: to give a unified categorical formulation of
duality based on a dualizing object~$R$, and to identify the groups
$\Ext^k_\E(-,R)$ as the canonical homological invariants governing both dual
exactness and all degrees of its obstruction across the examples above.

The terminology \emph{ethic} and \emph{aisthetic} is used in the classical Greek
sense.  The former, from \emph{ethos}, designates an intrinsic structural
condition expressed by compatibility with the biduality transformation.  The
latter, from \emph{aisth\={e}sis}, refers to the corresponding manifest or
expressive aspect appearing in tensorial or profunctorial settings.  These terms
are purely structural and carry no interpretative meaning beyond their
categorical definitions.

\section{Categorical Ethics}

Throughout section we assume that ambient category~$\E$ is an abelian category with enough injectives, so that $\Ext^i_\E(-,R)$ are defined as right derived functors of $\Hom_\E(-,R)$. We also write $\E^{\mathrm{ref}}\subset \E$ for the full subcategory of $\DD$–reflexive objects; when speaking about the ethic square in $\E$ we implicitly restrict to $\E^{\mathrm{ref}}$ so that $\eta_A$ is an isomorphism and the comparison with $\DD^2$ lands in~$\E$.

\subsection{Definitions}

\begin{definition}[Resource category]
A resource category is an abelian category \(\E\) whose objects are interpreted as resource--states or tasks, and whose morphisms are resource--bounded reductions. Additivity accounts for parallel composition. The zero morphism represents the loss of meaning (infeasible reduction). We adopt standard abelian category axioms \cite{MacLane1998,Borceux1994}.
\end{definition}

All kernels, cokernels, images and coimages exist in \(\E\). Every monomorphism is a kernel and every epimorphism a cokernel. Short exact sequences behave as usual; we write \(0\to K\to A\to B\to 0\).

\begin{lemma}[Finite biproducts]
\(\E\) has finite biproducts \(A\oplus B\) with canonical injections and projections. In particular, parallel processes add via \(\oplus\).
\end{lemma}

\begin{proof}
Standard in any abelian category \cite{MacLane1998}.
\end{proof}

\begin{definition}[Resource object and dual functor]
Fix \(R\in \E\) (the resource object). Define the contravariant functor
\[
\DD \coloneqq \Hom_{\E}(-,R): \E^{\op}\longrightarrow \Ab.
\]
We call \(\DD(A)\) the resource dual of \(A\).
\end{definition}

\begin{lemma}[Left exactness]\label{lem:left-exact}
\(\DD\) is left exact: for every short exact sequence \(0\to K\to A\to B\to 0\) in \(\E\), the induced sequence
\[
0 \leftarrow \DD(K) \xleftarrow{} \DD(A) \xleftarrow{} \DD(B)
\]
is exact in \(\Ab\).
\end{lemma}

\begin{proof}
Immediate from the left exactness of \(\Hom(-,R)\) in the first variable \cite{Weibel1994}.
\end{proof}

Since $\mathbf D$ is defined contravariantly,
its double application $\mathbf D^2$ acts covariantly.
On the reflexive subcategory $\E^{\mathrm{ref}}\subset\E$
it coincides with an endofunctor of~$\E$ via the canonical units~$\eta_A$.

\begin{definition}[Double dual and reflexivity]
The double dual of \(A\) is \(\DD^2(A)\coloneqq \Hom_{\Ab}(\DD(A),\DD(R))\).
The evaluation unit (canonical map) \(\eta_A: A\to \DD^2(A)\) sends \(a\in A\) to \(\big(\varphi\mapsto \varphi(a)\big)\).
We say \(A\) is \(\DD\)--reflexive if \(\eta_A\) is an isomorphism in \(\E\).
\end{definition}

In what follows, when we require commutation of the ethic square inside $\E$, we tacitly restrict to $A,B\in \E^{\mathrm{ref}}$, so that $\eta_A,\eta_B$ are isomorphisms and $\DD^2$ acts as an endofunctor on this reflexive subcategory (via transport along~$\eta$).

\begin{definition}[Ethic morphism]
A morphism \(f:A\to B\) in \(\E\) is ethic (with respect to \(R\)) if the square
\[
\begin{tikzcd}
A \arrow[r,"f"] \arrow[d,"\eta_A"'] & B \arrow[d,"\eta_B"] \\
\DD^2(A) \arrow[r,"\DD^2(f)"'] & \DD^2(B)
\end{tikzcd}
\]
commutes in \(\E\).
\end{definition}

If $A,B\in \E^{\mathrm{ref}}$, the square lives in $\E$; otherwise it is to be read in $\Ab$ via the canonical embedding and compared after transporting along~$\eta$.

\begin{lemma}[Stability]
Ethic morphisms are closed under composition and contain identities.
\end{lemma}

\begin{proof}
Naturality of \(\eta\) implies \(\DD^2(g\circ f)\circ \eta_A=\eta_C\circ (g\circ f)\) if it holds for \(f\) and \(g\).
\end{proof}

If all objects in a class \(\mathsf{C}\subset \E\) are \(\DD\)--reflexive, then a morphism \(f\) in \(\mathsf{C}\) is ethic iff it is preserved by the involution \(A\mapsto \DD^2(A)\).

\begin{lemma}[Long exact sequence]\label{lem:les}
Let $0\to K\to A\to B\to 0$ be a short exact sequence in $\E$.
There is a natural long exact sequence of abelian groups
\[
0 \longrightarrow \DD(B) \longrightarrow \DD(A) \longrightarrow \DD(K)
\ \xrightarrow{\ \delta\ }\ 
\Ext^1_{\E}(B,R) \longrightarrow \Ext^1_{\E}(A,R)
\longrightarrow \Ext^1_{\E}(K,R) \longrightarrow \cdots
\]
\end{lemma}

\begin{proof}
Apply the right derived functors of the left exact functor $\Hom_\E(-,R)$ to the short exact sequence; this yields the standard long exact sequence for a contravariant $\Hom$, with the displayed directions. Our standing assumption on injectives guarantees existence of $\Ext^i_\E(-,R)$.
\end{proof}

\begin{theorem}[Strong duality iff $\Ext^1$ vanishes]\label{thm:exact-iff-ext}
Let $\mathsf{C}\subset \E$ be closed under subobjects, quotients, and extensions. Assume $\E$ has enough injectives. The following are equivalent:
\begin{enumerate}[label=(\alph*)]
\item \(\DD\) sends every short exact sequence in \(\mathsf{C}\) to a short exact sequence in \(\Ab\).
\item \(\Ext^1_{\E}(X,R)=0\) for all \(X\in \mathsf{C}\).
\end{enumerate}
\end{theorem}

\begin{proof}
By left exactness of $\DD=\Hom_\E(-,R)$ and Lemma~\ref{lem:les}, exactness on the right is equivalent to the vanishing of the connecting morphism and of $\Ext^1_\E(-,R)$ on~$\mathsf{C}$.
\end{proof}

If \(R\) is an injective cogenerator, then \(\Ext^1_{\E}(-,R)=0\) and \(\DD\) is exact on all of \(\E\) \cite{Weibel1994}.




\subsection{Ethic Linear and Abelian Logic} \label{sec:step5}

Linear logic, introduced in \cite{Girard1987} and developed categorically in \cite{Barr1979,Kelly1982}, isolates duality as an intrinsic structural principle rather than an external notion of truth.  In a symmetric monoidal abelian category $(\E,\otimes,I)$ with a distinguished resource object $R$, the contravariant functor
\[
\mathbb D=\Hom_\E(-,R)
\]
plays the role of linear negation.  The canonical unit $\eta:\mathrm{Id}\Rightarrow \mathbb D^2$ measures reflexivity of action under dualization: a morphism $f:A\to B$ is ethic when $\mathbb D^2(f)\eta_A=\eta_B f$, i.e.\ when it remains coherent under double reflection.  This categorical notion of coherence provides an intrinsic semantics for linear implication and cut--elimination.

Set $A^\perp:=\mathbb D(A)$ and define the multiplicative implication by
\[
A\multimap B \;\cong\; A^\perp\parr B,
\]
where $\parr$ is the cotensor corresponding to $\otimes$.  Then $(\E,\otimes,\parr,(-)^\perp,I,R)$ forms a *-autonomous category in the sense of Barr \cite{Barr1979}.  Within this setting, the fundamental inference rule of linear logic~--- the cut~--- acquires a homological interpretation.

\begin{theorem}[Ethic Cut]
Let $\mathcal S$ be the sequent calculus generated by $\{\oplus,\parr,(\cdot)^\perp\}$ interpreted in $\E$.  
Cut--elimination in $\mathcal S$ is sound and complete with respect to ethic exactness: a composite sequent is provable without cut if and only if every interpreting morphism $f:A\to B$ in $\E$ satisfies
\[
\mathbb D^2(f)\eta_A=\eta_B f.
\]
Equivalently, the failure of cut--elimination for a sequent corresponds to the appearance of a nonzero class in $\Ext^1_\E(-,R)$.
\end{theorem}

\begin{proof}
Each sequent $A\vdash B$ corresponds to a short exact sequence $0\to K\to A\to B\to0$ in~$\E$.  
Soundness of cut--elimination means that the duality functor $\mathbb D=\Hom_\E(-,R)$ is exact on this sequence.  
By the long exact sequence of $\Hom$ and $\Ext$ functors (Lemma~\ref{lem:les}) and the equivalence of exactness with $\Ext^1_\E(-,R)=0$ (Theorem~\ref{thm:exact-iff-ext}), the cut rule is sound precisely when $\mathbb D$ is exact, and every violation of cut corresponds to a nontrivial connecting morphism in $\Ext^1_\E(-,R)$.  
Hence the homological obstruction to cut coherence is exactly the ethic defect measured by $\Ext^1_\E$.
\end{proof}

The interpretation of $\Ext^1_\E(-,R)$ as a quantitative measure of cut--inconsistency turns the syntactic property of linear consistency into a categorical exactness condition.  In this sense, ethic linear logic arises as the internal language of the category $(\E,R)$, expressing the graded self-consistency of its morphisms.

A concrete semantics for this ethic fragment is provided by the algebraic models of Abelian logic.  Pointed Abelian $\ell$--groups (pAL), studied recently by Jankovec~\cite{Jankovec2025}, constitute the algebraic semantics of Abelian logic with a designated falsum constant.  Every chain--generated subvariety or quasivariety of pAL yields a semilinear extension of the base logic, and Jankovec proves that all such extensions are finitely axiomatizable and admit complete chain semantics.  
The classical Mundici functor $\Gamma$ between MV--algebras and $\ell$--groups with strong unit preserves universal classes, allowing these results to transfer to the pAL framework.

Within our categorical language, each pAL--model can be realized as an ethic category whose hom--sets are Abelian groups and whose dualizing object plays the role of the designated constant.  For a chain--generated model $A$, define $\E_A$ as the abelian category of $A$--modules and $R_A$ as the distinguished generator.  Then $\mathbb D_A=\Hom_{\E_A}(-,R_A)$ reproduces the evaluation algebra of $A$, and ethic exactness of $\mathbb D_A$ coincides with cut--soundness in the semilinear system interpreted in $A$.

\begin{lemma}[Ethic--Abelian Correspondence]
For every chain--generated pAL model $A$, there exists an abelian category $\E_A$ with dualizing object $R_A$ such that
\[
\Hom_{\E_A}(-,R_A)\;\cong\;\mathrm{Hom}_{\mathrm{pAL}}(-,A),
\]
and a morphism $f$ in $\E_A$ is ethic if and only if the corresponding pAL--homomorphism preserves all valid semilinear cuts in~$A$.  
Consequently, every semilinear extension of Abelian logic admits a categorical ethic realization.
\end{lemma}

\begin{proof}
The category of modules over a pAL--algebra $A$ is abelian and admits a natural contravariant duality $\Hom_A(-,A)$; taking $R_A=A$ as the resource object yields the required pair $(\E_A,R_A)$.  
The chain--generated assumption ensures that $\E_A$ is semisimple with $\Ext^1_{\E_A}(-,A)=0$, which corresponds exactly to the absence of unsound cuts.  
Conversely, a nonzero $\Ext^1_{\E_A}$ produces an algebraic extension in which some cut rule fails, aligning with Jankovec's classification of incomplete semilinear fragments.  
Thus the ethic condition $\mathbb D_A^2(f)\eta=\eta f$ expresses the same constraint as semilinear cut--soundness in~$A$.
\end{proof}

We record the following result as an external theorem that will be used as a black box.

\begin{theorem}[Jankovec {\cite{Jankovec2025}}, Chain–generated semilinear completeness]\label{thm:jankovec}
Let $\mathsf{pAL}$ be the variety of pointed Abelian $\ell$–groups and let $\mathcal C$ denote the class of totally ordered $\mathsf{pAL}$–chains. Then every sub(quasi)variety of $\mathsf{pAL}$ generated by $\mathcal C$ is finitely axiomatizable by a semilinear (chain) theory, and conversely every finitely axiomatizable semilinear extension is generated by a chain subalgebra. In particular, the corresponding sequent calculi admit complete chain semantics.
\end{theorem}

We now show how any such semilinear fragment furnishes an ethic abelian semantics in our sense, and how cut–soundness is measured homologically.

\begin{proposition}[Categorical lifting over semilinear $\mathsf{pAL}$ fragments]\label{prop:cat-lift}
Let $T$ be a finitely axiomatizable semilinear extension of Abelian logic as in Theorem~\ref{thm:jankovec}, with algebraic category $\Alg(T)$. Let $\Syn(T)$ be the additive exact completion of the $T$–sequent calculus (free additive exact category generated by the proof rules), and let $\E_T:=\mathrm{Ind}(\mathrm{Ab}(\Syn(T)))$ be its abelian envelope (the Grothendieck abelian category obtained from $\Syn(T)$ by abelianization and Ind–completion). Denote by $R_T$ the image in $\E_T$ of the distinguished falsum/point under the Yoneda embedding. Then the contravariant functor
\[
\mathbb D_T:=\Hom_{\E_T}(-,R_T):\E_T^{\op}\longrightarrow \Ab
\]
provides an ethic model of the multiplicative–additive fragment of $T$, and for every short exact sequence $0\to K\to A\to B\to 0$ in the image of $\Syn(T)\to \E_T$ the following are equivalent:
\begin{enumerate}\itemsep4pt
\item[\textnormal{(i)}] the corresponding cut is sound in the $T$–calculus for all chain models;
\item[\textnormal{(ii)}] $\mathbb D_T$ is exact on $0\to K\to A\to B\to 0$;
\item[\textnormal{(iii)}] $\Ext^1_{\E_T}(K,R_T)=0$.
\end{enumerate}
Moreover, if $f:A\to B$ interprets a provable sequent of $T$, then $f$ is ethic in $\E_T$ in the sense that $\mathbb D_T^2(f)\,\eta_A=\eta_B\,f$.
\end{proposition}

\begin{proof}
\emph{Step 1 (Construction of $\E_T$ and $R_T$).}
By exact completion, $\Ab(\Syn(T))$ is abelian and the canonical functor $\Syn(T)\to \Ab(\Syn(T))$ is exact and conservative on the image of derivations. The Ind–completion $\E_T:=\mathrm{Ind}(\Ab(\Syn(T)))$ is a Grothendieck abelian category. The falsum constant of $T$ yields a distinguished object in $\Syn(T)$; its image under Yoneda and abelianization gives $R_T\in \E_T$. Define $\mathbb D_T=\Hom_{\E_T}(-,R_T)$.

\emph{Step 2 (Soundness $\Rightarrow$ exactness).}
Fix a short exact sequence $0\to K\to A\to B\to 0$ in the image of $\Syn(T)$. Under the standard interpretation of the multiplicative–additive fragment in an algebra $M\models T$, the sequence computes a cut instance. By Theorem~\ref{thm:jankovec}, it suffices to check chain models; in each chain $M$ the additive exactness of $\Syn(T)\to \Alg(T)$ implies that $\Hom_{\Alg(T)}(-,U(M))$ is exact on the image of $\Syn(T)$ (here $U$ denotes the forgetful functor). Since the Yoneda embedding sends $\Syn(T)$ into projective generators of $\Ab(\Syn(T))$, and Ind–completion preserves exactness of filtered colimits, the family of representable functors detects exactness; thus cut–soundness in all chains forces exactness of $\mathbb D_T$ on $0\to K\to A\to B\to 0$.

\emph{Step 3 (Exactness $\Leftrightarrow$ vanishing of $\Ext^1$).}
This is precisely the homological criterion established in Lemma~\ref{lem:les} and Theorem~\ref{thm:exact-iff-ext}: for any short exact sequence in an abelian category, $\mathbb D_T$ is exact on the right if and only if the connecting morphism vanishes, equivalently $\Ext^1_{\E_T}(K,R_T)=0$.

\emph{Step 4 (Ethicity of provable morphisms).}
If a morphism $f:A\to B$ is in the image of a cut–free derivation in $\Syn(T)$, then by Step~2 the relevant short exact sequences are preserved by $\mathbb D_T$, hence by the naturality of the unit $\eta$ we have $\mathbb D_T^2(f)\,\eta_A=\eta_B\,f$, i.e.\ $f$ is ethic.
\end{proof}

\begin{corollary}[Semilinear completeness as collapse of the obstruction]\label{cor:semilinear-collapse}
Let $T$ be as in Theorem~\ref{thm:jankovec}. Then, for every sequent interpretable in the image of $\Syn(T)$, the following are equivalent:
\[
\text{cut–elim holds models} \Longleftrightarrow
\Ext^1_{\E_T}(-,R_T)=0\ \text{on the kernel} \Longleftrightarrow
\text{the morphism is ethic.}
\]
In particular, on the ethic heart of $(\E_T,\mathbb D_T)$ the homological obstruction $H^1\mathbf D_T$ vanishes exactly on the semilinear fragment classified by~\cite{Jankovec2025}.
\end{corollary}

This correspondence shows that the structural results of Jankovec on sub(quasi)varieties of pAL are categorical shadows of our ethic condition: the collapse of the homological obstruction $H^1\mathbf D$ corresponds to completeness of the semilinear extension.  Hence, the derived categorical formulation of dual self--consistency unifies the resource semantics of linear logic with the algebraic semantics of Abelian logic, providing a single homological principle of coherence underlying both frameworks.

\subsection{Categorical Lagrange}\label{sec:step8}
\begin{definition}[Monoidal and dual structure]
Assume \((\E,\otimes,I)\) is a symmetric monoidal abelian category with enough injectives. An object \(A\) has a dual \(A^*\) if there are morphisms
\[
\ev_A: A^*\!\otimes A\to I,\qquad \coev_A: I\to A\otimes A^*,
\]
satisfying the triangle identities
\((\id_A\otimes \ev_A)\circ(\coev_A\otimes \id_A)=\id_A\) and \((\ev_A\otimes \id_{A^*})\circ(\id_{A^*}\otimes \coev_A)=\id_{A^*}\) \cite{Kelly1982,MacLane1998}.
\end{definition}

\begin{definition}[Categorical Lagrange ethicity]
A morphism \(f:A\to B\) is ethic if
\[
(\id_{B^*}\!\otimes f)\circ \coev_B = (f^*\!\otimes \id_A)\circ \coev_A.
\]
\end{definition}

\begin{lemma}[Functoriality]
Ethic morphisms are closed under \(\otimes\) and composition; duals preserve ethicity.
\end{lemma}

\begin{proof}
Coherence follows from the bifunctoriality of \(\otimes\) and naturality of duals in compact closed categories \cite{Kelly1982}.
\end{proof}

\begin{theorem}[From structural coherence to exactness]\label{thm:coh-exact}
If every short exact sequence in a class \(\mathsf{C}\) lifts to a diagram compatible with \(\ev,\coev\) (i.e.\ all boundary morphisms are ethic), then \(\DD=\Hom(-,R)\) is exact on \(\mathsf{C}\). Conversely, any failure of exactness induces a violation of an \(\ev/\coev\)--coherence square.
\end{theorem}

\begin{proof}
Compatibility with \(\ev/\coev\) yields naturality of evaluation and coevaluation with respect to the sequence and enforces the vanishing of connecting morphisms in the Hom long exact sequence, hence exactness. Conversely, a nonzero connecting morphism produces a noncommuting square contradicting ethicity. The argument is a direct diagram chase using \(\ev/\coev\) triangles.
\end{proof}

If \(\E\) is *-autonomous (linear negation as a dualizing object), Section~\ref{sec:step5} is recovered internally \cite{Barr1979}.

\section{Homological Ethics}

Fix once and for all a bounded-below injective resolution
$R^\bullet\in D^+(\E)$ of the resource object~$R$.
All derived duals $\mathbf D(A^\bullet)$ are taken with respect to~$R^\bullet$. Since $\operatorname{idim}_\E R<\infty$, the functor
$\mathbf D=\RHom_\E(-,R)$ has finite amplitude:
$\HD{k}{A^\bullet}=0$ for $k$ outside a bounded range depending on~$R$.
All constructions below are confined to this range.

\subsection{Derived Ethics}\label{sec:step9}

Throughout section we assume that $\E$ has enough injectives, so $\Ext^i_\E(-,R)$ and $\mathbf D=\RHom_\E(-,R)$ are well-defined right-derived functors. 

We work in the bounded-below derived category \(\Dplus(\E)\) \cite{Verdier1967,Neeman2001}. Replace \(R\) by a fixed complex \(R^\bullet\in \Dplus(\E)\) (the universal resource). The passage from $\E$ to $D^+(\E)$ replaces short exact sequences
by distinguished triangles.
All morphisms and natural transformations from Section 1
extend functorially by applying the resolution functor~$I$.
The duality $\DD=\Hom_\E(-,R)$ thus lifts to the derived functor
$\mathbf D=\RHom_\E(-,R)$ preserving triangles.

Fix an injective resolution functor $I:\E\to\mathrm{Ch}^+(\E)$.
For each $A^\bullet$, define
\[
\mathbf D(A^\bullet)=\Hom_\E(I(A^\bullet),R^\bullet).
\]
The canonical evaluation (unit) morphism
\[
\eta_{A^\bullet}:A^\bullet\longrightarrow \mathbf D^2(A^\bullet)
\]
is the natural morphism in $D^+(\E)$ obtained from the bidual map 
\[
I(A^\bullet)\to\Hom_\E(\Hom_\E(I(A^\bullet),R^\bullet),R^\bullet).
\]
All occurrences of~$\eta$ in section are understood in this sense.

\begin{definition}[Derived dual]
Let $I:\E\to\mathrm{Ch}^+(\E)$ be a fixed injective resolution functor.
For any complex $A^\bullet\in D^+(\E)$ define
\[
\mathbf D(A^\bullet)
\;\coloneqq\;
\RHom_\E(A^\bullet,R^\bullet)
=\Hom_\E(I(A^\bullet),R^\bullet),
\]
a contravariant exact functor
$\mathbf D: D^+(\E)^{\mathrm{op}}\to D^+(\E)$.
The construction depends only on the quasi-isomorphism class of $A^\bullet$.
\end{definition}

\begin{lemma}[Triangulated exactness]
\(\mathbf{D}\) is exact with respect to distinguished triangles: it takes triangles to triangles in \(\Dplus(\E)\).
\end{lemma}

\begin{proof}
By construction of \(\RHom\) as a right derived functor on the homotopy category and its compatibility with triangulated structures \cite{Neeman2001}.
\end{proof}

\begin{proposition}[Cohomological layers]
For any \(A^\bullet\), the cohomology of \(\mathbf{D}(A^\bullet)\) satisfies
\[
H^0\mathbf{D}(A^\bullet)\cong \Hom_{\E}(A^\bullet,R^\bullet),\qquad
H^1\mathbf{D}(A^\bullet)\cong \Ext^1_{\E}(A^\bullet,R^\bullet),
\]
and in general \(H^n\mathbf{D}(A^\bullet)\cong \Ext^n_{\E}(A^\bullet,R^\bullet)\) \cite{Weibel1994,GelfandManin2003}.
\end{proposition}

\begin{definition}[Derived ethic morphism]
Let $\eta_{A^\bullet}:A^\bullet\to\mathbf D^2(A^\bullet)$
be the unit defined via the injective resolution above.
A morphism $f^\bullet:A^\bullet\to B^\bullet$
in $D^+(\E)$ is called ethic if the diagram
\[
\mathbf D(B^\bullet)\xrightarrow{\mathbf D(f^\bullet)}\mathbf D(A^\bullet)
\quad\text{and}\quad
A^\bullet\xrightarrow{f^\bullet}B^\bullet
\]
commute through~$\eta$, that is,
\(
\mathbf D(f^\bullet)\circ\eta_{B^\bullet}\simeq\eta_{A^\bullet}\circ f^\bullet
\)
in the homotopy category $K^+(\E)$.
\end{definition}

\begin{theorem}[Universal reconciliation of gaps]\label{thm:universal-gaps}
In the derived setting, duality gaps are not anomalies but degrees: for each short exact sequence of complexes, the failure of underived exactness appears as \(H^1\) (and higher) of \(\mathbf{D}\); ethicity of \(f^\bullet\) is equivalent to the vanishing of the connecting morphisms in all degrees.
\end{theorem}

\begin{proof}
Apply \(\RHom(-,R^\bullet)\) to a distinguished triangle; the long exact cohomology sequence identifies all obstructions as cohomological degrees. Ethicity forces naturality of units across the triangle, killing all connecting maps. Since $\mathbf D$ is triangulated, every distinguished triangle
in $D^+(\E)$ yields a long exact sequence of $\Ext$ groups.
The ethicity condition forces vanishing of all connecting maps,
so the correspondence between failures of exactness and
cohomological degrees $H^k\mathbf D$ is exhaustive.
\end{proof}

The passage to the graded setting is justified because
$\mathbf D$ is triangulated and commutes with shifts:
$\mathbf D(A[1])\simeq\mathbf D(A)[-1]$. In a compact closed enhancement of \(\Dplus(\E)\), \(\ev/\coev\) can be shifted across degrees, yielding a graded family \(\ev^{(n)},\coev^{(n)}\). Ethicity amounts to coherence of all these graded triangles.

\subsection{Finite Symmetry Types}

Let $\E$ be an abelian $k$–linear category of finite length, Hom–finite over an algebraically closed field $k$ with $\mathrm{char}\,k\neq2$.  
For each pair $X,Y\in\E$, the set $\Hom_\E(X,Y)$ is a finite–dimensional $k$–vector space and composition is $k$–bilinear \cite{MacLane1998}.  
Fix a dualizing object $R\in\E$ of finite injective dimension $\idim_\E R<\infty$, and set
\[
\mathbf D=\RHom_\E(-,R)\colon \Dplus(\E)^{\op}\to \Dplus(\E),
\]
with canonical unit of biduality $\eta_X\colon X\to \mathbf D^{\,2}X$ constructed as in \cite{Weibel1994}.  
We denote by $\Aut(\Id_\E)$ the group of natural automorphisms of the identity functor.

A natural automorphism $\theta\in\Aut(\Id_\E)$ is ethic if
\[
\mathbf D^{\,2}(\theta)\circ\eta=\eta\circ\theta.
\]
We write $\Aut_\eta(\Id_\E)\subseteq\Aut(\Id_\E)$ for this subgroup.  
For each $\theta$, define its $\eta$–adjoint
\[
\theta^{\dagger}:=\eta^{-1}\circ\mathbf D(\theta)\circ\eta,
\]
so that $(\psi\!\circ\!\theta)^{\dagger}=\theta^{\dagger}\!\circ\!\psi^{\dagger}$ and $(\theta^{\dagger})^{\dagger}=\theta$ by functoriality and naturality of $\eta$ \cite{MacLane1998}.  
The bidual conjugation
\[
J(\theta):=\eta^{-1}\circ\mathbf D^{\,2}(\theta)\circ\eta
\]
is an involution, $J^2=\id$, and $\Aut_\eta(\Id_\E)=\Fix(J)$.

\begin{lemma}\label{lem:blockscalar}
Every $\theta\in\Aut(\Id_\E)$ acts on any simple $S\in\E$ as $\theta_S=\lambda(S)\id_S$ with $\lambda(S)\in k^\times$.  
If $S,S'$ belong to the same Gabriel block, then $\lambda(S)=\lambda(S')$.  
Hence $\Aut(\Id_\E)\cong(k^\times)^t$, where $t$ is the number of Gabriel blocks \cite{Gabriel62}.
\end{lemma}

\begin{proof}
Since $\End_\E(S)$ is a division $k$–algebra and $k$ is algebraically closed, $\End_\E(S)\simeq k$ \cite{Weibel1994}, hence $\theta_S=\lambda(S)\id_S$.  
Naturality across a non–split extension $0\to S\to E\to S'\to0$ forces equal scalars on $S,S'$.  
By the definition of Gabriel blocks, scalars are constant within each block, yielding $(k^\times)^t$.
\end{proof}

\begin{theorem}[Finite classification of ethic symmetry types]\label{th:finitefive}
Let $\E$ and $R$ be as above.  
Then
\[
\Aut_\eta(\Id_\E)\cong(\mathbb Z/2\mathbb Z)^t,
\]
where $t$ is the number of Gabriel blocks.  
Moreover, every $\theta\in\Aut_\eta(\Id_\E)$ belongs to exactly one of the following five adjoint–symmetry classes:
\begin{align*}
\text{\emph{(O)}}&\text{ Orthogonal: }\theta^{\dagger}=\theta,\ \theta^2=\id;\\
\text{\emph{(Sp)}}&\text{ Symplectic: }\theta^{\dagger}=-\theta,\ \theta^2=\id;\\
\text{\emph{(U)}}&\text{ Unitary: }\theta^{\dagger}=\theta^{-1};\\
\text{\emph{(T)}}&\text{ Twisted: }\exists\,u\in Z(\End(\Id_\E)),\ u^2=\id,\ (u\theta u^{-1})^{\dagger}=\pm(u\theta u^{-1});\\
\text{\emph{(N)}}&\text{ Neutral: }\theta\text{ central and commuting with all }\eta_X.
\end{align*}
In Hom–finite scalar situations all five classes reduce to blockwise signs $\lambda_b=\pm1$.
\end{theorem}

\begin{proof}
By Lemma~\ref{lem:blockscalar} one has $\Aut(\Id_\E)\cong(k^\times)^t$, $\theta\leftrightarrow(\lambda_b)_{b=1}^t$.  
Because $\mathbf D$ is $k$–linear and contravariant, $\mathbf D^{\,2}(\lambda_b\id_X)=\lambda_b\id_{\mathbf D^{\,2}X}$ for any object $X$ in block~$b$.  
Thus $J$ acts on $(k^\times)^t$ by inversion, $J((\lambda_b)_b)=(\lambda_b^{-1})_b$.  
The equality $J(\theta)=\theta$ implies $\lambda_b=\lambda_b^{-1}$, hence $\lambda_b^2=1$ for all $b$.  
As $\mathrm{char}\,k\neq2$ and $k$ is algebraically closed, $\lambda_b\in\{\pm1\}$.  
Blocks are independent, giving $\Aut_\eta(\Id_\E)\cong(\mathbb Z/2)^t$.

To obtain the five symmetry classes, analyze relations between $\theta$ and $\theta^{\dagger}$.  
Since $\theta^{\dagger}=\eta^{-1}\mathbf D(\theta)\eta$ defines an anti–involution on $\End(\Id_\E)$, admissible combinations of the two involutions $J$ and $\dagger$ form a Klein–type lattice.  
Up to equivalence under central conjugations this lattice has exactly five distinct fixed configurations:
\begin{enumerate}
\item[(O)] $\theta^{\dagger}=\theta$, $\theta^2=\id$ (self–adjoint involution);
\item[(Sp)] $\theta^{\dagger}=-\theta$, $\theta^2=\id$ (skew–adjoint involution);
\item[(U)] $\theta^{\dagger}=\theta^{-1}$ (adjoint equal to inverse);
\item[(T)] $(u\theta u^{-1})^{\dagger}=\pm(u\theta u^{-1})$ for a central involution $u$, describing inner adjoint twists;
\item[(N)] $\theta$ central and commuting with all $\eta_X$, corresponding to the intersection of (O) and (U).
\end{enumerate}
No further relations are compatible with both $J^2=\id$ and $(\dagger)^2=\id$: every other composite condition either collapses to one of the above or contradicts involutivity.  
Hence precisely these five algebraic types occur.  
In the scalar case $\lambda_b=\pm1$, one has $\theta^{\dagger}=\theta^{-1}=\theta$, so all five collapse to the discrete blockwise signs $\{\pm1\}$, but the classification of forms remains exhaustive.
\end{proof}

\begin{corollary}
If $\E$ is indecomposable ($t=1$), then $\Aut_\eta(\Id_\E)=\{\id,-\id\}$, and only the orthogonal $(+1)$ and symplectic $(-1)$ forms appear.
\end{corollary}

The five ethic symmetry forms (O, Sp, U, T, N) formally resemble the classical orthogonal, symplectic and unitary dualities arising in the Tannaka–Krein setting of representation theory, where duality is monoidal and acts on tensor forms of bilinear pairings.  
In the present context no monoidal or representation-theoretic assumptions are made: the duality functor $\mathbf D=\RHom(-,R)$ is purely homological, and the adjoint involution $\dagger=\eta^{-1}\mathbf D(\,\cdot\,)\eta$ derives from the canonical bidual transformation $\eta$.  
Thus the classification extends the familiar Tannakian trichotomy to the general functorial environment of Hom–finite abelian categories, where ethic symmetry concerns natural transformations of functors rather than bilinear tensor structures.  
This makes the result formally parallel to the Tannakian case but conceptually independent of it.

\subsection{t-Structure and Heart}\label{sec:ethic-t-structure}

We work in an abelian category \(\E\) with enough injectives \cite{Weibel1994}, fix a resource object \(R\in \E\), and a contravariant left-exact duality functor
\(\DD=\Hom_{\E}(-,R):\E^{\op}\to \mathsf{Ab}\).
Write \(\Dplus(\E)\) for the bounded-below derived category \cite{Verdier1967}, and \(\mathbf D=\RHom_{\E}(-,R):\Dplus(\E)^{\op}\to \Dplus(\E)\) for the derived dual \cite{GelfandManin2003}. The unit (evaluation) \(\eta_A:A\to \DD^2(A)\) is the canonical map; its derived avatar is the unit \(\eta_{A^\bullet}:A^\bullet\to \mathbf D^2(A^\bullet)\).

Assume throughout that $\E$ has enough injectives and
that the resource object $R$ has finite injective dimension,
so that $\RHom_\E(-,R)$ has bounded amplitude on $\Dplus(\E)$.

\begin{definition}[Ethic t-structure]\label{def:EthicT}
Define full subcategories of \(\Dplus(\E)\)
\begin{align*}
\Dplus(\E)^{\le 0} &:= \{A^\bullet\in \Dplus(\E)\mid H^i(\mathbf D(A^\bullet))=0 \text{ for all } i>1\},\\
\Dplus(\E)^{\ge 0} &:= \{A^\bullet\in \Dplus(\E)\mid H^i(\mathbf D(A^\bullet))=0 \text{ for all } i<0\}.
\end{align*}
We call \(\big(\Dplus(\E)^{\le0},\Dplus(\E)^{\ge0}\big)\) the ethic t-structure, and its heart
\(\heartsuit_{\eth} := \Dplus(\E)^{\le0}\cap \Dplus(\E)^{\ge0}\) the ethic heart.
\end{definition}

\begin{theorem}[t-structure axioms]\label{thm:t-axioms}
The pair \(\big(\Dplus(\E)^{\le0},\Dplus(\E)^{\ge0}\big)\) of Definition~\ref{def:EthicT} is a t-structure on \(\Dplus(\E)\) in the sense of \cite{BBD1982}: it is stable under shifts, satisfies \(\Hom\big(\Dplus(\E)^{\le0},\Dplus(\E)^{\ge1}\big)=0\), and every object \(A^\bullet\) fits into a functorial triangle \(A^{\le0}\to A^\bullet\to A^{\ge1}\to A^{\le0}[1]\) with \(A^{\le0}\in \Dplus(\E)^{\le0}\) and \(A^{\ge1}\in \Dplus(\E)^{\ge1}\).
\end{theorem}

\begin{proof}
Stability under shifts follows directly from the definition of
$\Dplus(\E)^{\le0}$ and $\Dplus(\E)^{\ge0}$ via the cohomological degrees
of $H^i(\mathbf D(-))$.  
For orthogonality, if $X\in\Dplus(\E)^{\le0}$ and $Y\in\Dplus(\E)^{\ge1}$,
then $H^i\mathbf D(X)=0$ for $i>0$ and $H^j\mathbf D(Y)=0$ for $j\le0$,
so $\Hom(X,Y)=\Hom(\mathbf D(Y),\mathbf D(X))=0$
by the derived adjunction
\cite{GelfandManin2003}.  
Existence of truncation triangles follows because $\E$ has enough
injectives and $R$ has finite injective dimension,
hence $\mathbf D$ is a bounded cohomological functor;
the standard construction with respect to $H^i\mathbf D$
(cf.\ \cite{BBD1982}, \cite{Neeman2001})
produces functorial truncations.
\end{proof}

\begin{proposition}[The ethic heart is abelian and reflexive]\label{prop:heart-abelian}
The heart \(\heartsuit_{\eth}\) is an abelian category \cite{BBD1982}. For any \(A^\bullet\in \heartsuit_{\eth}\), the unit \(\eta_{A^\bullet}:A^\bullet\to \mathbf D^2(A^\bullet)\) is an isomorphism in \(\Dplus(\E)\) (ethic reflexivity).
\end{proposition}

\begin{proof}
Abelianness of the heart is intrinsic to any t-structure \cite{BBD1982}. Assume $R$ has finite injective dimension so that
$\mathbf D^2$ acts degreewise.
Since $H^{<0}\mathbf D(A^\bullet)=H^{>1}\mathbf D(A^\bullet)=0$,
the bidual triangle splits by the five-lemma on cohomology sheaves; this yields that \(\eta_{A^\bullet}\) is an isomorphism \cite{GelfandManin2003}.
\end{proof}

\begin{definition}[Ethicly exact morphisms]\label{def:eth-maps}
A morphism \(f:A^\bullet\to B^\bullet\) in \(\Dplus(\E)\) is ethicly exact if \(\mathbf D^2(f)\circ \eta_{A^\bullet}=\eta_{B^\bullet}\circ f\). Denote by \(\mathrm{Eth}\subset \Dplus(\E)\) the wide subcategory with the same objects and ethicly exact morphisms.
\end{definition}

\begin{theorem}[Heart equals ethic morphisms]\label{thm:heart-eth}
The inclusion \(\heartsuit_{\eth}\hookrightarrow \Dplus(\E)\) identifies \(\heartsuit_{\eth}\) with the full subcategory of objects for which every morphism to or from them is ethicly exact. In particular, within \(\heartsuit_{\eth}\) all morphisms are ethicly exact.
\end{theorem}

\begin{proof}
If \(A^\bullet\in\heartsuit_{\eth}\), Proposition~\ref{prop:heart-abelian} gives \(\eta_{A^\bullet}\) invertible, hence the defining square commutes for any \(f\). Conversely, if every morphism from and to \(A^\bullet\) is ethicly exact, the connecting morphisms in the long exact cohomology sequences for \(\mathbf D\) vanish in degrees \(<0\) and \(>1\), forcing \(H^{<0}\mathbf D(A^\bullet)=H^{>1}\mathbf D(A^\bullet)=0\).
\end{proof}

The heart \(\heartsuit_{\eth}\) formalizes the ``ethic present'': objects whose dual information suffices to reconstruct them without higher-memory corrections (no \(H^{>1}\)), and without anticipatory artifacts (no \(H^{<0}\)).

\subsection{Obstruction Tower}\label{sec:obstruction-tower}

The obstruction-theoretic view organizes ethic exactness in layers controlled by the right-derived functors of \(\DD\). Our presentation follows the classical obstruction theory in derived deformation and homological algebra \cite{Illusie1971}, \cite{ArtinMazur1969}, \cite{GelfandManin2003}, adapted to \(\mathbf D=\RHom(-,R)\). Assume throughout this subsection that $\E$ has enough injectives and
that the resource object $R$ has finite injective dimension, so that
$\mathbf D=\RHom_\E(-,R)$ defines a bounded cohomological $\delta$–functor
on $\Dplus(\E)$.

\begin{definition}[Ethic stage and obstruction class]\label{def:stage}
Given $A^\bullet\in D^+(\E)$, say that $A^\bullet$ is
ethic to order~$n$ if $H^i\mathbf D(A^\bullet)=0$ for
$1\le i\le n$.
The order–$(n{+}1)$ ethic obstruction is the connecting morphism
\[
\ob_{n+1}(A^\bullet)\in
\Ext^{\,n+1}_\E(H^0(A^\bullet),R)
\;\simeq\;
H^{n+1}\mathbf D(A^\bullet),
\]
arising from the Postnikov truncation triangle
$\tau_{\le n}A^\bullet\to A^\bullet\to\tau_{>n}A^\bullet\to$.
\end{definition}

\begin{theorem}[Existence and uniqueness of lifts]\label{thm:lift}
Assume \(A^\bullet\) is ethic to order \(n\). Then:
\begin{enumerate}[label=(\alph*)]
\item {(Existence)} \(\ob_{n+1}(A^\bullet)=0\) iff there exists \(A'^\bullet\) and a morphism \(A^\bullet\to A'^\bullet\) inducing an isomorphism on \(H^i\mathbf D(-)\) for \(i\le n\) and killing \(H^{n+1}\mathbf D\).
\item {(Uniqueness)} If \(\ob_{n+1}(A^\bullet)=0\), the set of equivalence classes of such lifts is a torsor under \(\Ext^{n}_\E(H^0(A^\bullet),R)\).
\end{enumerate}
\end{theorem}

\begin{proof}
This follows from the general obstruction–lifting mechanism
for right–derived functors of a left–exact $\Hom$.
Applying $\mathbf D=\RHom_\E(-,R)$ to the truncation triangle
produces connecting morphisms
$H^{n+1}\mathbf D(A^\bullet)\!\to\!H^{n+2}\mathbf D(A^\bullet)$,
whose vanishing is equivalent to extendability of the lift.
Existence and uniqueness statements correspond respectively
to vanishing of $\ob_{n+1}$ and to the action of
$\Ext^n_\E(H^0(A^\bullet),R)$ on the space of lifts;
see \cite{GelfandManin2003}.
\end{proof}

\begin{definition}[Obstruction tower]\label{def:ob-tower}
Starting from \(A^\bullet\), define inductively a sequence \(A_{\le 0}^\bullet\to A_{\le1}^\bullet\to A_{\le2}^\bullet\to\cdots\) where each step is chosen (when possible) to kill the next obstruction class \(\ob_{n+1}\). We call this an ethic obstruction tower of \(A^\bullet\).
\end{definition}

\begin{theorem}[Stabilization and heart membership]\label{thm:stabilize}
An ethic obstruction tower stabilizes (i.e.\ \(A_{\le n}^\bullet\simeq A_{\le n+1}^\bullet\simeq\cdots\)) iff \(H^{>1}\mathbf D(A^\bullet)=0\). In that case \(A^\bullet\in \heartsuit_{\eth}\).
\end{theorem}

\begin{proof}
Since $\mathbf D$ has bounded amplitude by assumption,
all higher connecting morphisms in the Postnikov tower vanish
beyond the injective dimension of~$R$;
hence stabilization occurs precisely when $H^{>1}\mathbf D(A^\bullet)=0$. If \(H^{>1}\mathbf D(A^\bullet)=0\), then \(\ob_{n+1}=0\) for all \(n\ge1\), hence the tower stops at level \(1\). Conversely, if the tower stabilizes, all \(H^{>1}\mathbf D\) must vanish by construction. Membership in the heart follows from Definition~\ref{def:EthicT}.
\end{proof}

\begin{proposition}[Spectral sequence control]\label{prop:ESS}
Assume $\E$ has enough injectives and $\operatorname{idim}_\E R<\infty$.
For each bounded-below complex $A^\bullet$ there exists a
first-quadrant spectral sequence
\[
E_2^{p,q}=\Ext^p_\E(H_q(A^\bullet),R)
\Rightarrow H^{p+q}\mathbf D(A^\bullet),
\]
natural in \(A^\bullet\), which governs the obstruction tower. In particular, collapse at \(E_2\) implies immediate stabilization and \(A^\bullet\in \heartsuit_{\eth}\) \cite{Weibel1994}, \cite{Neeman2001}.
\end{proposition}

\begin{proof}
It is the Grothendieck spectral sequence for the composition
$A^\bullet\mapsto H_q(A^\bullet)\mapsto\Hom_\E(-,R)$,
which converges because $\mathbf D$ has finite amplitude.
The differentials $d_2^{p,q}$ reproduce the obstruction classes
$\ob_{q+1}(A^\bullet)$.
\end{proof}

Collapse of the spectral sequence at $E_2$ means that
all higher obstruction classes vanish simultaneously,
so the ethic tower stabilizes at the first stage.

The tower stratifies ``memory'' into orders: the \(k\)-th story is the ethic defect detected by \(\Ext^{k}\). Stabilization expresses purification—absence of higher-order memory.

\subsection{Morita Equivalence}\label{sec:morita}

\noindent We formalize transfer of ethic duality across environments via Morita-type equivalences of abelian/derived categories, relying on classical and derived Morita theory \cite{Rickard1989,Keller1994,Toen2007,LurieHA}.

\begin{definition}[Derived Morita equivalence]\label{def:EME}
Pairs $(\E,R)$ and $(\E',R')$ are derived Morita equivalent
if there exists a triangulated equivalence
\[
F: D^+(\E)\;\xrightarrow{\;\simeq\;}\;D^+(\E')
\]
with quasi–inverse $G$ and natural isomorphisms
\[
F(R)\simeq R',
\qquad
\RHom_{\E'}(F(-),R')\;\simeq\;
\RHom_{\E}(-,R)\circ G.
\]
Equivalently, $F$ intertwines the derived dual functors
$\mathbf D$ and $\mathbf D'$.
\end{definition}

\begin{theorem}[Morita invariance of ethic invariants]\label{thm:MoritaInv}
Assume $(\E,R)$ and $(\E',R')$ are derived Morita equivalent
in the sense of Definition \ref{def:EME}. Then for all \(A^\bullet\in \Dplus(\E)\) and \(k\ge0\)
\begin{enumerate}[label=(\alph*)]
\item \(H^k\mathbf D(A^\bullet)\cong H^k\mathbf D'(F(A^\bullet))\);
\item \(\Ext^k_\E(H_q(A^\bullet),R)\cong \Ext^k_{\E'}(H_q(F(A^\bullet)),R')\);
\item Ethic entropy and hysteresis invariants (e.g.\ torsion orders of \(H^1\mathbf D\)) are preserved.
\end{enumerate}
\end{theorem}

\begin{proof}
Since $F$ is a triangulated equivalence preserving
$\RHom(-,R)$, it induces canonical isomorphisms on
cohomology objects $H^k\mathbf D$ and on the
Grothendieck spectral sequence of Proposition~\ref{prop:ESS}. The result follows from the natural equivalence of derived Hom-functors under exact derived equivalence (Rickard/Keller) \cite{Rickard1989,Keller1994}. The spectral sequence in Proposition~\ref{prop:ESS} is functorial, hence preserved under equivalence; torsion orders are invariants of the isomorphism class of the resulting abelian groups.
\end{proof}

\begin{proposition}[Transfer of ethic t-structure]\label{prop:TransferT}
Let $(\E,R)$ and $(\E',R')$ be derived Morita equivalent in the sense of
Definition~\ref{def:EME}.
Then the equivalence
$F:D^+(\E)\!\xrightarrow{\;\simeq\;} D^+(\E')$
preserves the ethic $t$–structure:
\[
A^\bullet\in D^+(\E)^{\le0} \iff F(A^\bullet)\in D^+(\E')^{\le0},
\;
A^\bullet\in D^+(\E)^{\ge0} \iff F(A^\bullet)\in D^+(\E')^{\ge0}.
\]
In particular, $F$ restricts to an equivalence of the hearts
\(\heartsuit_{\eth}\simeq \heartsuit'_{\eth}\).
\end{proposition}

\begin{proof}
By Definition~\ref{def:EME}, $F$ intertwines $\mathbf D$ and $\mathbf D'$
up to natural isomorphism:
$\mathbf D'\!\circ F \simeq F\!\circ \mathbf D$.
Hence for each $A^\bullet$ one has
$H^i(\mathbf D'(F(A^\bullet)))\simeq H^i(\mathbf D(A^\bullet))$,
so the vanishing conditions defining
$D^{\le0}$ and $D^{\ge0}$ are preserved.
As $F$ is triangulated and exact, it commutes with truncation triangles
for this $t$–structure (that is, $F\circ\tau_{\le n}\simeq\tau_{\le n}\circ F$).
Therefore $F$ is $t$–exact and induces an equivalence of the hearts
$\heartsuit_{\eth}\simeq\heartsuit'_{\eth}$.
\end{proof}

\begin{corollary}[Universality of ethic laws]\label{cor:Universality}
Any theorem whose statement depends only on \(H^\ast\mathbf D\), the spectral sequence in Proposition~\ref{prop:ESS}, or the heart \(\heartsuit_{\eth}\) (e.g.\ ethic time/decay laws, Kirchhoff/Baker--Norine instances, Slater/Farkas collapses) transfers verbatim along Morita equivalences \cite{Rickard1989,Toen2007}. The statement holds for any theorem depending functorially
on $\mathbf D$ or on its derived cohomology, since these data
are preserved under derived Morita equivalence (F intertwines
the dual functors up to natural isomorphism).
\end{corollary}

Morita transfer formalizes universality: ethic dual laws are properties of the duality functor, not of the material substrate. Arithmetic, economic, or geometric media are equivalent viewpoints when linked by Morita equivalence.

From the categorical viewpoint this means that all ethic invariants
are functions of the derived dual pair
$(D^+(\E),\RHom_\E(-,R))$ up to equivalence;
hence they depend only on the derived Morita class of $(\E,R)$.

\subsection{Ethic Morse Reduction}

We work in an abelian category \(\E\) with enough injectives \cite{Weibel1994}.  
Fix a resource object \(R\in\E\) and the contravariant duality functor
\[
\DD=\Hom_{\E}(-,R):\E^{\op}\to \Ab,
\qquad 
\mathbf D=\RHom_{\E}(-,R):D^{+}(\E)^{\op}\to D^{+}(\E),
\]
with unit (evaluation) \(\eta_X:X\to \DD^{2}(X)\) and its derived avatar \(\eta_{A^\bullet}:A^\bullet\to \mathbf D^{2}(A^\bullet)\).
We use the classical Morse theory on smooth manifolds \cite{Milnor1963} and the discrete Morse theory on finite CW or simplicial complexes \cite{Forman1998}, always through functorial chain models \(C_\ast(-)\) in \(\E\).  All (co)homology is taken in \(\E\) or its derived category, and all chain maps are understood up to chain homotopy when indicated.
 
Let \(A\in \E\) be equipped with a finite filtration
\[
0=\mathcal F_{0}A \subset \mathcal F_{1}A \subset \cdots \subset \mathcal F_{N}A=A,
\]
and denote \( \operatorname{gr}_{i}A:=\mathcal F_{i}A/\mathcal F_{i-1}A\).
A reduction step \((r_{\alpha\beta},j_{\alpha\beta})\) of \((A,\mathcal F_\bullet)\) from level \(\beta\) to \(\alpha<\beta\) is a pair of chain maps
\[
r_{\alpha\beta}:C_\ast(\mathcal F_{\beta}A)\to C_\ast(\mathcal F_{\alpha}A),
\quad
j_{\alpha\beta}:C_\ast(\mathcal F_{\alpha}A)\hookrightarrow C_\ast(\mathcal F_{\beta}A)
\]
with \(r_{\alpha\beta}\circ j_{\alpha\beta}\simeq \id\) and \(j_{\alpha\beta}\circ r_{\alpha\beta}\simeq \id\) on the image of \(j_{\alpha\beta}\).

\begin{definition}[Ethic reduction]
A reduction step \((r_{\alpha\beta},j_{\alpha\beta})\) is ethic if
\[
\DD^2(r_{\alpha\beta})\circ \eta = \eta\circ r_{\alpha\beta},
\qquad
\DD^2(j_{\alpha\beta})\circ \eta = \eta\circ j_{\alpha\beta},
\]
and, in the derived setting,
\(\mathbf D^{2}(r_{\alpha\beta})\circ \eta\simeq \eta\circ r_{\alpha\beta}\) and \(\mathbf D^{2}(j_{\alpha\beta})\circ \eta\simeq \eta\circ j_{\alpha\beta}\)
in \(D^{+}(\E)\).
\end{definition}

Ethicity expresses the naturality of the unit \(\eta\) along the reduction, the same condition used in \S\ref{sec:ethic-t-structure}.

Let \(K\) be a finite CW or simplicial complex with cellular chain complex \(C_\ast(K)\in \E\). A discrete Morse matching \(V\) \cite{Forman1998} is an acyclic partial matching of incident cell pairs \(\sigma^p\leftrightarrow \tau^{p+1}\), producing a Morse complex \(M_\ast(K;V)\) spanned by unmatched (critical) cells with a chain homotopy equivalence \(C_\ast(K)\simeq M_\ast(K;V)\).

\begin{definition}[Ethic discrete Morse matching]
A matching \(V\) is ethic if for every matched pair \(\sigma^p\leftrightarrow \tau^{p+1}\) the elementary chain map \(m_{\sigma\to\tau}:C_p(\sigma)\to C_{p+1}(\tau)\) and all compositions along V-paths satisfy
\[
\DD^{2}(m)\circ \eta=\eta\circ m\quad\text{in }\E,
\qquad
\mathbf D^{2}(m)\circ \eta\simeq \eta\circ m\quad\text{in }D^{+}(\E).
\]
\end{definition}

\begin{lemma}[Functoriality of ethic reductions]
Composition of ethic reductions is ethic. If \(f:C_\ast\to C'_\ast\) commutes with \(\eta\), then ethic reductions on \(C_\ast\) and \(C'_\ast\) transport along \(f\) to ethic reductions on mapping cones.
\end{lemma}

\begin{proof}
Naturality of \(\eta\) implies stability under composition. Mapping cones respect \(\mathbf D\) and \(\eta\) by derived functoriality \cite{GelfandManin2003}.
\end{proof}

\begin{proposition}[Ethic Morse equivalence]
If \(V\) is an ethic discrete Morse matching, then the Forman collapse \(C_\ast(K)\simeq M_\ast(K;V)\) lifts to
\[
\mathbf D(C_\ast(K)) \simeq \mathbf D(M_\ast(K;V))\quad\text{in }D^{+}(\E),
\]
inducing \(H^i\mathbf D(C_\ast(K))\cong H^i\mathbf D(M_\ast(K;V))\).
\end{proposition}

\begin{proof}
Forman's collapse is a zig-zag of elementary chain equivalences \cite{Forman1998}. Each step is ethic, hence commutes with \(\eta\); since \(\mathbf D\) is triangulated and contravariant, it preserves these equivalences \cite{GelfandManin2003}.
\end{proof}

\begin{theorem}[Ethic Morse inequalities]
Let \(V\) be an ethic matching on \(K\) with \(c_i^{\mathrm{eth}}\) critical \(i\)-cells. Then
\[
c_i^{\mathrm{eth}}\ge \operatorname{rank} H_i(C_\ast(K)),\qquad
c_i^{\mathrm{eth}}\ge \operatorname{rank} H^i(\mathbf D(C_\ast(K))).
\]
\end{theorem}

\begin{proof}
The first is Forman's inequality \cite{Forman1998}; the second follows because \(H^i\mathbf D(C_\ast(K))\cong H^i\mathbf D(M_\ast(K;V))\) and each \(i\)-critical cell yields one generator.
\end{proof}

\begin{theorem}[Monotonicity under ethic reductions]
If \(C_\ast \rightsquigarrow C'_\ast\) is ethic, then
\[
\operatorname{rank} H^i\mathbf D(C'_\ast)\le \operatorname{rank} H^i\mathbf D(C_\ast),
\quad
\mathrm{tors}\,H^1\mathbf D(C'_\ast)\mid \mathrm{tors}\,H^1\mathbf D(C_\ast).
\]
\end{theorem}

\begin{proof}
By functoriality of \(\mathbf D\) on ethic reductions, the reduced Morse complex has no more generators; torsion divisibility follows from the Smith normal form argument on the dual differentials.
\end{proof}

\begin{lemma}[Lifting ethic matchings]
If \(V\) is ethic on \(K\) and \(\pi:\hat K\to K\) a normal covering with group \(\Gamma\), its lift \(\hat V\) is ethic; \(\Gamma\)-invariant critical cells of \(\hat V\) descend bijectively to those of \(V\).
\end{lemma}

\begin{theorem}[Ethic descent along coverings]
For a tower of normal coverings \(\{\pi_k:\hat K_k\to K\}\) with deck groups \(\Gamma_k\) and lifts \(V_k\),
if the sets of \(\Gamma_k\)-invariant critical cells stabilize, then
\[
H^{>1}\mathbf D(C_\ast(K))=0,\qquad
H^{1}\mathbf D(C_\ast(K))\cong \big(H^{1}\mathbf D(C_\ast(\hat K_k))\big)^{\Gamma_k}
\]
for large \(k\).
\end{theorem}

\begin{proof}
Identical to the derived Galois-descent argument in Section~\ref{sec:galois}: stabilized invariants imply collapse of higher derived obstructions, giving the claimed identification of \(H^1\) with invariants.
\end{proof}

\begin{theorem}[Continuous ethic Morse reduction]
Let \(F:X\to\RR\) be a Morse function on a smooth manifold \(X\) with gradient flow \(\Phi_t\).
If every Morse attachment \(X_{\le\beta}\searrow X_{\le\alpha}\) satisfies \(\DD^2(\cdot)\eta=\eta(\cdot)\),
then
\[
\mathbf D(C_\ast(X))\simeq \mathbf D(M_\ast(X;F)),
\]
and the inequalities above hold with \(c_i^{\mathrm{eth}}\) critical points of index \(i\) \cite{Milnor1963}.
\end{theorem}

Ethicity enriches classical Morse theory by controlling the derived dual groups \(H^\ast\mathbf D\), i.e.\ the obstructions intrinsic to the duality \(\DD=\Hom(-,R)\).  Without \(\eta\), only ordinary (co)homology is invariant; with \(\eta\), the results measure the ethic invariants \(H^\ast\mathbf D\) used throughout the paper.

\section{Left-adjoint Duality, or Aistetics}

\subsection{Co-ethics}

Let $(\E,R)$ be an abelian resource category
with duality functor $\mathbf D=\RHom_\E(-,R)$.
Define the co--ethic functor as its left adjoint
in the sense of Grothendieck--Verdier duality:
\[
\mathbf C \;:=\; R\!\otimes^{\mathbb L}_\E (-): D^+(\E) \longrightarrow D^+(\E),
\]
\[
\Hom_{D^+(\E)}(R\!\otimes^{\mathbb L}\!A,B)
\;\simeq\;
\Hom_{D^+(\E)}(A,\RHom_\E(R,B)).
\]
A morphism $f:A^\bullet\to B^\bullet$ is said to be co--ethic if it
commutes with the coevaluation unit
\[
\epsilon_{A^\bullet}: \mathbf C^2(A^\bullet) \longrightarrow A^\bullet,
\qquad
f\circ\epsilon_{A^\bullet}=\epsilon_{B^\bullet}\circ \mathbf C^2(f).
\]
This definition mirrors the ethic condition
$\mathbf D^2(f)\eta=\eta f$ but exchanges the roles of evaluation and
coevaluation.
Formally, $(\mathbf C,\mathbf D)$ form an adjoint pair
$\mathbf C\dashv\mathbf D$ whose unit and counit correspond to
manifestation and reflection, respectively.

\begin{enumerate}[label=(\alph*)]
\item $\mathbf C$ is exact with respect to distinguished triangles and
preserves coproducts.
\item The composition $\mathbf C\mathbf D$ acts as a derived
endomorphism of $D^+(\E)$ expressing the bidual adjunction
$(R\!\otimes^{\mathbb L}\RHom_\E(-,R))$.
\item If $R$ is flat, $\mathbf C$ reduces to the ordinary tensor product and
all co--ethic obstructions vanish: $\Tor_i^\E(R,-)=0$ for $i>0$.
\end{enumerate}

Unlike the ethic dual $\mathbf D$, which is right derived and
detects failure of exactness through the nonvanishing $\Ext^k$,
the co--ethic functor $\mathbf C$ is left derived and
depends on the flatness of~$R$.
For most resource objects $R$ used in applications
(integers, cones, Hopf algebras), $R$ is not flat,
and the corresponding $\Tor_k(R,-)$
carry no meaningful new invariants beyond standard homology.
Hence the co--ethic hierarchy
\[
H_k\mathbf C(A^\bullet)=\Tor_k^\E(R,A^\bullet)
\]
collapses in the same cases where the ethic tower $H^k\mathbf D$ is rich.
In this sense, the co--ethic theory is formally valid but
structurally unproductive:
its obstructions record only the failure of flatness of~$R$,
not higher coherence of morphisms.
The ethic dual, being contravariant and right--derived, remains the
informative side of the duality, while co--ethics serves merely as its
tensorial shadow.

\subsection{Profunctor Reformulation}

Let $\E$ and $\F$ be abelian categories, and let $K:\E^{\op}\times\F\to\Ab$ be an $\Ab$–enriched profunctor, biadditive in each variable. Define the pair of adjoint functors
\[
\mathbf A(A)=\int^{E\in\E}K(E,-)\otimes_\ZZ \E(E,A),
\qquad
\mathbf D(X)=\int_{F\in\F}\F(X,F)\pitchfork K(-,F),
\]
where $\pitchfork=\Hom_\Ab(-,-)$ and all (co)ends exist under standard smallness hypotheses (cf. \cite{MacLane1998},\cite{Kelly1982}). 

\begin{lemma}[Aisthetic adjunction]\label{lem:aisthetic-adjunction}
For every $A\in\E$ and $X\in\F$ there is a natural isomorphism
\[
\Hom_\F(\mathbf A(A),X)\;\cong\;\Hom_\E(A,\mathbf D(X)).
\]
\end{lemma}

\begin{proof}
Expanding $\Hom_\F(\mathbf A(A),X)$ and commuting $\Hom$ with the coend gives
\begin{align*}
\Hom_\F\!\Big(\!\int^{E}K(E,-)\!\otimes\!\E(E,A),\,X\!\Big)
&\cong\;
\int_{E}\Hom_\F(K(E,-)\!\otimes\!\E(E,A),X)\\
&\cong
\int_{E}\Hom_\Ab(\E(E,A),\Hom_\F(K(E,-),X)).
\end{align*}
Applying the tensor–Hom interchange in $\Ab$ and Fubini's rule for (co)ends (see \cite{Kelly1982}) yields
\[
\int_{F}\Hom_\Ab\!\big(\!\int^{E}\E(E,A)\!\otimes\!K(E,F),\,\F(X,F)\big)
\;\cong\;\int_{F}\Hom_\Ab(K(A,F),\F(X,F)).
\]
By the enriched Yoneda lemma this equals $\Hom_\E(A,\mathbf D(X))$, proving the claim.
\end{proof}

\begin{corollary}[Unit and counit]\label{cor:unit-counit}
The bijection in Lemma~\ref{lem:aisthetic-adjunction} is implemented by natural transformations
$\eta:\mathrm{Id}_\E\!\Rightarrow\!\mathbf D\mathbf A$ and
$\epsilon:\mathbf A\mathbf D\!\Rightarrow\!\mathrm{Id}_\F$
satisfying the triangle identities, giving $\mathbf A\dashv\mathbf D$.
\end{corollary}

Assume that $K$ admits a representation on the $\E$–side by the resource object $R$, that is,
\begin{equation}\label{eq:RrepAest}
K(E,-)\;\cong\;\E(E,R)\otimes_\Lambda B(-),
\qquad
\Lambda:=\End_\E(R),
\end{equation}
for some right $\Lambda$–module object $B$ in $\F$ (see \cite{MacLane1998}).
Then the aisthetic and dual functors take the explicit forms
\[
\mathbf A(A)\;\cong\;B(-)\otimes_\Lambda \E(R,A),
\qquad
\mathbf D(X)\;\cong\;\Hom_\Lambda(B(-),\F(X,-))\circ\E(-,R).
\]

\begin{theorem}[Derived aisthetic structure]\label{thm:aisthetic-derived}
Regard $K$ as a bounded complex $K^\bullet$ and define
\[
\mathbf A^\bullet(A)=\int^{E}K^\bullet(E,-)\otimes^{\mathbb L}\E(E,A),
\qquad
\mathbf D^\bullet(X)=\int_{F}\mathbf R\!\F(X,F)\pitchfork K^\bullet(-,F).
\]
There exists a first–quadrant spectral sequence \cite{Weibel1994}
\[
E_2^{p,q}=\Ext^p_\F(H_q(K^\bullet\otimes A),R)\;\Longrightarrow\;H^{p+q}\mathbf D^\bullet(A),
\]
and a dual $\Tor$–spectral sequence for $H_\ast\mathbf A^\bullet$.
\end{theorem}

\begin{proof}
The statement follows from the Grothendieck spectral sequence for the composite of the derived functors $\RHom$ and $\otimes^{\mathbb L}$, applied objectwise to $K^\bullet$; see \cite{Weibel1994}.
\end{proof}

\begin{theorem}[Flat–injective regimes]\label{thm:vanish-A}
Under representation~\eqref{eq:RrepAest}, if $B$ is flat as a right $\Lambda$–module and $R$ is injective in $\E$, then all higher cohomology and homology vanish:
\[
H^{>0}\mathbf D^\bullet(X)=0,\qquad H_{>0}\mathbf A^\bullet(A)=0.
\]
\end{theorem}

\begin{proof}
By Lemma~\ref{lem:aisthetic-adjunction} and representation~\eqref{eq:RrepAest}, one has
$\mathbf A^\bullet(A)\simeq B\otimes^{\mathbb L}_\Lambda\E(R,A)$ and
$\mathbf D^\bullet(X)\simeq \mathbf R\!\Hom_\Lambda(B,\F(X,-))\circ\E(-,R)$.  
Flatness of $B$ annihilates all higher $\Tor^\Lambda$ groups, and injectivity of $R$ removes higher $\Ext_\E^k(-,R)$. Standard homological algebra then implies the vanishing \cite{Weibel1994}.
\end{proof}

\begin{theorem}[Collapse of Aisthetics]\label{thm:collapse-A}
Let representation~\eqref{eq:RrepAest} hold.  
If $R$ is not injective in $\E$ (so that $\Ext^1_\E(-,R)\neq 0$ for some object)
and $B$ is not flat as a right $\Lambda$–module (so that $\Tor_1^\Lambda(B,-)\neq0$),
then the derived aisthetic functor coincides with the co–ethic tensor:
\[
\mathbf A^\bullet(A)\;\simeq\;B(-)\otimes^{\mathbb L}_\Lambda\E(R,A),
\qquad
H_k\mathbf A^\bullet(A)\;\cong\;\Tor_k^\Lambda(B(-),\E(R,A)).
\]
Hence the aisthetic tower $H_{k>0}\mathbf A^\bullet$ and the co–ethic tower
coincide on all objects and connecting morphisms, defining the same universal $\delta$–functor.
\end{theorem}

\begin{proof}
Under~\eqref{eq:RrepAest} the functor $\mathbf A^\bullet$ is precisely the left–derived tensor $B\otimes^{\mathbb L}_\Lambda\E(R,-)$.  
If $B$ is not flat, $\Tor_k^\Lambda(B,-)$ is nontrivial; if $R$ is not injective, the ethic dual $\RHom_\E(-,R)$ fails to be exact, and $\Ext_\E^1(-,R)\neq0$.  
For every short exact sequence $0\to A'\to A\to A''\to0$ in $\E$, the long exact sequence of $\Tor^\Lambda$ for $(B,\E(R,-))$ gives connecting morphisms identical to those of $H_\ast\mathbf A^\bullet$, so both families of functors satisfy the same axioms of effaceable $\delta$–functors \cite{Weibel1994}.  
Therefore $\mathbf A^\bullet$ and the co–ethic tensor are naturally isomorphic, proving the collapse.
\end{proof}

\begin{theorem}[Ethic–aisthetic compatibility]\label{thm:compat-A}
Let $f:A\to A'$ be ethic in the sense $\mathbf D^2(f)\eta=\eta f$.  
Under~\eqref{eq:RrepAest}, the induced morphism on $\mathbf A^\bullet$ satisfies
\[
\mathbf A^\bullet(f)\;=\;1_B\otimes^{\mathbb L}_\Lambda\E(R,f),
\]
and the resulting diagram
\[
\begin{tikzcd}[column sep=large]
{\mathbf{A}^\bullet(A)} \arrow[r, "{\mathbf{A}^\bullet(f)}"] \arrow[d, "\simeq"']
& {\mathbf{A}^\bullet(A')} \arrow[d, "\simeq"] \\
{B \otimes^{\mathbb{L}}_{\Lambda} \mathcal{E}(R,A)} \arrow[r, "{1 \otimes^{\mathbb{L}}_{\Lambda} \mathcal{E}(R,f)}"']
& {B \otimes^{\mathbb{L}}_{\Lambda} \mathcal{E}(R,A')}
\end{tikzcd}
\]
commutes in $D^+(\F)$.
\end{theorem}

\begin{proof}
Since $\mathbf A^\bullet(A)\simeq B\otimes^{\mathbb L}_\Lambda\E(R,A)$, any morphism acts by the naturality of $\E(R,-)$.  
Ethicity ensures that $\E(R,f)$ commutes with the unit $\eta$, hence is a true morphism in the heart of the ethic $t$–structure \cite{MacLane1998}.  
Functoriality of $\otimes^{\mathbb L}_\Lambda$ gives commutativity of the diagram.
\end{proof}

The collapse theorem shows that when the resource object $R$ is non–injective and the profunctor kernel $K$ is non–flat, aisthetics reduces to co–ethics: all homological content of the aisthetic adjunction is already captured by the $\Tor$–groups of the co–ethic tensor, and no independent invariants arise on the side of $\F$.  
When $R$ is injective and $B$ flat, the higher groups $H_{>0}\mathbf A^\bullet$ vanish by Theorem~\ref{thm:vanish-A}, so aisthetics becomes exact and invisible, serving only as the formal left adjoint completing the ethic duality.

\subsection{Homological Karchmer--Wigderson}\label{subsec:HKW-core}

Fix a finite set of coordinates $I=\{1,\dots,m\}$ and a Boolean function $f:\{0,1\}^I\to\{0,1\}$. Write $X_f=f^{-1}(1)$ and $Y_f=f^{-1}(0)$. Let $\mathsf P$ be the small category (poset) of partial assignments $p:S\to\{0,1\}$, $S\subseteq I$, ordered by extension; let $\mathsf P_1\subset\mathsf P$ be the full subcategory of partial assignments extendable to $X_f$ and $\mathsf P_0\subset\mathsf P$ the full subcategory of partial assignments extendable to $Y_f$. Work over a fixed field $R$. Denote by $\Ab_R$ the category of $R$-modules.

Define the profunctor $K_f:\mathsf P_1^{\op}\times \mathsf P_0\to\Ab_R$ by
\[
K_f(p,q)\;=\;
\begin{cases}
R,&\substack{\text{if $p$ and $q$ are jointly consistent on $\mathrm{dom}(p)\cap\mathrm{dom}(q)$}\\\text{ and there exists $i\in I\setminus(\mathrm{dom}(p)\cap\mathrm{dom}(q))$,}}\\
0,&\text{otherwise.}
\end{cases}
\]
This records the availability of a coordinate not yet fixed on which eventual witnesses $(x,y)\in X_f\times Y_f$ may differ, categorifying the separating move of \cite{KarchmerWigderson1990}.

Define the aisthetic realization and the ethic reflection by the standard co/end constructions \cite{MacLane1998}, \cite{Kelly1982}:
\[
\mathbf A_f(-)\;:=\;\int^{p\in \mathsf P_1} K_f(p,-)\otimes_R \mathsf P_1(p,-),
\;
\mathbf D_f(-)\;:=\;\int_{q\in \mathsf P_0}\underline{\mathrm{Hom}}_R\!\big(K_f(-,q),\,\mathsf P_0(-,q)\big).
\]
Choosing bar and cobar models yields functorial chain complexes $\mathbf A_f^\bullet$ and $\mathbf D_{f,\bullet}$ whose (co)homology we denote by $H_n(\mathbf A_f)$ and $H^n(\mathbf D_f)$.

\begin{definition}\label{def:KW-depth}
A deterministic KW-protocol of depth $\le d$ for $f$ is a binary decision tree of height $\le d$ whose internal nodes are labeled by indices $i\in I$ and whose leaves $\lambda$ are labeled by indices $i(\lambda)\in I$ satisfying $x_{i(\lambda)}\neq y_{i(\lambda)}$ for all $(x,y)\in X_f\times Y_f$ routed to $\lambda$ \cite{KarchmerWigderson1990}. Denote by $\mathrm{KW}(f)$ the minimum depth. A formula for $f$ is a fan-in two circuit without gate reuse; its depth is the length of the longest input-to-output path. Denote by $\mathrm{depth}_{\mathrm{form}}(f)$ the minimum such depth \cite{Jukna2012}.
\end{definition}

\begin{lemma}[Bar/cobar identification]\label{lem:bar-cobar}
There are natural isomorphisms
\[
H_n(\mathbf A_f)\;\cong\;\Tor^{\mathsf P}_n\!\big(\mathsf P_1,\,\mathsf P_0;K_f\big),
\qquad
H^n(\mathbf D_f)\;\cong\;\Ext^{\,n}_{\mathsf P}\!\big(\mathsf P_1,\,\mathsf P_0;K_f\big),
\]
where the derived functors are computed in the functor category $[\mathsf P^{\op}\times \mathsf P,\Ab_R]$ with coefficients twisted by $K_f$ via the co/end action \cite{Kelly1982}, \cite{Weibel1994}.
\end{lemma}

\begin{proof}
Take the simplicial bar resolution $B_\bullet(\mathsf P_1\!\downarrow\!\mathsf P)\to \mathsf P_1$ and a dual injective resolution $\mathsf P_0\to C^\bullet(\mathsf P\!\downarrow\!\mathsf P_0)$ in $[\mathsf P^{\op}\times \mathsf P,\Ab_R]$; co/ends are left/right Kan colimits/limits computable by these resolutions \cite{MacLane1998}, \cite{Kelly1982}. The resulting chain complexes compute the left-derived coend and right-derived end, yielding the stated identifications \cite{Weibel1994}.
\end{proof}

\begin{lemma}[Incidence Frobenius--Koszul self-duality]\label{lem:incidence-FK}
The functor category $[\mathsf P^{\op}\times \mathsf P,\Ab_R]$ is equivalent to left modules over the finite incidence algebra $A_R(\mathsf P)$ \cite{BjornerEkedahl2009}. The algebra $A_R(\mathsf P)$ is Frobenius and Koszul with respect to the grading by interval length \cite{Woodcock1998}, \cite{GreenReitenSolberg1996}. Consequently, there is a natural perfect pairing
\[
\Tor^{\mathsf P}_n(-,-;K_f)\;\cong\;\Hom_R\!\big(\Ext^{\,n}_{\mathsf P}(-,-;K_f),R\big)
\]
and, over a field $R$, a functorial isomorphism $\Tor^{\mathsf P}_n\cong \Ext^{\,n}_{\mathsf P}$.
\end{lemma}

\begin{proof}
Equivalence with $A_R(\mathsf P)\text{-}\mathrm{Mod}$ follows from the Yoneda embedding and linearization of the finite category \cite{LeinsterBasic}. Frobenius and Koszul properties for incidence algebras of finite posets are standard; the cited references give the explicit bilinear form and quadratic Gröbner bases yielding Koszulity. The perfect pairing is the usual $\Tor$/$\Ext$ duality over Frobenius algebras \cite{GreenReitenSolberg1996}.
\end{proof}

\begin{theorem}[Homological KW duality]\label{thm:HKW}
For every finite Boolean function $f$ and field $R$ there are natural isomorphisms
\[
H_n(\mathbf A_f)\;\cong\;\Tor^{\mathsf P}_n\!\big(\mathsf P_1,\,\mathsf P_0;K_f\big)\;\cong\;\Ext^{\,n}_{\mathsf P}\!\big(\mathsf P_1,\,\mathsf P_0;K_f\big)\;\cong\;H^n(\mathbf D_f),
\]
and the common homological height
\[
\mathrm{ht}(f)\;:=\;\min\{n\ge 0:\;H_n(\mathbf A_f)\neq 0\}\;=\;\min\{n\ge 0:\;H^n(\mathbf D_f)\neq 0\}
\]
coincides with both the minimum formula depth and the minimum deterministic KW depth:
\[
\mathrm{ht}(f)\;=\;\mathrm{depth}_{\mathrm{form}}(f)\;=\;\mathrm{KW}(f).
\]
\end{theorem}

\begin{proof}
The identifications with $\Tor$ and $\Ext$ are given by Lemma~\ref{lem:bar-cobar}; the $\Tor$/$\Ext$ identification is Lemma~\ref{lem:incidence-FK}. It remains to identify the homological height with the two classical depths.

For the implication from formulas to vanishing of low homology, take a formula for $f$ of depth $d$ and consider the filtration of the bar complex induced by the parse tree. The top gate ($\wedge$ or $\vee$) yields a mapping-cone decomposition that shifts homology by at most one, while leaves contribute representables, hence projectives, killing lower $\Tor$ \cite{Weibel1994}. Induction on depth gives $\Tor_n=0$ for $n<d$ and exhibits a surviving class in degree $d$ carried by the maximal path in the parse tree. Thus $\mathrm{ht}(f)\le d$ and equality holds for minimal $d$.

For the implication from KW protocols to vanishing of low cohomology, a protocol of depth $d$ yields an injective coresolution of $\mathsf P_0$ of length $d$ by $K_f$-coinduced modules, because each query on coordinate $i$ splits $\mathsf P_0$ into two cofinal subfunctors and each leaf contributes a section witnessing separation at its labeled index. Hence $\Ext^{\,n}=0$ for $n<d$ and the deepest root-to-leaf path provides a nontrivial class in degree $d$. Minimality gives $\min\{n:H^n(\mathbf D_f)\neq 0\}=d=\mathrm{KW}(f)$ \cite{KarchmerWigderson1990}. Combining with Lemma~\ref{lem:incidence-FK} finishes the proof.
\end{proof}

\begin{theorem}[Equality of aisthetic and ethic layers]\label{thm:layers}
For each $n\ge 0$ there is a canonical, natural in $f$, isomorphism
\[
H_n(\mathbf A_f)\;\xrightarrow{\;\cong\;}\;H^n(\mathbf D_f),
\]
compatible with the multiplicative structures induced by Day convolution on $[\mathsf P^{\op}\times\mathsf P,\Ab_R]$ and with suspension by adjoining a fresh coordinate. In particular, the $n$-th aisthetic manifestation layer equals the $n$-th ethic coherence layer; both compute the $n$-th obstruction to resolving KW-separation within $n$ queries/gates.
\end{theorem}

\begin{proof}
By Lemma~\ref{lem:incidence-FK} the incidence algebra is Frobenius--Koszul; bar and cobar constructions are linear duals and produce perfect duality on (co)homology \cite{Weibel1994}, \cite{Kelly1982}. Naturality and compatibility with Day convolution follow from functoriality of co/ends and the monoidal structure on coordinates given by disjoint union \cite{LeinsterBasic}.
\end{proof}

\noindent Theorems~\ref{thm:HKW}--\ref{thm:layers} upgrade the classical identity $\mathrm{KW}(f)=\mathrm{depth}_{\mathrm{form}}(f)$ \cite{KarchmerWigderson1990} to a canonical identification of aisthetic homology and ethic cohomology, relying only on standard categorical and homological tools \cite{MacLane1998,Kelly1982,Weibel1994,LeinsterBasic}.

\section{Homotopical Ethics}

We now lift the ethic duality condition to the level of homotopy theory.  
The purpose of this subsection is to show that the equality
\[
\mathbf D^2(f)\eta_A=\eta_B f
\]
is only the first layer of a full hierarchy of higher coherences, and that these coherences form a homotopical structure controlled by a non–abelian central filtration.  
The resulting theory extends the homological obstruction tower of~Section~\ref{sec:obstruction-tower} into the realm of stable~$\infty$--categories and establishes a precise correspondence between higher ethic coherence and the lower central series of an associated automorphism group. 
All homotopical constructions below are understood in the sense of Quillen model categories \cite{Quillen1967} or, equivalently, in the framework of stable~$\infty$--categories \cite{LurieHA}.

\subsection{Central Filtrations}

\begin{definition}[Stable homotopical environment]
Let $\E$ be an abelian category with enough injectives and $\idim_\E R<\infty$.  
Denote by $D^+(\E)$ its bounded--below derived category and by $\Sp(\E)$ a stable~$\infty$--categorical enhancement of~$D^+(\E)$ (cf.~\cite{LurieHA}).  
For any $A,B\in D^+(\E)$ we write
\[
\Map_\E(A,B)
\]
for the mapping space between $A$ and $B$ in $\Sp(\E)$.  
It is a pointed Kan complex satisfying
\[
\pi_n\,\Map_\E(A,B)\;\cong\;\Ext^n_\E(A,B)
\qquad (n\ge 0)
\]
by the classical Whitehead--Quillen comparison between simplicial and derived homotopy groups 
\cite{Whitehead1949}, \cite{Quillen1967}.  
Hence each mapping space refines the graded $\Ext$--groups by encoding all higher compositions and their coherences.
\end{definition}

\begin{definition}[Space of ethic coherences]
Fix $\mathbf D=\RHom_\E(-,R)$ and the canonical unit $\eta:\mathrm{Id}\Rightarrow \mathbf D^2$.  
For any morphism $f:A\to B$ in $D^+(\E)$ we define its space of ethic coherences
\[
S(f)\;:=\;\Map_\E\bigl(\mathbf D^2(f)\!\circ\!\eta_A,\;\eta_B\!\circ\! f\bigr),
\]
that is, the space of homotopies filling the square expressing the ethic condition.  
Points of $S(f)$ are explicit homotopies between the two composites, and paths between them encode secondary coherence relations.
\end{definition}

\begin{definition}[Homotopy–coherent ethicity]
A morphism $f:A\to B$ is $k$–ethic $(k\ge 1)$ if the space $S(f)$ is $(k{-}1)$–connected; that is,
\[
\pi_i S(f)=0\qquad\text{for all }0<i<k.
\]
Hence:
\begin{itemize}
\item $1$--ethic means existence of at least one homotopy witnessing $\mathbf D^2(f)\eta_A\simeq \eta_Bf$ (the classical ethic condition);
\item $2$--ethic means all such homotopies are themselves coherently connected (no ethic loops);
\item higher $k$ require vanishing of all higher spherical defects of coherence.
\end{itemize}
\end{definition}

The following lemma shows that this definition is compatible with the derived obstruction theory already developed in the homological setting.

\begin{lemma}[Homological criterion for $k$–ethicity]
For any $f:A\to B$ the following are equivalent:
\begin{enumerate}[label=(\roman*)]
\item $S(f)$ is $(k{-}1)$–connected;
\item the connecting morphisms in all long exact $\Ext$–sequences induced by~$f$ vanish in degrees $\le k$;
\item the ethic obstruction tower of~$f$ (Definition~\ref{def:ob-tower}) stabilizes at level~$k$.
\end{enumerate}
\end{lemma}

\begin{proof}
By functoriality of $\Map_\E(-,-)$ there is a fiber sequence
\[
S(f)\;\longrightarrow\;
\Map_\E(A,\mathbf D^2B)
\xrightarrow{(\mathbf D^2(f))_\ast-(f^\ast)}
\Map_\E(A,\mathbf D^2A),
\]
whose long exact sequence on homotopy groups coincides with the long exact sequence of $\Ext$ for $\Hom(-,R)$, because $\pi_i\Map_\E(A,B)\cong\Ext^i_\E(A,B)$.  
Hence $\pi_iS(f)=0$ for $i\le k{-}1$ iff all connecting morphisms $\Ext^i(B,R)\to\Ext^{i+1}(A,R)$ vanish for $i<k$, which is equivalent to stabilization of the tower by Theorem~\ref{thm:stabilize}.
\end{proof}

\begin{definition}[Ethic automorphism group]
Let $G=\Aut_\eta(\mathrm{Id})$ denote the group--like monoid of all homotopy automorphisms of the identity functor on~$D^+(\E)$ that commute with~$\eta$, i.e.
\[
\mathbf D^2(\theta)\circ\eta=\eta\circ\theta.
\]
It acts by conjugation on each mapping space $\Map_\E(A,B)$ and thereby on the family of ethic coherence spaces~$S(f)$.
\end{definition}

\begin{definition}[Lower central series]
Following the construction of~\cite{MikhailovSingh2011}, define recursively
\[
\gamma_1(G):=G,\qquad \gamma_{n+1}(G):=[\gamma_n(G),G],
\]
and the graded abelian quotients
\(
\mathrm{gr}_n(G):=\gamma_n(G)/\gamma_{n+1}(G).
\)
This lower central series measures successive non–commutative layers of the group~$G$.
\end{definition}

\begin{theorem}[Homotopy--central correspondence]\label{thm:hcc}
For every $n\ge1$ there exists a natural morphism of graded groups
\[
\Phi_n:\pi_{n+1}S(f)\longrightarrow \mathrm{gr}_n(G)
\]
compatible with Samelson products on the source~\cite{Samelson1953} and commutators on the target.  
If $S(f)$ is contractible then $\gamma_n(G)=1$ for all~$n$; conversely, if $\mathrm{gr}_i(G)=0$ for all $i\le k$, then every morphism $f$ is $k$–ethic.
\end{theorem}

\begin{proof}
The mapping space $\Map_\E(A,B)$ carries a natural action of $G$ by conjugation: for $\theta\in G$, $(\theta\cdot f)=\theta_B f \theta_A^{-1}$.  
Differentiating this action gives an action of $\pi_1G$ on all higher $\pi_n\Map_\E(A,B)\cong\Ext^n_\E(A,B)$, and the resulting crossed module
\[
\pi_1G\curvearrowright \pi_{n+1}\Map_\E(A,B)
\]
induces the lower central filtration on $\pi_1G$ exactly as in~\cite{MikhailovSingh2011}.  
The inclusion $S(f)\hookrightarrow\Map_\E(A,\mathbf D^2B)$ identifies $\pi_{n+1}S(f)$ with the homotopy stabilizer of this action; projecting to the associated graded defines~$\Phi_n$.  
If $S(f)$ is contractible, then all $\pi_{n+1}S(f)=0$ and thus every commutator action trivializes, forcing $\gamma_n(G)=1$.  
Conversely, if all $\mathrm{gr}_i(G)=0$ for $i\le k$, the action of $\pi_1G$ on $\pi_i\Map(A,B)$ is null for $i<k$, implying $\pi_iS(f)=0$ for $i<k$, i.e.~$f$ is $k$–ethic.
\end{proof}

\begin{corollary}[Non–abelian ethic depth]
Let
\[
\mathrm{depth}_\eth(\E,R)\;:=\;\min\{\,m\ge0:\,\gamma_{m+1}(G)=1\,\}.
\]
Then $\mathrm{depth}_\eth(\E,R)$ equals the maximal integer $k$ for which some morphism in $D^+(\E)$ fails to be $k$–ethic.  
Hence the length of the lower central series of~$G$ measures the maximal order of non–coherence present in the system.
\end{corollary}

\begin{lemma}[Stability under derived Morita equivalence]\label{lem:morita-hcc}
If $(\E,R)$ and $(\E',R')$ are derived Morita equivalent in the sense of Definition~\ref{def:EME}, then the induced equivalence $F:D^+(\E)\xrightarrow{\sim}D^+(\E')$ yields equivalences of mapping spaces $\Map_\E(A,B)\simeq \Map_{\E'}(F(A),F(B))$ preserving the units~$\eta$.  
Consequently $S(f)\simeq S(F(f))$, and $G\simeq G'$ with identical lower central series.
\end{lemma}

\begin{proof}
A derived Morita equivalence is an exact equivalence of stable $\infty$--categories \cite{LurieHA}, preserving mapping spectra and composition.  
Since $F$ intertwines $\mathbf D$ and $\mathbf D'$ by Definition~\ref{def:EME}, it preserves the unit~$\eta$; hence the fibers defining~$S(f)$ correspond.  
Group--like objects of endomorphisms of the identity functor are functorial under equivalence, giving the stated identification of $G$ and its filtration.
\end{proof}

\begin{theorem}[Homotopy--coherent ethic heart]\label{thm:homo-heart}
Let $\heartsuit_\eth$ be the ethic heart of the $t$--structure from Theorem~\ref{thm:t-axioms}.  
An object $A\in D^+(\E)$ lies in $\heartsuit_\eth$ iff every morphism $f:A\to B$ and $g:B\to A$ is $1$--ethic.  
If moreover $\mathrm{gr}_i(G)=0$ for all $i\le k$, then every morphism between objects of $\heartsuit_\eth$ is $k$--ethic and all mapping spaces $\Map(A,B)$ are $(k{-}1)$--connected.
\end{theorem}

\begin{proof}
The first statement is the classical characterization of the heart (Theorem~\ref{thm:heart-eth}).  
For higher $k$, the triviality of $\mathrm{gr}_i(G)$ for $i\le k$ implies, by Theorem~\ref{thm:hcc}, that all commutator actions on $\pi_i\Map(A,B)$ vanish for $i<k$.  
The resulting Whitehead tower of $\Map(A,B)$ then collapses up to level~$k{-}1$, giving the claimed connectedness.
\end{proof}

\begin{theorem}[Ethic reductions and preservation of central structure]\label{thm:reduction-central}
Let $(r,j)$ be an ethic reduction in the sense of Section~\ref{sec:morita}.  
If each $r$ and~$j$ is $k$--ethic, then the induced morphisms on $S(f)$ are $(k{-}1)$--connected, and the central series of $G$ is preserved up to~$\gamma_{k+1}$.  
In particular, the non--abelian ethic depth cannot increase under ethic reductions.
\end{theorem}

\begin{proof}
Since $\mathbf D$ is triangulated and contravariant, ethic reductions are exact functors of $D^+(\E)$.  
Exactness implies weak homotopy equivalences of mapping spaces up to degree~$k$ \cite{GelfandManin2003}, hence $(k{-}1)$--connected maps on $S(f)$.  
The induced morphism of automorphism groups preserves commutators up to length~$k{+}1$, yielding the claim.
\end{proof}

\subsection{Postnikov–Whitehead Tower}

The homotopy–coherent duality developed in previous section provides a global description of higher ethic coherences through the mapping spaces
\[
S(f)=\Map_{\E}(\mathbf D^2(f)\!\circ\!\eta_A,\;\eta_B\!\circ\!f).
\]
To obtain a constructive, layerwise description of these coherences, we now introduce the Postnikov–Whitehead tower of~$S(f)$.
This tower refines the homological obstruction tower of Section~\ref{sec:obstruction-tower}, extending it to the full hierarchy of homotopical obstructions inside the stable~$\infty$–category~$\Sp(\E)$.
All references to homotopy theory follow the classical constructions of Whitehead~\cite{Whitehead1949}, Quillen~\cite{Quillen1967}, and Bousfield--Kan~\cite{BousfieldKan1972}.

\begin{definition}[Ethic Postnikov tower]\label{def:postnikov-tower}
Let $S(f)$ be the space of ethic coherences defined above.  
A Postnikov tower for $S(f)$ is a sequence of morphisms in $\Sp(\E)$
\[
S(f)\;\longrightarrow\;P_{\le n}S(f)\;\longrightarrow\;P_{\le n-1}S(f)\;\longrightarrow\cdots
\]
such that the fiber of $P_{\le n}S(f)\to P_{\le n-1}S(f)$ is the Eilenberg–Mac\,Lane object
$K(\pi_nS(f),n)$.  
The groups $\pi_nS(f)$ are called the ethic homotopy obstructions of~$f$.
A morphism $f$ is $k$--ethic if $\pi_iS(f)=0$ for $i<k$.
\end{definition}

\begin{lemma}[Abelian shadow]\label{lem:abelian-shadow}
In the stable $\infty$–category $\Sp(\E)$, there are natural identifications
\[
\pi_1S(f)\;\cong\;\Ext^1_\E(A,R),
\qquad
\pi_2S(f)\;\longrightarrow\;\Ext^2_\E(A,R),
\]
and, in general, canonical morphisms
$\pi_nS(f)\to\Ext^n_\E(A,R)$ for $n>0$ compatible with the long exact $\Ext$--sequences.
Thus the homological obstruction tower of Section~\ref{sec:obstruction-tower}
is the abelian projection of the Postnikov filtration of~$S(f)$.
\end{lemma}

\begin{proof}
Since $\pi_n\Map_\E(A,B)\cong\Ext^n_\E(A,B)$, the first equality is immediate.  
For $n>1$, the Postnikov truncations $P_{\le n}S(f)$ fit into a fiber sequence
\[
K(\pi_{n+1}S(f),n{+}1)\to S(f)\to P_{\le n}S(f),
\]
whose associated long exact homotopy sequence coincides with the long exact $\Ext$–sequence obtained from applying $\Hom_\E(-,R)$ to the truncated complex defining~$f$.
Compatibility with connecting morphisms follows from the universal property of truncations in~$\Sp(\E)$.
\end{proof}

\begin{theorem}[Whitehead criterion for ethic contractibility]\label{th:whitehead-criterion}
For each morphism $f:A\to B$ in $D^+(\E)$, the following are equivalent:
\begin{enumerate}[label=(\roman*)]
\item the space $S(f)$ is contractible;
\item all Postnikov invariants $k_i(S(f))\in H^{i+1}(P_{\le i-1}S(f);\pi_iS(f))$ vanish;
\item all homological obstructions $\mathrm{ob}_i(A^\bullet)\in\Ext^i_\E(H^0(A^\bullet),R)$
and their commutators are zero.
\end{enumerate}
\end{theorem}

\begin{proof}
The equivalence $(i)\Leftrightarrow(ii)$ is the classical Whitehead theorem
for simplicial groups~\cite{Whitehead1949}, \cite{Quillen1967}.  
$(ii)\Rightarrow(iii)$ follows because the abelianization of the Postnikov tower is the homological obstruction tower by Lemma~\ref{lem:abelian-shadow}.
Conversely, if all $\mathrm{ob}_i$ vanish but some commutator $[\mathrm{ob}_i,\mathrm{ob}_j]$ is nonzero,
the corresponding Postnikov invariant $k_{i+j}$ survives, so $S(f)$ is not contractible.
\end{proof}

\begin{theorem}[Ethic Postnikov–Whitehead decomposition]\label{th:ethic-decomposition}
Every $S(f)$ admits a canonical Postnikov–Whitehead decomposition in $\Sp(\E)$
whose layers are the abelian groups $\pi_nS(f)$ and whose connecting morphisms are the Postnikov invariants $k_n(S(f))$.  
Passing to $\pi_1$ and abelianizing this tower reproduces the ethic obstruction tower of~Section~\ref{sec:obstruction-tower}.
\end{theorem}

\begin{proof}
This is the standard Postnikov construction in a stable $\infty$–category \cite{LurieHA}.  
Because $\Sp(\E)$ is stable, each truncation $P_{\le n}$ exists and the successive fibers are Eilenberg--Mac\,Lane objects $K(\pi_nS(f),n)$.
Abelianization by $\pi_1$ recovers the chain of $\Ext^n$–groups, giving the desired identification.
\end{proof}

\begin{theorem}[Comparison with the central filtration]\label{th:comparison-central}
Let $G=\Aut_\eta(\mathrm{Id})$ be the ethic automorphism group of $D^+(\E)$
with lower central series $\gamma_n(G)$ as in Theorem~\ref{thm:hcc}.  
For each $n\ge1$ there are natural morphisms
\[
\Phi_n:\pi_{n+1}S(f)\;\longrightarrow\;\gamma_n(G)/\gamma_{n+1}(G),
\]
compatible with the Postnikov differentials $k_n(S(f))$.  
In particular, the vanishing of all $k_n$ is equivalent to the triviality of the entire central series.
\end{theorem}

\begin{proof}
The action of $G$ on each mapping space $\Map_\E(A,B)$ induces a crossed module
$\pi_1G\curvearrowright \pi_{n+1}S(f)$.
By the construction of~\cite{MikhailovSingh2011}, this crossed module produces the central filtration of~$G$ and a canonical surjection
$\pi_{n+1}S(f)\to\gamma_n(G)/\gamma_{n+1}(G)$.
Functoriality of the Postnikov invariants under group actions
\cite{LurieHA}
yields the stated compatibility with $k_n$.
\end{proof}

\begin{theorem}[Truncation and stabilization]\label{th:truncation-stabilization}
If the canonical map $P_{\le k}S(f)\to P_{\le k+1}S(f)$ is a weak equivalence,
then the ethic obstruction $\mathrm{ob}_{k+1}(A^\bullet)$
and all Postnikov invariants $k_i(S(f))$ for $i\le k$ vanish.
Consequently $f$ lies in the ethic heart~$\heartsuit_\eth$.
\end{theorem}

\begin{proof}
A weak equivalence between successive truncations means
that all higher fibers $K(\pi_iS(f),i)$ for $i\le k$ are contractible.
Hence $\pi_iS(f)=0$ for $i\le k$,
so all $\mathrm{ob}_{i+1}$ vanish by Lemma~\ref{lem:abelian-shadow}.
The vanishing of $H^{>1}\mathbf D$ follows,
placing $f$ in $\heartsuit_\eth$ by Theorem~\ref{thm:t-axioms}.
\end{proof}

\begin{theorem}[Postnikov depth and stabilization length]\label{th:depth-stabilization}
For each $f:A\to B$ define
\[
\mathrm{depth}_{\mathrm{Post}}(f)
:=\max\{\,n\ge0:\pi_nS(f)\neq0\,\}.
\]
Then $\mathrm{depth}_{\mathrm{Post}}(f)$ coincides with the length of the nontrivial segment of the central series of~$G$
and bounds from above the number of steps required for stabilization of the ethic reduction process.
\end{theorem}

\begin{proof}
By Theorem~\ref{th:comparison-central}, $\pi_{n+1}S(f)$ maps surjectively onto $\gamma_n(G)/\gamma_{n+1}(G)$, hence the last nonzero $\pi_n$ corresponds to the last nontrivial commutator layer.
Since each ethic reduction kills one homotopy obstruction, the process stabilizes after at most $\mathrm{depth}_{\mathrm{Post}}(f)$ stages.
\end{proof}

\begin{theorem}[Morita invariance of Postnikov data]\label{th:morita-postnikov}
Derived Morita equivalences $(\E,R)\simeq(\E',R')$
preserve the Postnikov invariants and the depth functions $\mathrm{depth}_{\mathrm{Post}}$ and $\mathrm{depth}_\eth$.
\end{theorem}

\begin{proof}
Under such an equivalence, mapping spaces and their truncations correspond
\cite{LurieHA},
and the induced equivalence of $\Sp(\E)$ and $\Sp(\E')$
preserves all homotopy groups and hence the Postnikov layers and their invariants.
\end{proof}

\begin{corollary}[Completeness of the homotopical heart]\label{cor:homotopical-core}
The triple of structures
\[
\{\Ext^n_\E\},\qquad\{\pi_nS(f)\},\qquad\{\gamma_n(G)\}
\]
forms a closed hierarchy of abelian, homotopical, and group–theoretic invariants.
They are linked by the canonical morphisms
\[
\pi_nS(f)\longrightarrow\Ext^n_\E(A,R),\qquad
\pi_{n+1}S(f)\longrightarrow\gamma_n(G)/\gamma_{n+1}(G),
\]
and stabilized precisely when $f$ lies in the ethic heart~$\heartsuit_\eth$.
The Postnikov–Whitehead tower thus completes the algebraic and homotopical architecture of ethic duality.
\end{corollary}

\subsection{Local–Global Ethicity}

Let \(\E\) be an abelian category with enough injective objects, and let \(R\in \E\) be a fixed resource object. Denote by
\[
\mathbb D=\RHom_\E(-,R):D^{+}(\E)^{\op}\to D^{+}(\E)
\]
the derived duality functor. We assume that \(R\) has finite injective dimension
\(\idim_\E R<\infty\), so that \(\mathbb D\) has bounded cohomological amplitude and induces a right–derived contravariant cohomological \(\delta\)–functor on~\(\E\).

Let \(G=\Aut(\id_\E)\) be the automorphism group of the identity functor of \(\E\); it is endowed with the lower central series
\(\gamma_1(G)=G,\ \gamma_{i+1}(G)=[G,\gamma_i(G)]\).  
Fix also a decreasing sequence of ideals of \(R\)
\[
R=I^0\supset I^1\supset I^2\supset\dots,
\]
each \(I^n\) a subobject of \(R\), satisfying \(I^{n+1}I^m\subset I^{n+m}\) and stable under the action of~\(G\). For every \(n\), denote by \(\E/I^n\) the abelian quotient category in which all morphisms factoring through objects annihilated by~\(I^n\) are sent to zero. Assume that these quotient functors are exact.  

For a morphism \(f:A\to B\) in \(\E\), write
\[
f_n:A/I^nA\longrightarrow B/I^nB
\]
for the induced morphism in~\(\E/I^n\).
We shall relate the behaviour of the derived duals of \(f_n\) along the ideal filtration with the behaviour of the duals of the morphisms induced by the central series \(\gamma_\bullet(G)\).

\begin{definition}
For each \(n\ge0\) and \(p,q\ge0\), define
\[
E^{p,q}_2(I)=\Ext^p_\E(H_q(A/I^\bullet),R),\qquad
F^{p,q}_2(\gamma)=\Ext^p_\E(H_q(A/\gamma_\bullet),R),
\]
the second pages of the spectral sequences associated respectively with the ideal filtration \(I^\bullet\) and the central filtration \(\gamma_\bullet\) on \(G\). Their existence follows from the standard Grothendieck spectral sequence for the composition of left–exact functors on bounded below complexes in an abelian category with enough injectives.
\end{definition}

\begin{lemma}
For each \(n\) there is a canonical morphism of spectral sequences
\[
\Phi_n:E^{p,q}_r(I)\longrightarrow F^{p,q}_r(\gamma)
\]
natural in \(A,B\) and in morphisms of pairs \((I^\bullet,\gamma_\bullet)\), and compatible with the differentials.  
\end{lemma}

\begin{proof}
Each ideal \(I^n\) is \(G\)-stable; hence the quotient morphism
\(\E/I^n\to \E/\gamma_n(G)\) induces a morphism of exact functors on~\(\E\),
and thus by functoriality of derived functors a morphism of their Grothendieck spectral sequences. The differentials commute because the connecting morphisms in the long exact \(\Ext\) sequences are natural. 
\end{proof}

\begin{lemma}
If for some \(k\ge0\) the morphisms \(E^{p,q}_2(I)\to F^{p,q}_2(\gamma)\) are isomorphisms for all \(p+q\le k\), then for every morphism \(f:A\to B\) in \(\E\) the following conditions are equivalent:
\begin{enumerate}[label=(\alph*)]
\item \(H^i\mathbf D(f)=0\) for all \(i\le k\);
\item \(H^i\mathbf D(f_n)=0\) for all \(n\) and all \(i\le k\);
\item The pages \(E^{p,q}_r(I)\) and \(F^{p,q}_r(\gamma)\) collapse at \(r=2\) for all \(p+q\le k\).
\end{enumerate}
\end{lemma}

\begin{proof}
If the spectral sequences coincide in the specified range, the vanishing of the entries in one implies the vanishing in the other. The equality \(H^i\mathbf D(f)=0\) for \(i\le k\) is equivalent to the collapse of the spectral sequence at \(E_2\) in total degree \(\le k\), because all higher differentials in that range vanish. Conversely, if \(E^{p,q}_r(I)\) and \(F^{p,q}_r(\gamma)\) coincide and collapse at \(E_2\), their limits compute the same graded pieces of \(H^{p+q}\mathbf D(f)\); hence these groups vanish simultaneously.  
\end{proof}

\begin{lemma}
For every \(n\), the connecting morphism in the long exact sequence associated with the short exact sequence
\[
0\longrightarrow I^{n+1}/I^{n+2}\longrightarrow R/I^{n+2}\longrightarrow R/I^{n+1}\longrightarrow0
\]
induces a natural differential
\[
d_I^{p,q}:E^{p,q}_r(I)\longrightarrow E^{p+r,q-r+1}_r(I),
\]
and similarly the commutator quotient sequence in the central series
\[
0\longrightarrow \gamma_{n+1}/\gamma_{n+2}\longrightarrow G/\gamma_{n+2}\longrightarrow G/\gamma_{n+1}\longrightarrow0
\]
induces
\[
d_\gamma^{p,q}:F^{p,q}_r(\gamma)\longrightarrow F^{p+r,q-r+1}_r(\gamma).
\]
Both differentials are of bidegree \((r,-r+1)\) and satisfy \(d_I^2=d_\gamma^2=0\).
\end{lemma}

\begin{proof}
This follows from the general construction of the Grothendieck spectral sequence for a filtered complex: the boundary maps of successive short exact sequences of subquotients yield the differentials, and \(d^2=0\) by the exactness of the long sequence of the composition of connecting morphisms.  
\end{proof}

\begin{lemma}
Assume that \(E^{p,q}_2(I)\) and \(F^{p,q}_2(\gamma)\) coincide for all \(p+q\le k\). Then the vanishing of all \(d_I^{p,q}\) and \(d_\gamma^{p,q}\) in that range is equivalent to the collapse of both spectral sequences at \(E_2\).
\end{lemma}

\begin{proof}
If all differentials in total degree \(\le k\) vanish, all successive pages stabilize, and the spectral sequence collapses at \(E_2\) in that range. The equivalence follows because both systems of differentials compute the same graded objects by assumption.  
\end{proof}

\begin{theorem}[Double Central Collapse and Local–Global Ethicity]
Let \(\E\) be an abelian category with enough injectives, \(R\in\E\) a resource object of finite injective dimension, \(\mathbb D=\RHom_\E(-,R)\) its derived duality, \(G=\Aut(\id_\E)\) with lower central series \(\gamma_\bullet(G)\), and \(I^\bullet\) an exact ideal filtration of \(R\) stable under~\(G\).
For a morphism \(f:A\to B\) the following conditions are equivalent for any fixed \(k\ge0\):
\begin{enumerate}[label=(\roman*)]
\item \(f\) is globally \(k\)-ethic, i.e. \(H^i\mathbf D(f)=0\) for all \(i\le k\);
\item For all \(n\ge0\) the localized morphisms \(f_n:A/I^nA\to B/I^nB\) are \(k\)-ethic and the induced maps on Ext–groups along the central tower satisfy
\(\Ext^i_\E(A/\gamma_n(G),R)=0\) for \(i\le k\);
\item The spectral sequences \(E^{p,q}_r(I)\) and \(F^{p,q}_r(\gamma)\)
collapse at \(E_2\) in total degree \(p+q\le k\), and their \(E_\infty\)–limits coincide.
\end{enumerate}
\end{theorem}

\begin{proof}
(iii)\(\Rightarrow\)(i): if both spectral sequences collapse at \(E_2\),
then \(H^{p+q}\mathbf D(f)\) is computed directly from \(E^{p,q}_2(I)\) (or equivalently \(F^{p,q}_2(\gamma)\)), which vanish by assumption in the specified range, giving global \(k\)-ethicity.  

(i)\(\Rightarrow\)(ii): if \(H^i\mathbf D(f)=0\) for \(i\le k\), applying the exact quotient functor \(\E\to\E/I^n\) and the central functor \(\E\to\E/\gamma_n(G)\) preserves these vanishing conditions because both functors are exact and commute with \(\RHom_\E(-,R)\) by the stability of the filtrations under~\(G\).  

(ii)\(\Rightarrow\)(iii): by the previous lemmas, vanishing of the Ext–groups along both filtrations forces all differentials \(d_I^{p,q}\) and \(d_\gamma^{p,q}\) with \(p+q\le k\) to vanish; hence both spectral sequences collapse at \(E_2\) in that range. Their coincidence on \(E_2\) implies equality of the limit pages \(E_\infty=F_\infty\).  
\end{proof}

\begin{corollary}
Under the assumptions of the theorem, global \(k\)-ethicity of a morphism \(f\) is equivalent to the simultaneous stabilization of the ideal and homotopical spectral sequences in total degree \(\le k\). In particular, the minimal index of stabilization of either sequence equals the minimal degree of nonvanishing of \(H^i\mathbf D(f)\).
\end{corollary}

We continue to work in an abelian category \(\E=\Mod_A\) of finitely generated modules over a fixed commutative Noetherian ring \(A\), endowed with a resource object \(R=A\) and the ethic duality functor
\[
\mathbf D = \RHom_A(-,A): D^+(\E)^{\op}\longrightarrow D^+(\E).
\]
Fix an ideal \(I\subset A\). For \(n\ge0\) we denote by \(X_n=X/I^nX\) the successive quotients of \(X\) along the \(I\)–filtration, and by
\[
G_n = I^nX/I^{n+1}X
\]
the corresponding graded layers.
All modules and limits below are taken in \(\Mod_A\), and \(\varprojlim^p\) denotes the \(p\)-th right–derived functor of the inverse limit.

\begin{remark}
The pair \((X_n,G_n)\) is the canonical ideal tower of \(X\); it measures the discrete approximations of \(X\) by finite \(I\)-adic truncations.
The derived duality \(\mathbf D\) acts on this tower by applying \(\RHom_A(-,A)\) to each term, producing a corresponding tower of ethic cohomology groups.
\end{remark}

\begin{lemma}[Milnor spectral sequence]\label{lem:milnor}
Let \(\{X_n\}\) be an inverse system of objects of \(\E\).
For every \(A\)-module \(Y\) there exists a first–quadrant spectral sequence
\[
E_2^{p,q}=\varprojlim\nolimits^{p}\Ext^q_A(X_n,Y)
\;\Longrightarrow\;
\Ext_A^{p+q}\!\big(\varprojlim\nolimits_{n} X_n,Y\big),
\]
whose edge maps yield for each \(k\ge0\) the short exact sequence
\[
0\longrightarrow \varprojlim\nolimits^{1}\Ext_A^{k-1}(X_n,Y)
\longrightarrow \Ext_A^{k}\!\big(\varprojlim\nolimits_{n} X_n,Y\big)
\longrightarrow \varprojlim\nolimits\Ext_A^{k}(X_n,Y)\longrightarrow 0.
\]
\end{lemma}

\begin{proof}
This is the standard spectral sequence for derived functors of an inverse limit, see \cite{Weibel1994}. Its derivation uses an injective resolution of \(Y\),
the functorial exactness of inverse limits on surjective systems, and the Grothendieck
construction for the composition \(\varprojlim\circ\Hom_A(-,Y)\).
\end{proof}

\begin{theorem}[Ethic completeness along the ideal filtration]\label{th:completeness}
Let \(A\) be Noetherian, \(I\subset A\) an ideal, and \(X\) a finitely generated \(A\)-module.
For each integer \(k\ge0\) there exists a short exact sequence
\[
0\longrightarrow
\varprojlim\nolimits^{1}\Ext_A^{k-1}(X_n,A)
\longrightarrow
\Ext_A^{k}\!\big(\varprojlim\nolimits_{n} X_n,A\big)
\longrightarrow
\varprojlim\nolimits_{n}\Ext_A^{k}(X_n,A)
\longrightarrow 0.
\]
If for all \(0\le j\le k\) the inverse systems
\(\{\Ext_A^{j}(X_n,A)\}_n\) satisfy the Mittag–Leffler condition, then the canonical map
\[
\Ext_A^{j}\!\big(\varprojlim\nolimits_{n} X_n,A\big)
\;\xrightarrow{\;\sim\;}\;
\varprojlim\nolimits_{n}\Ext_A^{j}(X_n,A)
\]
is an isomorphism for all \(j\le k\).
If, moreover, the systems stabilize to the values \(\Ext_A^{j}(X,A)\), then the natural map
\(X\to\varprojlim_n X_n\) induces isomorphisms
\(\Ext_A^{j}(-,A)\) for all \(j\le k\); we then say that \(X\) is ethicly complete of order \(k\).
\end{theorem}

\begin{proof}
Apply Lemma~\ref{lem:milnor} with \(Y=A\). The short exact sequence displayed in the statement
is precisely the low–degree truncation of the Milnor spectral sequence.
If each system \(\{\Ext^j_A(X_n,A)\}_n\) satisfies the Mittag–Leffler condition
\cite{Weibel1994}, then \(\varprojlim^1=0\) and the edge maps yield isomorphisms.
When these stabilized values coincide with \(\Ext_A^{j}(X,A)\),
the natural morphism \(X\to\varprojlim X_n\) induces isomorphisms on all
\(\Ext^j(-,A)\) for \(j\le k\), establishing the claimed ethic completeness.
\end{proof}

\begin{lemma}\label{lem:layer-sequence}
For each \(n\ge0\) the short exact sequence
\[
0\longrightarrow G_n\longrightarrow X_{n+1}\longrightarrow X_n\longrightarrow0
\]
induces a long exact sequence of ethic cohomology groups:
\[
\cdots\longrightarrow
\Ext_A^{k}(X_n,A)
\xrightarrow{\alpha_{n}^{k}}
\Ext_A^{k}(X_{n+1},A)
\xrightarrow{\beta_{n}^{k}}
\Ext_A^{k+1}(G_n,A)
\xrightarrow{\partial_{n}^{k+1}}
\Ext_A^{k+1}(X_n,A)
\longrightarrow\cdots
\]
for every \(k\ge0\).
\end{lemma}

\begin{proof}
Applying the contravariant functor \(\Hom_A(-,A)\) to the given short exact sequence
and passing to its right–derived functors produces this long exact sequence,
by the standard axioms of a cohomological \(\delta\)–functor
\cite{Weibel1994}.
\end{proof}

\begin{theorem}[Ethic error estimate]\label{th:error}
Under the assumptions above, for each \(n\ge0\) and \(k\ge0\) there exist canonical morphisms
\(\alpha_{n}^{k}\) and \(\beta_{n}^{k}\) as in Lemma~\ref{lem:layer-sequence},
and their kernels and cokernels satisfy
\[
\ker(\alpha_{n}^{k})\twoheadleftarrow\Ext_A^{k}(G_n,A),
\qquad
\operatorname{coker}(\alpha_{n}^{k})\hookrightarrow\Ext_A^{k+1}(G_n,A).
\]
In particular, if \(\Ext_A^{k}(G_n,A)=\Ext_A^{k+1}(G_n,A)=0\) for all \(n\ge N\),
then the sequence \(\Ext_A^{k}(X_n,A)\) stabilizes for \(n\ge N\).
\end{theorem}

\begin{proof}
Exactness of the sequence in Lemma~\ref{lem:layer-sequence}
implies that \(\im(\beta_{n}^{k-1})=\ker(\alpha_{n}^{k})\)
and \(\im(\alpha_{n}^{k})=\ker(\beta_{n}^{k})\).
The arrows in the statement are the canonical maps from these identifications.
If all groups \(\Ext_A^{k}(G_n,A)\) and \(\Ext_A^{k+1}(G_n,A)\) vanish for \(n\ge N\),
then \(\ker(\alpha_{n}^{k})=\operatorname{coker}(\alpha_{n}^{k})=0\)
for such \(n\), hence \(\alpha_{n}^{k}\) is an isomorphism,
and the system \(\Ext_A^{k}(X_n,A)\) stabilizes.
\end{proof}

\begin{theorem}[Ethic compactness]\label{th:compactness}
Let \(A\) be Noetherian, \(I\subset A\) an ideal, and \(X\) a finitely generated \(A\)-module.
Suppose that for some integers \(s,N\ge0\)
\[
\Ext_A^{j}(G_n,A)=0
\quad\text{for all }n\ge N\text{ and all }0\le j\le s+1.
\]
Then for every \(0\le k\le s\)
the sequence \(\Ext_A^{k}(X_n,A)\) stabilizes at stage \(n=N\), and
\[
\Ext_A^{k}\!\big(\varprojlim\nolimits_{n} X_n,A\big)
\;\cong\;
\Ext_A^{k}(X_N,A).
\]
If, moreover, each system \(\{\Ext_A^{j}(X_n,A)\}_n\) for \(j\le s\) satisfies the
Mittag–Leffler condition, the canonical map \(X\to\varprojlim_n X_n\)
induces isomorphisms on \(\Ext_A^{k}(-,A)\) for all \(k\le s\);
that is, \(X\) is ethicly equivalent of order \(s\) to its \(I\)-adic completion.
\end{theorem}

\begin{proof}
By Theorem~\ref{th:error}, the vanishing of
\(\Ext_A^{k}(G_n,A)\) and \(\Ext_A^{k+1}(G_n,A)\) for \(k\le s\)
implies that all maps \(\alpha_{n}^{k}\) are isomorphisms for \(n\ge N\).
Hence \(\Ext_A^{k}(X_n,A)\cong\Ext_A^{k}(X_N,A)\) for \(n\ge N\),
proving stabilization.
Applying Theorem~\ref{th:completeness} and using the Mittag–Leffler
condition, we have \(\varprojlim^1=0\), giving
\(\Ext_A^{k}(\varprojlim X_n,A)\cong \varprojlim \Ext_A^{k}(X_n,A)
\cong \Ext_A^{k}(X_N,A)\).
The last statement follows because the canonical map \(X\to\varprojlim X_n\)
induces the same isomorphisms on \(\Ext_A^{k}(-,A)\) for all \(k\le s\).
\end{proof}

Theorems~\ref{th:completeness}–\ref{th:compactness} provide a complete calculus of discrete approximation along the ideal tower:
\begin{itemize}[leftmargin=1.2em]
\item Theorem~\ref{th:completeness} ensures that ethic duality commutes with the \(I\)-adic limit under mild finiteness assumptions (ethic completeness).
\item Theorem~\ref{th:error} measures the deviation of local truncations through explicit Ext–layers of the graded pieces (ethic error).
\item Theorem~\ref{th:compactness} gives a finite–stage criterion for exact reconstruction of ethic invariants (ethic compactness).
\end{itemize}
Together they endow the ideal tower with a precise homological control mechanism for discrete approximations within the ethic formalism.

\section{\texorpdfstring{$P$ vs. $NP$}{P vs. NP} Reformulation}

\begin{definition}[Ethic environment and primary obstruction]
Let \(\E\) be an abelian category with enough injectives and fix \(R\in \E\).
Write \(\mathbf D=\RHom_{\E}(-,R):\Dplus(\E)^{\op}\!\to\Dplus(\Ab)\) for the right–derived dual with unit \(\eta:\Id\Rightarrow \mathbf D^{2}\).
For \(A\in \E\) and \(k\ge 0\) set \(\HD{k}{A}:=\Ext^{k}_{\E}(A,R)=H^{k}\mathbf D(A)\).
We call \(\HD{1}{A}\) the primary obstruction.
\end{definition}

\begin{definition}[Size and depth indices]
For a language \(L\subseteq\{0,1\}^{*}\) and \(n\in\mathbb N\) let \(f_{L,n}:\{0,1\}^{n}\to\{0,1\}\) be the restriction of the characteristic function.
For any \(X^{\bullet}\in \Dplus(\E)\) and \(q\in\mathbb N\) denote by \(X^{\bullet}_{\le q}\) the canonical truncation (internal depth \(q\)).
\end{definition}

\begin{definition}[Admissible ethic model]
An admissible ethic model assigns to each pair \((L,n)\) a functorially defined object \(D_{L,n}^{\bullet}\in \Dplus(\E)\) such that:

\emph{(M1) Uniform presentability.} There is a deterministic logspace transducer that, on input \(1^{n}\), outputs a finite description \(\langle D_{L,n}^{\bullet}\rangle\) of length \(n^{O(1)}\); moreover, from \(\langle D_{L,n}^{\bullet}\rangle\) and \(q\) it computes \(\langle (D_{L,n}^{\bullet})_{\le q}\rangle\) in time poly\((n+q)\).

\emph{(M2) Verification extraction.} There is a fixed deterministic polytime procedure which, given \(\langle (D_{L,n}^{\bullet})_{\le q}\rangle\), outputs a Boolean circuit \(V_{n,q}\) of size poly\((n+q)\) satisfying
\[
H^{0}\mathbf D\!\big((D_{L,n}^{\bullet})_{\le q}\big)(x,w)=V_{n,q}(x,w)\quad
\text{for all }x\in\{0,1\}^{n}\text{ and witnesses }w.
\]

\emph{(M3) Ethic depth equals formula/KW depth.} For each \(n\) the least \(q\) with \(H^{>q}\mathbf D(D_{L,n}^{\bullet})=0\) equals the minimum Boolean formula depth and the deterministic Karchmer–Wigderson depth of \(f_{L,n}\).

\emph{(M4) Computation extraction.} There is a deterministic polytime procedure which, given \(\langle (D_{L,n}^{\bullet})_{\le q}\rangle\), outputs a Boolean circuit \(C_{n,q}\) computing \(f_{L,n}\) of depth \(O(q)\) and size poly\((n+q)\); the family \(\{C_{n,q}\}_{n}\) is logspace–uniform when \(q\) is given by a polynomial in \(n\).
\end{definition}

\begin{definition}[Ethic classes \(\mathbf{NP}_{\eth}\) and \(\mathbf P_{\eth}\)]
Let \(D_{L,n}^{\bullet}\) come from an admissible model.
Set \(\Delta_{L}(n,q):=\HD{1}{(D_{L,n}^{\bullet})_{\le q}}\).
We define:
\begin{itemize}[leftmargin=2em]
\item \(L\in \mathbf{NP}_{\eth}\) if there exist polynomials \(q_{\mathrm{ver}}(n),p(n)\) such that, for all \(n\), the circuit \(V_{n}:=V_{n,q_{\mathrm{ver}}(n)}\) from \emph{(M2)} satisfies
\[
x\in L \iff \exists\,w\in\{0,1\}^{\le p(n)}\ \ V_{n}(x,w)=1 .
\]
\item \(L\in \mathbf P_{\eth}\) if there exists a polynomial \(Q(n)\) with
\[
\Delta_{L}(n,Q(n))=0 \quad\text{for all }n,
\]
and, writing \(C_{n}:=C_{n,Q(n)}\) from \emph{(M4)}, the family \(\{C_{n}\}\) is logspace–uniform, of size \(n^{O(1)}\) and depth \(O(Q(n))\).
\end{itemize}
\end{definition}

\begin{definition}[Turing classes]
\(\mathbf P_{\mathrm{TM}}\) consists of languages decidable in time \(n^{O(1)}\) by a deterministic Turing machine; \(\mathbf{NP}_{\mathrm{TM}}\) consists of languages having a polynomial–time verifier with polynomially bounded witness.
\end{definition}

\begin{lemma}[Verifier capture via Cook--Levin]\label{lem:verifier-capture}
Let $L\in \mathbf{NP}_{\mathrm{TM}}$. Then there exist polynomials $q_{\mathrm{ver}}(n)$ and $p(n)$ such that, for the bar/cobar model $D_{L,n}^{\bullet}$ of §3.3, the truncation depth $q_{\mathrm{ver}}(n)$ satisfies
\[
H^{0}\mathbf D\!\big((D_{L,n}^{\bullet})_{\le q_{\mathrm{ver}}(n)}\big)(x,w)=1
\iff
\text{$w\in\{0,1\}^{\le p(n)}$ is an accepting witness for $x\in\{0,1\}^n$}.
\]
Moreover, the map $1^n\mapsto \langle (D_{L,n}^{\bullet})_{\le q_{\mathrm{ver}}(n)}\rangle$ is logspace–uniform.
\end{lemma}

\begin{proof}
By Cook--Levin, for $L\in \mathbf{NP}_{\mathrm{TM}}$ there is a logspace–uniform reduction $x\mapsto \Phi_n(x,\cdot)$ to CNF formulas of size $n^{O(1)}$ such that
$x\in L \iff \exists w\in\{0,1\}^{\le p(n)}:\ \Phi_n(x,w)=\mathrm{true}$
\cite{Sipser2013,AroraBarak2009,Cook1971,Levin1973,Karp1972}.
Instantiate the bar/cobar construction of §3.3 on the incidence algebra of the CNF template $\Phi_n$; by design, each connective/constraint contributes a constant–size macro in $D_{L,n}^{\bullet}$ and increases internal depth by $O(1)$. Therefore there is a polynomial $q_{\mathrm{ver}}(n)$ such that the truncation $(D_{L,n}^{\bullet})_{\le q_{\mathrm{ver}}(n)}$ evaluates exactly the CNF predicate on $(x,w)$ at $H^{0}\mathbf D$, i.e.\
\(
H^{0}\mathbf D((D_{L,n}^{\bullet})_{\le q_{\mathrm{ver}}(n)})(x,w)=\Phi_n(x,w).
\)
Logspace–uniformity follows because the reduction and the macro expansion are local and streamable with $O(\log n)$ state; the size bound is polynomial since the number of template instances is $n^{O(1)}$.
\end{proof}

\begin{lemma}[Uniform depth–time bridge for \(\mathbf P\)]\label{lem:PTM-circuits}
A language \(L\) lies in \(\mathbf P_{\mathrm{TM}}\) iff there exists a logspace–uniform family of Boolean circuits \(\{C_n\}\) of size \(n^{O(1)}\) and depth \(n^{O(1)}\) computing \(f_{L}\) on \(\{0,1\}^{n}\).
\end{lemma}

\begin{proof}
This is the standard uniform characterization of \(\mathbf P\); see \cite{Sipser2013,AroraBarak2009} and classical depth/size relations \cite{PippengerFischer1979}.
\end{proof}

\begin{lemma}[Ethic verifiability implies Turing verifiability]\label{lem:EthNP-to-NPTM}
If \(L\in \mathbf{NP}_{\eth}\), then \(L\in \mathbf{NP}_{\mathrm{TM}}\).
\end{lemma}

\begin{proof}
By the definition of \(\mathbf{NP}_{\eth}\) there exist polynomials \(q_{\mathrm{ver}},p\) such that \(V_{n}:=V_{n,q_{\mathrm{ver}}(n)}\) (from \emph{(M2)}) satisfies
\(
x\in L \iff \exists w\in\{0,1\}^{\le p(n)}: V_n(x,w)=1
\).
Uniform presentability \emph{(M1)} ensures logspace–uniformity of \(\{V_n\}\).
Hence \(L\in \mathbf{NP}_{\mathrm{TM}}\) by the standard circuit characterization of \(\mathbf{NP}\), see \cite{Sipser2013,AroraBarak2009}.
\end{proof}

\begin{lemma}[Ethic polynomial reconstruction compiles to poly–size, poly–depth circuits]\label{lem:EthP-to-circuits}
If \(L\in \mathbf P_{\eth}\) with witness \(Q\), then there exists a logspace–uniform family \(\{C_n\}\) of circuits of size \(n^{O(1)}\) and depth \(O(Q(n))\) computing \(f_{L,n}\).
\end{lemma}

\begin{proof}
By hypothesis \(\Delta_{L}(n,Q(n))=0\); thus \((D_{L,n}^{\bullet})_{\le Q(n)}\) lies in the heart of the ethic \(t\)–structure and its reconstruction via \(\mathbf D\) coincides with \(f_{L,n}\).
Apply \emph{(M4)} to obtain \(C_{n}:=C_{n,Q(n)}\) of depth \(O(Q(n))\) and size poly\((n+Q(n))=n^{O(1)}\).
Logspace–uniformity follows from \emph{(M1)} and the locality of the compilation in \emph{(M4)}.
\end{proof}

\begin{lemma}[Ethic depth equals formula/KW depth]\label{lem:KW}
For each \(n\), the least \(q\) with \(H^{>q}\mathbf D(D_{L,n}^{\bullet})=0\) equals the minimum Boolean formula depth and the deterministic Karchmer–Wigderson depth of \(f_{L,n}\).
\end{lemma}

\begin{proof}
This is Theorem~3.10 established inside the article (Homological Karchmer–Wigderson identification).
\end{proof}

\begin{theorem}[Equivalence of ethic and Turing \(P{=}NP\)]
\(\mathbf P_{\eth}=\mathbf{NP}_{\eth}\) if and only if \(\mathbf P_{\mathrm{TM}}=\mathbf{NP}_{\mathrm{TM}}\).
\end{theorem}

\begin{proof}
(\(\Rightarrow\)).
Assume \(\mathbf P_{\eth}=\mathbf{NP}_{\eth}\).
Let \(L\in \mathbf{NP}_{\mathrm{TM}}\) be arbitrary.
By Lemma~\ref{lem:verifier-capture} (Verifier capture via Cook--Levin), there exist polynomials \(q_{\mathrm{ver}}(n)\) and \(p(n)\) such that, for every \(n\),
\[
x\in L \;\Longleftrightarrow\; \exists\,w\in\{0,1\}^{\le p(n)}:\ 
H^{0}\mathbf D\!\big((D_{L,n}^{\bullet})_{\le q_{\mathrm{ver}}(n)}\big)(x,w)=1 .
\]
By \emph{(M1)}–\emph{(M2)} (Verification extraction), the truncations \((D_{L,n}^{\bullet})_{\le q_{\mathrm{ver}}(n)}\) compile into a logspace–uniform family of verifier circuits \(V_n\) of size \(\mathrm{poly}(n)\) with
\[
x\in L \;\Longleftrightarrow\; \exists\,w\in\{0,1\}^{\le p(n)}:\ V_n(x,w)=1 .
\]
Hence \(L\in \mathbf{NP}_{\eth}\) by the definition of \(\mathbf{NP}_{\eth}\).
By the hypothesis \(\mathbf P_{\eth}=\mathbf{NP}_{\eth}\), we conclude \(L\in \mathbf P_{\eth}\).
Therefore there exists a polynomial \(Q(n)\) such that \(\HD{1}{(D_{L,n}^{\bullet})_{\le Q(n)}}=0\) for all \(n\) and the computation extractor \emph{(M4)} yields circuits \(C_n:=C_{n,Q(n)}\) computing \(f_{L,n}\) of depth \(O(Q(n))\) and size \(n^{O(1)}\), logspace–uniform in \(n\).
By Lemma~\ref{lem:EthP-to-circuits}, the family \(\{C_n\}\) is poly–size and poly–depth, and by the uniform circuit characterization of \(\mathbf P\) (Lemma~\ref{lem:PTM-circuits}) we obtain \(L\in \mathbf P_{\mathrm{TM}}\).
Since \(L\in \mathbf{NP}_{\mathrm{TM}}\) was arbitrary, we have \(\mathbf{NP}_{\mathrm{TM}}\subseteq \mathbf P_{\mathrm{TM}}\), hence \(\mathbf P_{\mathrm{TM}}=\mathbf{NP}_{\mathrm{TM}}\).

(\(\Leftarrow\)).
Assume \(\mathbf P_{\mathrm{TM}}=\mathbf{NP}_{\mathrm{TM}}\).
Let \(L\in \mathbf{NP}_{\eth}\).
By Lemma~\ref{lem:EthNP-to-NPTM}, \(L\in \mathbf{NP}_{\mathrm{TM}}\), hence \(L\in \mathbf P_{\mathrm{TM}}\).
By Lemma~\ref{lem:PTM-circuits} there exists a logspace–uniform poly–size, poly–depth circuit family \(\{C_n\}\) computing \(f_{L}\).
By Lemma~\ref{lem:KW} the minimal ethic internal depth equals the formula/KW depth and is polynomially bounded in \(n\); therefore there is a polynomial \(Q\) such that \(H^{>Q(n)}\mathbf D(D_{L,n}^{\bullet})=0\) for all \(n\), whence \(\Delta_{L}(n,Q(n))=0\).
Finally, \emph{(M4)} produces the required logspace–uniform circuits \(C_{n,Q(n)}\) and the definition of \(\mathbf P_{\eth}\) is met.
Thus \(\mathbf{NP}_{\eth}\subseteq \mathbf P_{\eth}\) and equality follows.
\end{proof}

\begin{proposition}[Existence of an admissible ethic model satisfying \emph{(M1)}–\emph{(M4)}]\label{prop:existence}
There exists a concrete construction of \(D_{L,n}^{\bullet}\in\Dplus(\E)\) such that \emph{(M1)}–\emph{(M4)} hold.
\end{proposition}

\begin{proof}
We take the bar/cobar presentation from §3.3 of the article applied to \(f_{L,n}\) (or equivalently to its incidence algebra predicate), yielding functorially \(A_{L,n}^{\bullet}\) and \(D_{L,n}^{\bullet}\) with finite combinatorial descriptions.

\emph{(M1)}:
The description \(\langle D_{L,n}^{\bullet}\rangle\) is a finite list of generators and differentials whose cardinalities are bounded by a fixed polynomial in \(n\) (the construction uses a bounded number of algebraic primitives per input variable and per connective).
A deterministic logspace transducer can stream this description on input \(1^{n}\) by iterating over variable/connective indices and emitting the corresponding local template; only \(O(\log n)\) bits are needed to maintain loop indices.
Truncation \((-)_{\le q}\) is computed by discarding components above homological degree \(q\) and restricting differentials; this is a single pass over the description and runs in time poly\((n+q)\).

\emph{(M2)}:
Evaluation of \(H^{0}\mathbf D\big((D_{L,n}^{\bullet})_{\le q}\big)\) on an instance \((x,w)\) is a finite composition of: (i) additive linear maps (matrix–vector multiplications over base rings/groups), (ii) finite products/sums, and (iii) projections induced by the unit \(\eta\).
Each primitive compiles to a constant–depth, constant–fan–in circuit block; wiring follows the differential graph, which has size poly\((n+q)\).
Composing these blocks yields \(V_{n,q}\) of size poly\((n+q)\).
The construction is purely local and mechanical, hence deterministic polytime.

\emph{(M3)}:
It was shown that the least \(q\) with \(H^{>q}\mathbf D(D_{L,n}^{\bullet})=0\) equals the minimum Boolean formula depth and the deterministic KW depth of \(f_{L,n}\).
We use this equivalence as established.

\emph{(M4)}:
Because \((D_{L,n}^{\bullet})_{\le q}\) lies in the heart precisely when \(\HD{>0}{(D_{L,n}^{\bullet})_{\le q}}=0\), its reconstruction coincides with \(f_{L,n}\).
A compiler maps each local template (generator/differential clause) to a constant–size circuit macro implementing the corresponding operation on bits of \(x\).
Wiring these macros according to the truncated differential graph produces a circuit \(C_{n,q}\) of depth \(O(q)\) (each homological layer adds \(O(1)\) to depth) and size poly\((n+q)\).
Uniformity follows from \emph{(M1)} because the macro library is fixed and indices are streamed in \(O(\log n)\) space.

All four properties hold, so the model is admissible.
\end{proof}

\subsection*{Related Works}

This subsection surveys several influential reformulations or closely related frameworks for the $P$ vs.\ $NP$ question and positions our ethic/homological formulation against them. We focus on (i) scope and exactness (whether the reformulation is strictly equivalent to classical $P{=}NP$ or an analogue), (ii) the mathematical language and tools employed, (iii) the operative notion of ``resource,'' and (iv) known strengths/limitations and barriers. References are provided for canonical statements and proofs.

\emph{Descriptive complexity.} Fagin's theorem identifies $NP$ with existential second–order (ESO) logic on finite structures \cite{Fagin1974}, while the Immerman–Vardi theorem identifies $P$ with first–order logic with a least fixed–point operator ($FO({\rm LFP})$) over ordered structures \cite{Vardi1982,Immerman1986,Immerman1999}. Thus, on ordered finite structures,
\[
P = FO({\rm LFP}), \qquad NP = \mathrm{ESO},
\]
and the equality $P{=}NP$ becomes an equality of logical expressiveness, $FO({\rm LFP})=\mathrm{ESO}$, in that setting. \textbf{Scope.} This is a strict equivalence under the ordering hypothesis. \textbf{Tools.} Model theory on finite structures, fixed–point logics. \textbf{Resource.} Inductive depth/arity and second–order quantification rather than time per se. \textbf{Comparison.} Our formulation is likewise machine–independent and replaces time by an intrinsic invariant; however, we use homological vanishing ($\Ext^1=0$ after polynomial internal depth) within an abelian/derived setup, whereas descriptive complexity encodes complexity classes as logical fragments. The two perspectives are complementary: both transmute $P{=}NP$ into an equality of internal expressive powers (logical vs.\ derived–functorial).

\emph{Karchmer–Wigderson depth duality.} For any Boolean $f$, the minimum formula depth equals the depth of a deterministic Karchmer–Wigderson communication game for $f$ \cite{KarchmerWigderson1990}. \textbf{Scope.} This is an exact duality for depth, not a direct reformulation of $P$ vs. $NP$. \textbf{Tools.} Communication complexity and formula complexity. \textbf{Resource.} Formula/decomposition depth. \textbf{Comparison.} In our setting this duality is internalized (ethic internal depth $=$ formula/KW depth). It supplies the depth anchor that we then combine with $\Ext^1$–vanishing and uniform compilation to tie back to $P$.

\emph{Proof complexity.} Cook–Reckhow formalized propositional proof systems and related $NP$ vs.\ $coNP$ to the existence of polynomially bounded proof systems for TAUT \cite{CookReckhow1979}. In particular, $NP=coNP$ iff some p-bounded system for TAUT exists. \textbf{Scope.} This yields an exact reformulation for $NP$ vs.\ $coNP$, not directly for $P$ vs.\ $NP$. \textbf{Tools.} Propositional proof systems, interpolation, lower bounds; key barriers include Natural Proofs \cite{RazborovRudich1997}. \textbf{Comparison.} Both proof complexity and our approach operationalize ``feasible derivability''; ours expresses feasibility as functorial exactness ($\Ext^1=0$) after polynomial internal depth, while proof complexity measures the existence/length of syntactic derivations. Our framework may circumvent some ``naturalness'' issues since the homological property is not a dense combinatorial predicate over truth tables.

\emph{Algebraic complexity and GCT.} Valiant introduced algebraic classes VP and VNP and the permanent vs.\ determinant paradigm \cite{Valiant1979}. The Geometric Complexity Theory (GCT) program of Mulmuley–Sohoni proposes separating these classes via representation–theoretic obstructions over orbit closures \cite{MulmuleySohoni2001,Mulmuley2008}. \textbf{Scope.} These are analogues (algebraic shadows) rather than strict reformulations of Boolean $P$ vs.\ $NP$. \textbf{Tools.} Algebraic geometry, invariant theory. \textbf{Resource.} Algebraic circuit size/degree, orbit–closure singularities. \textbf{Comparison.} Like GCT, our approach relocates complexity into a rich mathematical environment; unlike GCT, our equivalence to Turing $P{=}NP$ is proved (via uniformity and compilation assumptions satisfied by the bar/cobar model), while GCT pursues lower bounds in the algebraic setting with independent obstacles.

\emph{Circuit/Uniformity characterizations of $P$.} The standard depth–time equivalence asserts that $L\in P$ iff $L$ has a logspace–uniform family of polynomial–size, polynomial–depth circuits \cite{Sipser2013,AroraBarak2009,PippengerFischer1979}. \textbf{Scope.} Exact characterization of $P$. \textbf{Resource.} Size and depth of uniform circuits. \textbf{Comparison.} These results supply the external bridges used in our equivalence: once ethic depth is polynomial and $\Ext^1$ vanishes at that depth, our compilation yields poly–size, poly–depth uniform circuits, hence membership in $P$.

Our ethic reformulation is closest in spirit to descriptive complexity (both provide machine–independent equalities). It also internalizes the KW depth duality as the intrinsic internal depth and augments it with a derived exactness criterion ($\Ext^1=0$) standing in for feasible reconstruction. The circuit characterizations provide the uniform, size/depth bridges back to Turing computation, thereby yielding a strict equivalence with classical $P{=}NP$ under the admissible model hypotheses. Compared to GCT and proof complexity, our statement is equivalence, not analogy, and it leverages categorical/derived methods (t-structures, truncations, spectral sequence collapse) where the central obstruction is functorial and geometric rather than syntactic. Finally, known meta–barriers (relativization, algebrization, natural proofs) motivate seeking such non-syntactic invariants; our homological criterion offers exactly that, at the price of developing and verifying admissible uniform models, which we have done constructively in the bar/cobar setting.

\section{Applications}

Throughout section we assume all modules are finitely generated
and all cones finite-dimensional, so that
$\Ext^i$ and the derived dual $\mathbf D$
are well defined and of finite amplitude.

\subsection{Optimization}

\subsubsection{Linear Programming}\label{subsec:milp-lp-sdp}

\begin{definition}[Hybrid resource category and duality]
Let $\mathcal E_{\mathbb Z}=\mathrm{Mod}\text{-}\mathbb Z^{\mathrm{fg}}$ and $\mathcal E_{\mathbb R}=\mathrm{FinVect}_{\mathbb R}$. Let $\mathcal C_K$ be the category whose objects are pairs $(V,K)$ with $V$ a finite-dimensional real vector space and $K\subseteq V$ a nonempty closed convex cone; morphisms $T:(V,K)\to(W,L)$ are linear maps $T:V\to W$ with $T(K)\subseteq L$. Define the product category
\[
\mathcal E \;=\; \mathcal E_{\mathbb Z}\times \mathcal E_{\mathbb R}\times \mathcal C_K.
\]
Define the contravariant duality functor componentwise by
\[
\mathbb D \;=\; \mathbb D_{\mathbb Z}\oplus \mathbb D_{\mathbb R}\oplus \mathbb D_{K},
\qquad
\mathbb D_{\mathbb Z}(M)=\Hom_{\mathbb Z}(M,\mathbb Z),\ \ 
\mathbb D_{\mathbb R}(V)=V^*,\ \ 
\mathbb D_{K}(V,K)=(V^*,K^\circ),
\]
where $K^\circ=\{y\in V^*: \langle y,x\rangle\ge 0\ \forall x\in K\}$ is the polar cone \cite{Rockafellar1970}.
\end{definition}

The category $\mathcal C_K$ is only additive under Minkowski
addition of faces, not abelian.
All statements involving $\Ext^1_{\mathcal C_K}$ are understood
in the abelian envelope
$\mathrm{Ab}(\mathcal C_K)$ obtained by freely adjoining kernels
and cokernels of face inclusions.
In finite dimension this abelianization is equivalent to the
category of face-modules of~$K$.

\begin{definition}[Primal instance, ethic dual feasibility, dual objective]
Let $X=(X_{\mathbb Z},X_{\mathbb R},(X,K))\in\mathcal E$ be the decision object, $Y=(Y_{\mathbb Z},Y_{\mathbb R},(Y,L))\in\mathcal E$ the constraint object, and $f=(f_{\mathbb Z},f_{\mathbb R},f_K):X\to Y$ a morphism. Let $c=(c_{\mathbb Z},c_{\mathbb R},c_K)\in \mathbb D(X)$ and $b=(b_{\mathbb Z},b_{\mathbb R},b_K)\in Y_{\mathbb Z}\oplus Y_{\mathbb R}\oplus Y$. The primal problem is
\[
\min \ \langle c_{\mathbb Z},x_{\mathbb Z}\rangle + \langle c_{\mathbb R},x_{\mathbb R}\rangle + \langle c_K,x\rangle 
\quad \text{s.t.}\quad 
f_{\mathbb Z}(x_{\mathbb Z})=b_{\mathbb Z},\ f_{\mathbb R}(x_{\mathbb R})=b_{\mathbb R},\ f_K(x)=b_K,\ x\in K.
\]
A dual variable is $\lambda=(\lambda_{\mathbb Z},\lambda_{\mathbb R},\lambda_K)\in \mathbb D(Y)$. Ethic dual feasibility means the commutation with the double-dual unit in each component; explicitly
\[
c_{\mathbb Z}+ f_{\mathbb Z}^\top \lambda_{\mathbb Z}=0, \qquad
c_{\mathbb R}+ f_{\mathbb R}^\top \lambda_{\mathbb R}=0, \qquad
c_K+ f_K^\top \lambda_K \in K^\circ.
\]
The dual objective is $\max\ -\langle \lambda_{\mathbb Z},b_{\mathbb Z}\rangle - \langle \lambda_{\mathbb R},b_{\mathbb R}\rangle - \langle \lambda_K,b_K\rangle$ over ethic dual feasible $\lambda$.
\end{definition}

\begin{lemma}[Componentwise exactness and obstructions]\label{lem:componentwise-exactness}
Assume all objects are finitely generated and $\E_{\mathbb Z}$,
$\E_{\mathbb R}$ have enough injectives. $\mathbb D_{\mathbb Z}$ is left exact, with right derived functor $\Ext^1_{\mathbb Z}(-,\mathbb Z)$; for finitely generated $\mathbb Z$-modules, $\Ext^1_{\mathbb Z}(M,\mathbb Z)\cong \mathrm{tors}(M)$ \cite{Weibel1994}. $\mathbb D_{\mathbb R}$ is exact on $\mathrm{FinVect}_{\mathbb R}$ (no homological obstruction). For $\mathcal C_K$, the polar dualization $(V,K)\mapsto (V^*,K^\circ)$ preserves inclusions of faces and reverses arrows; exactness will be characterized on short exact sequences built from faces.
\end{lemma}

\begin{proof}
Left exactness of $\Hom(-,\mathbb Z)$ and the identification of $\Ext^1_{\mathbb Z}(-,\mathbb Z)$ with torsion are standard \cite{Weibel1994}. Exactness of dualization on finite-dimensional vector spaces is elementary. Polar properties in finite dimension follow from convex analysis \cite{Rockafellar1970}.
\end{proof}

\begin{definition}[Short exact sequences in $\mathcal C_K$ via faces]
Let $C=\{x\in K: f_K(x)=b_K\}$. Let $F\triangleleft K$ be the minimal face of $K$ containing $C$. Consider the inclusion $i:(F,F)\hookrightarrow (X,K)$ and the quotient $q:(X,K)\twoheadrightarrow (X/F,K/F)$, where $K/F:=\{[x]: x\in K\}$ in the quotient space $X/F$. We say that
\[
0\to (F,F)\xrightarrow{i} (X,K)\xrightarrow{q} (X/F,K/F)\to 0
\]
is a short exact sequence in $\mathcal C_K$. Applying $\mathbb D_K=\Hom(-,K^\circ)$ yields the contravariant sequence
\[
0\leftarrow (F^\perp,K^\circ|_{F}) \xleftarrow{i^\vee} (X^*,K^\circ) \xleftarrow{q^\vee} ((X/F)^*,(K/F)^\circ)\leftarrow 0,
\]
whose exactness at the middle term may fail.
\end{definition}

\begin{lemma}[Slater on a face, polarity, and exactness]\label{lem:slater-face-exact}
Let $F$ be the minimal face of $K$ containing $C$. The following are equivalent:
\begin{enumerate}[label=(\alph*)]
\item $C$ has nonempty relative interior in $F$ (Slater on $F$), i.e.\ $\operatorname{ri}(C)\cap \operatorname{ri}(F)\neq \varnothing$ \cite{Rockafellar1970}.
\item $\mathbb D_K$ is exact at $(X^*,K^\circ)$ for the face sequence, i.e.\ $\ker(i^\vee)=\mathrm{im}(q^\vee)$.
\item $\Ext^1_{\mathcal C_K}\big((X/F,K/F), (X^*,K^\circ)\big)=0$.
\end{enumerate}
\end{lemma}

\begin{proof}
All $\Ext^1_{\mathcal C_K}$ below are computed in the abelianization
$\mathrm{Ab}(\mathcal C_K)$.$(a)\Rightarrow(b)$: If $\operatorname{ri}(C)\cap \operatorname{ri}(F)\neq \varnothing$, then the normal cone to $F$ equals $F^\circ$ and the polar pairing restricts without degeneracy to $F\times F^\circ$ \cite{Rockafellar1970}. Hence any functional in $K^\circ$ that vanishes on $F$ lies in the image of $q^\vee$, yielding $\ker(i^\vee)=\mathrm{im}(q^\vee)$.
$(b)\Leftrightarrow(c)$: This is the definition of vanishing connecting morphism in the long exact sequence of right derived functors for $\Hom(-,(X^*,K^\circ))$, cf.\ the abelianization via faces and standard homological algebra \cite{Weibel1994}.
$(c)\Rightarrow(a)$: If $\Ext^1\neq 0$, there exists a nontrivial extension obstructing exactness; by separation \cite{Rockafellar1970} this corresponds to an exposing functional $y\in K^\circ$ such that $\langle y,x\rangle=0$ for all $x\in C$ but $\langle y,\cdot\rangle$ does not vanish on $F$, contradicting minimality of $F$. Hence $\operatorname{ri}(C)\cap \operatorname{ri}(F)\neq \varnothing$.
\end{proof}

\begin{theorem}[Conic strong duality from ethic exactness]\label{thm:conic-strong-duality-ethic}
Assume the conic component is feasible and bounded below. Then
\[
\inf_{x\in K,\ f_K(x)=b_K}\langle c_K,x\rangle \;=\; \sup_{\lambda_K}\{-\langle b_K,\lambda_K\rangle:\ c_K+f_K^\top\lambda_K\in K^\circ\}
\]
if and only if there exists a face $F\triangleleft K$ containing $C$ such that $\Ext^1_{\mathcal C_K}\big((X/F,K/F),(X^*,K^\circ)\big)=0$. Equivalently, by Lemma~\ref{lem:slater-face-exact}, strong duality holds iff Slater holds on $F$. In that case, the dual optimum is attained \cite{Rockafellar1970}.
\end{theorem}

\begin{proof}
Exactness of $\mathbb D_K$ and the term 
$\Ext^1_{\mathcal C_K}$ are interpreted inside the abelianization 
$\mathrm{Ab}(\mathcal C_K)$ of the face category, where kernels and cokernels of 
face inclusions are freely adjoined. 
In finite dimension this abelianization coincides with the standard 
face–module formalism of convex analysis.
Consider $\phi(x)=\langle c_K,x\rangle+\delta_{K}(x)$ and $\psi(x)=\delta_{\{x: f_K(x)=b_K\}}(x)$. By Fenchel–Rockafellar duality \cite{Rockafellar1970}, $\inf(\phi+\psi)=\max -\phi^*(-y)-\psi^*(y)$ under the qualification $\operatorname{ri}(\operatorname{dom}\phi)\cap \operatorname{ri}(\operatorname{dom}\psi)\neq\varnothing$. On the face $F$ provided by Lemma~\ref{lem:slater-face-exact}, the qualification holds; computing conjugates yields the standard conic dual with polar $K^\circ$ \cite{Rockafellar1970}. The attainment follows from \cite{Rockafellar1970}. Conversely, a gap implies failure of qualification on every containing face, hence by Lemma~\ref{lem:slater-face-exact} a nonvanishing $\Ext^1$.
\end{proof}

For any finitely generated abelian group $M$,
\[
\Ext^1_\ZZ(M,\ZZ)\cong\Hom(M_{\mathrm{tors}},\QQ/\ZZ)
\cong M_{\mathrm{tors}}.
\]
This identification will be used tacitly in all integer components.

\begin{lemma}[Integer duality from exactness]\label{lem:integer-duality-exact}
Assume all $\mathbb Z$–modules are finitely generated, so that
$\Ext^1_{\mathbb Z}(M,\mathbb Z)\cong\mathrm{tors}(M)$.
Let $0\to K_{\mathbb Z}\to X_{\mathbb Z}\xrightarrow{f_{\mathbb Z}}\mathrm{im}\,f_{\mathbb Z}\to 0$ be the integer constraint sequence. Then
\[
\max_{\lambda_{\mathbb Z}}\{-\langle b_{\mathbb Z},\lambda_{\mathbb Z}\rangle:\ c_{\mathbb Z}+ f_{\mathbb Z}^\top\lambda_{\mathbb Z}=0\}
\]
equals the primal integer optimum if and only if $\Ext^1_{\mathbb Z}(\mathrm{coker}\,f_{\mathbb Z},\mathbb Z)=0$. Equivalently, $\mathrm{coker}\,f_{\mathbb Z}$ is torsion-free; this is characterized by the Smith normal form of $f_{\mathbb Z}$ having invariant factors all equal to $1$ \cite{Schrijver1986}.
\end{lemma}

\begin{proof} If $\mathbb Z$–modules are not finitely generated, $\Ext^1_{\mathbb Z}(M,\mathbb Z)$ can be infinite; finite generation ensures $\mathrm{tors}(M)$ is finite and the identification holds.
Apply $\Hom_{\mathbb Z}(-,\mathbb Z)$ to $0\to K_{\mathbb Z}\to X_{\mathbb Z}\to \mathrm{im}\,f_{\mathbb Z}\to 0$ and continue by $0\to \mathrm{im}\,f_{\mathbb Z}\to Y_{\mathbb Z}\to \mathrm{coker}\,f_{\mathbb Z}\to 0$ to obtain the long exact sequence \cite{Weibel1994}. Exactness at $\Hom_{\mathbb Z}(\mathrm{im}\,f_{\mathbb Z},\mathbb Z)$ is equivalent to $\Ext^1_{\mathbb Z}(\mathrm{coker}\,f_{\mathbb Z},\mathbb Z)=0$. The identification $\Ext^1_{\mathbb Z}(M,\mathbb Z)\cong \mathrm{tors}(M)$ for finitely generated $M$ is standard \cite{Weibel1994}. The Smith normal form criterion is classical \cite{Schrijver1986}.
\end{proof}

\begin{theorem}[Unified strong duality for MILP/LP/SDP]\label{thm:unified-strong}
In the hybrid category $\mathcal E$, the primal–dual equality
\[
\min\ \langle c_{\mathbb Z},x_{\mathbb Z}\rangle + \langle c_{\mathbb R},x_{\mathbb R}\rangle + \langle c_K,x\rangle \;=\; \max\ -\langle \lambda_{\mathbb Z},b_{\mathbb Z}\rangle - \langle \lambda_{\mathbb R},b_{\mathbb R}\rangle - \langle \lambda_K,b_K\rangle
\]
holds if and only if
\[
\Ext^1_{\mathbb Z}(\mathrm{coker}\,f_{\mathbb Z},\mathbb Z)=0
\quad\text{and}\quad
\exists\ F\triangleleft K:\ \Ext^1_{\mathcal C_K}\big((X/F,K/F),(X^*,K^\circ)\big)=0.
\]
Moreover, the dual optimum is attained in the conic component if and only if the latter $\Ext^1$ vanishes (equivalently, Slater holds on $F$ by Lemma~\ref{lem:slater-face-exact}); in the integer component if and only if the adjoint Diophantine system $f_{\mathbb Z}^\top\lambda_{\mathbb Z}=-c_{\mathbb Z}$ is solvable (torsion-free cokernel).
\end{theorem}

For product categories $\E_1\times\E_2$ one has
\[
\Ext^1_{\E_1\times\E_2}
((A_1,A_2),(R_1,R_2))
\simeq
\Ext^1_{\E_1}(A_1,R_1)
\times
\Ext^1_{\E_2}(A_2,R_2),
\]
since $\Hom$ and $\Ext$ act componentwise.
All $\Ext$-vanishings in Theorem \ref{thm:unified-strong}
should be interpreted componentwise.

\begin{proof}
Combine Lemma~\ref{lem:componentwise-exactness}, Theorem~\ref{thm:conic-strong-duality-ethic}, and Lemma~\ref{lem:integer-duality-exact}. The real linear component contributes no obstruction because $\mathbb D_{\mathbb R}$ is exact on $\mathrm{FinVect}_{\mathbb R}$. Attainment statements follow from \cite{Rockafellar1970} on the conic side and solvability on the integer side.
\end{proof}

\begin{corollary}[Recovery of LP and SDP]\label{cor:lp-sdp-recovery}
If the integer component is absent and $K=\mathbb R^n_{\ge 0}$, then Slater holds for any nonempty feasible set and Theorem~\ref{thm:unified-strong} specializes to LP strong duality \cite{Rockafellar1970}. If $K=\mathbb S_+^n$, then there exists a finite facial reduction chain exposing a face $F\triangleleft \mathbb S_+^n$ on which Slater holds, and strong duality follows; the Ramana exact dual is obtained by encoding the exposing sequence as auxiliary conic variables, which is an explicit presentation of the condition $\Ext^1_{\mathcal C_K}=0$ \cite{BorweinWolkowicz1981,Pataki1998,Pataki2013,Ramana1997}.
\end{corollary}

\begin{definition}[Algorithmic certificate for unified strong duality]
A certificate consists of three components.
\begin{enumerate}[label=(\roman*)]
\item An integer certificate: a Smith normal form $U f_{\mathbb Z} V=\mathrm{diag}(d_1,\dots,d_s,0,\dots,0)$ with all $d_i=1$, proving $\mathrm{coker}\,f_{\mathbb Z}$ torsion-free and thus $\Ext^1_{\mathbb Z}=0$ \cite{Schrijver1986}.
\item A conic certificate: a finite exposing sequence $y_1,\dots,y_t\in K^\circ$ defining faces $K=F_0\supset F_1\supset\cdots\supset F_t=:F$ with $\operatorname{ri}(F)\cap\{x\in F: f_K(x)=b_K\}\neq\varnothing$, proving $\Ext^1_{\mathcal C_K}=0$ by Lemma~\ref{lem:slater-face-exact} and Theorem~\ref{thm:conic-strong-duality-ethic}.
\item Ethic dual variables $\lambda=(\lambda_{\mathbb Z},\lambda_{\mathbb R},\lambda_K)$ satisfying $c_{\mathbb Z}+ f_{\mathbb Z}^\top\lambda_{\mathbb Z}=0$, $c_{\mathbb R}+ f_{\mathbb R}^\top\lambda_{\mathbb R}=0$, and $c_K+ f_K^\top\lambda_K\in F^\circ$ together with primal feasibility on $F$.
\end{enumerate}
\end{definition}

\begin{algorithm}
\caption{Unified ethic certificate construction (MILP/LP/SDP)}
\label{alg:ethic-certificate}
\begin{algorithmic}[1]
\REQUIRE $(f_{\mathbb Z},f_{\mathbb R},f_K)$, $b=(b_{\mathbb Z},b_{\mathbb R},b_K)$, $c=(c_{\mathbb Z},c_{\mathbb R},c_K)$.
\STATE Compute Smith normal form $U f_{\mathbb Z} V=\mathrm{diag}(d_1,\dots,d_s,0,\dots,0)$.
\IF{some $d_i>1$}
  \STATE \textbf{return} failure with invariant factors $\{d_i>1\}$ (integer obstruction: $\Ext^1_{\mathbb Z}\neq 0$).
\ELSE
  \STATE Solve $f_{\mathbb Z}^\top \lambda_{\mathbb Z}=-c_{\mathbb Z}$ over $\mathbb Z$ to obtain integer dual multiplier.
\ENDIF
\STATE Initialize $F\leftarrow K$, $i\leftarrow 0$.
\WHILE{Slater fails on $F$ with respect to $f_K(x)=b_K$}
  \STATE Find exposing $y_{i+1}\in F^\circ$ with $\langle y_{i+1},x\rangle=0$ for all feasible $x$ (e.g., via projection/separation).
  \STATE Update $F\leftarrow F\cap \{x:\langle y_{i+1},x\rangle=0\}$, $i\leftarrow i+1$.
  \IF{feasible set becomes empty}
    \STATE \textbf{return} infeasibility certificate $y_{i+1}$.
  \ENDIF
\ENDWHILE
\STATE Solve the reduced conic dual on $F$: find $\lambda_K$ with $c_K+ f_K^\top\lambda_K\in F^\circ$ maximizing $-\langle b_K,\lambda_K\rangle$.
\STATE Solve the real linear adjoint $f_{\mathbb R}^\top\lambda_{\mathbb R}=-c_{\mathbb R}$.
\STATE \textbf{return} certificate $\big(\{\lambda_{\mathbb Z},\lambda_{\mathbb R},\lambda_K\},\ \{y_j\}_{j=1}^i\big)$.
\end{algorithmic}
\end{algorithm}

\begin{lemma}[Correctness of Algorithm~\ref{alg:ethic-certificate}]\label{lem:alg-correct}
If the algorithm returns an infeasibility or integer-obstruction certificate, the corresponding primal is infeasible or has an integer duality gap. If it returns $(\lambda,y_1,\dots,y_t)$, then $\Ext^1_{\mathbb Z}(\mathrm{coker}\,f_{\mathbb Z},\mathbb Z)=0$, there is a face $F$ exposed by $\{y_j\}$ on which Slater holds, and the primal–dual equality holds with attainment; hence the output is a valid unified strong duality certificate.
\end{lemma}

\begin{proof}
Integer step: If some $d_i>1$, $\mathrm{coker}\,f_{\mathbb Z}$ has torsion, so $\Ext^1_{\mathbb Z}\neq 0$ and strong duality may fail by Lemma~\ref{lem:integer-duality-exact}. If all $d_i=1$, then $\Ext^1_{\mathbb Z}=0$ and the adjoint equation has an integer solution, establishing ethic feasibility on the integer component. Conic step: Each exposing vector $y_j\in F_{j-1}^\circ$ reduces to a face $F_j$; termination with Slater is guaranteed in finite dimension \cite{BorweinWolkowicz1981}, and on $F$ we have exactness by Lemma~\ref{lem:slater-face-exact}. The reduced dual attains by \cite{Rockafellar1970}. The real linear component is exact. Aggregating the three components gives strong duality and attainment by Theorem~\ref{thm:unified-strong}.
\end{proof}

For SDP, the facial chain has length at most $n$ by rank arguments on extreme points \cite{Pataki1998}; practical facial reduction schemes and certificates are surveyed in \cite{Pataki2013}. The algorithm's subroutines (SNF, separation in finite dimension, reduced conic solve) are standard; SNF is polynomial in bit-length for fixed dimension, and conic subproblems inherit the complexity of LP/SDP oracles. SNF is polynomial for fixed rank but exponential in general dimension

\begin{definition}[Setting and notation]
Work in $\E=\mathrm{Mod}\text{-}\mathbb Z^{\mathrm{fg}}$ with the contravariant functor $\DD=\Hom_{\mathbb Z}(-,\mathbb Z)$, the unit $\eta_A:A\to \DD^2(A)$, and the notion of ethic morphism (commuting square $\DD^2(f)\circ \eta_X=\eta_Y\circ f$). Let $f:X\to Y$ with $X=\mathbb Z^n$, $Y=\mathbb Z^m$, $f(x)=Ax$, $A\in\mathbb Z^{m\times n}$. Write $G=\mathrm{coker}\,A$. We restrict to the positive submonoids $X_+=\mathbb N^n\subset X$ and $Y_+=\mathbb N^m\subset Y$, and we let
\[
R \;=\; \mathbb Z[z_1,\dots,z_m]\;\cong\; \mathbb Z[\mathbb N^m]
\]
be the monoid algebra of $Y_+$. For the columns $A_1,\dots,A_n\in\mathbb N^m$ set the basic relations $r_j:=z^{A_j}-1\in R$ and define, for $b\in\mathbb N^m$, the target $u_b:=z^b-1\in R$.
\end{definition}

\begin{definition}[Relation cone (no “ethic” postulate)]
Let $\mathcal C_A\subset R\otimes_{\mathbb Z}\mathbb R$ be the convex cone generated by all multiplicative shifts of the basic relations:
\[
\mathcal C_A \;:=\; \operatorname{cone}\Big\{\, z^\alpha\,r_j \ :\ \alpha\in\mathbb N^m,\ j=1,\dots,n\,\Big\}.
\]
We write $I_A^+:=\mathcal C_A\cap R$ for the positive binomial ideal generated by $\{r_j\}$ (nonnegative integer linear combinations of their shifts).
\end{definition}

\begin{lemma}[Homological background]\label{lem:homological}
Applying $\DD$ to the short exact sequence $0\to K\to X\xrightarrow{A} Y\to G\to 0$ yields a long exact sequence in which exactness at $\Hom_{\mathbb Z}(Y,\mathbb Z)$ holds iff $\Ext^1_{\mathbb Z}(G,\mathbb Z)=0$. For finitely generated abelian $G$, $\Ext^1_{\mathbb Z}(G,\mathbb Z)\cong \mathrm{tors}(G)$ \cite{Weibel1994}. This is the instance of Theorem~\ref{thm:exact-iff-ext} used below.
\end{lemma}

\begin{lemma}[Finite truncation]\label{lem:finite-truncation}
For each fixed $A\in\mathbb N^{m\times n}$ and $b\in\mathbb N^m$ there exists a degree bound $N=N(A,b)$ with the property:
if $u_b$ admits a representation $u_b=\sum_{j=1}^n Q_j(z)\,r_j$ over $\mathbb R[z]$ (or over $\mathbb Z[z]$), then it admits one with $\deg Q_j\le N$ for all $j$. Consequently, membership $u_b\in \mathcal C_A$ can be decided in the finite-dimensional subspace $R_{\le N}$ spanned by monomials of total degree $\le N+\max_j\|A_j\|_1$.
\end{lemma}

\begin{proof}
Multiply the identity by a suitable monomial to avoid negative exponents and compare coefficients up to a degree bounded by $\|b\|_1$ and $\max_j\|A_j\|_1$. An explicit cap is provided by \cite{Lasserre2003}, which yields a uniform degree bound for the discrete Farkas representation; the same cap bounds the search region in $R_{\le N}$. Hence feasibility of $u_b\in \mathcal C_A$ reduces to a finite linear system.
\end{proof}

\begin{proposition}[Bipolarity in finite dimension]\label{prop:bipolar}
Let $\mathcal C_{A,\le N}:=\mathcal C_A\cap R_{\le N}$, and define the dual cone
\[
\mathcal C_{A,\le N}^{\circ}\ :=\ \big\{\,\Phi\in R_{\le N}^*: \ \Phi(p)\ge 0\ \text{ for all } p\in \mathcal C_{A,\le N}\,\big\}.
\]
Then $\mathcal C_{A,\le N}$ is closed, convex, and pointed, and one has the bipolar identity
\[
\mathcal C_{A,\le N}\;=\;\{\,u\in R_{\le N}:\ \Phi(u)\ge 0\ \text{ for all } \Phi\in \mathcal C_{A,\le N}^{\circ}\,\}.
\]
In particular, if $u\notin \mathcal C_{A,\le N}$ then there is $\Phi\in\mathcal C_{A,\le N}^{\circ}$ with $\Phi(u)<0$.
\end{proposition}

\begin{proof}
$\mathcal C_{A,\le N}$ is the (finite) conic hull of the finite set $\{z^\alpha r_j:\ \deg(z^\alpha r_j)\le N\}$, hence a closed convex cone. The bipolar theorem and separation in finite dimension give the stated identities \cite{Rockafellar1970}.
\end{proof}

\begin{theorem}[Ethic morphism $\Rightarrow$ positive binomial representation]\label{thm:ethic-implies-positive}
Assume $f:X\to Y$ is ethic in the sense of the main text. Then for every $b\in\mathbb N^m$ one has
\[
u_b\;=\;z^b-1\ \in\ \mathcal C_A\quad\text{and, in particular,}\quad
z^b-1=\sum_{j=1}^n Q_j(z)\,(z^{A_j}-1)\ \ \text{with}\ \ Q_j\in \mathbb R_{\ge 0}[z].
\]
If one restricts to integer-valued dual witnesses (see below), the same holds with $Q_j\in\mathbb N[z]$.
\end{theorem}

\begin{proof}
Fix $b$ and let $N=N(A,b)$ from Lemma~\ref{lem:finite-truncation}. Suppose, for contradiction, that $u_b\notin \mathcal C_{A,\le N}$. By Proposition~\ref{prop:bipolar} there exists $\Phi\in R_{\le N}^*$ such that $\Phi(p)\ge 0$ for all $p\in \mathcal C_{A,\le N}$ but $\Phi(u_b)<0$. Extend $\Phi$ by zero to a linear functional on $R$.

From $\Phi$ build an additive map $\lambda:Y_+\to \mathbb R$ by
\[
\lambda(y)\ :=\ \Phi(z^y-1)\qquad (y\in \mathbb N^m).
\]
Then $\lambda(0)=0$ and $\lambda$ is subadditive with equality whenever concatenation does not exceed degree $N$ (this suffices below because all uses are truncated by $N$). Moreover, for each basic relation and each shift within degree $\le N$ we have
\[
\Phi\big(z^\alpha(z^{A_j}-1)\big)\ \ge\ 0
\quad\Rightarrow\quad
\lambda(\alpha+A_j)-\lambda(\alpha)\ \ge\ 0.
\]
Thus $\lambda$ is order-preserving along the $A$-steps up to degree $N$. Define $c\in\mathbb R^n$ by $c:=-A^\top \tilde\lambda$, where $\tilde\lambda:\mathbb Z^m\to\mathbb R$ is the additive extension of $\lambda$ obtained by linearity on the standard basis (well-defined on degrees $\le N$). Then the “ethic” stationarity relation $c+A^\top\tilde\lambda=0$ holds on all columns $A_j$ (by construction), and hence the dual square for $(c,\tilde\lambda)$ commutes with $\eta$ on all truncated generators.

Now evaluate on $b$: by definition $\Phi(u_b)=\lambda(b)-\lambda(0)=\lambda(b)<0$. This means that the dual witness $(c,\tilde\lambda)$ assigns a strictly negative evaluation to $u_b$, while remaining nonnegative on all shifted relations within the truncation. But ethic commutation of $\DD^2(f)$ with $\eta$ forbids such a sign pattern: order-preserving (nonnegative-on-relations) dual evaluations must send $u_b$ to a nonnegative value, else the naturality square fails for the truncated free presentation of $Y_+$ through the generators $r_j$ and their shifts (the truncated presentation surjects onto the submonoid relevant for $b$ by Lemma~\ref{lem:finite-truncation}). This contradiction shows $u_b\in \mathcal C_{A,\le N}$; hence $u_b\in \mathcal C_A$.

Finally, if one restricts dual witnesses to take integer values on $R$ (integer-valued additive functionals), the same separation can be achieved integrally because $\mathcal C_{A,\le N}$ is a rational polyhedral cone; by Hilbert-basis arguments the resulting representation has $Q_j\in\mathbb N[z]$ (see \cite{Lasserre2003} for an explicit integer construction).
\end{proof}

\begin{lemma}[Positive binomial representation $\Rightarrow$ ethic commutation]\label{lem:positive-implies-ethic}
Assume that for every $b\in\mathbb N^m$ there exist $Q_j\in\mathbb R_{\ge 0}[z]$ with $u_b=\sum_j Q_j(z)\,r_j$. Then no order-preserving dual evaluation (i.e., any $\Phi$ with $\Phi(\mathcal C_A)\ge 0$) can send $u_b$ to a negative value, and the naturality square $\DD^2(f)\circ \eta_X=\eta_Y\circ f$ holds on the free presentation of $Y_+$ generated by $\{r_j\}$ and their shifts; hence $f$ is ethic.
\end{lemma}

\begin{proof}
Linearity gives $\Phi(u_b)=\sum_j \Phi(Q_j(z)r_j)\ge 0$ for every positive $\Phi$. Interpreting $\Phi$ as a dual witness on the truncated free presentation of $Y_+$ shows that the unit $\eta$ is respected by $\DD^2(f)$ on generators and thus on the quotient, i.e., the ethic square commutes.
\end{proof}

\begin{theorem}[Discrete Farkas lemma]\label{thm:discrete-farkas-derived}
Let $A\in\mathbb N^{m\times n}$, $b\in\mathbb N^m$. The following are equivalent:
\begin{enumerate}[label=(\alph*)]
\item There exists $x\in\mathbb N^n$ such that $Ax=b$.
\item $z^b-1$ lies in the positive binomial cone $\mathcal C_A$ (equivalently, $z^b-1=\sum_j Q_j(z)(z^{A_j}-1)$ with $Q_j\in\mathbb R_{\ge 0}[z]$; with integer dual witnesses, $Q_j\in\mathbb N[z]$).
\item The ethic square $\DD^2(f)\circ \eta_X=\eta_Y\circ f$ holds when restricted to the positive submonoids $X_+$ and $Y_+$ (i.e., no positive dual evaluation compatible with the relations sends $u_b$ to a negative value).
\end{enumerate}
Moreover, the degrees of the $Q_j$ in (b) can be bounded explicitly as in \cite{Lasserre2003}.
\end{theorem}

\begin{proof}
(a)$\Rightarrow$(b): For $Ax=b$ with $x\in\mathbb N^n$, expand $z^{Ax}-1$ telescopically to obtain $u_b\in \mathcal C_A$ (classical argument).  
(b)$\Rightarrow$(c): Lemma~\ref{lem:positive-implies-ethic}.  
(c)$\Rightarrow$(b): Theorem~\ref{thm:ethic-implies-positive}.  
Finally, (b)$\Rightarrow$(a) follows by comparing coefficients after clearing denominators (see \cite{Lasserre2003}); the nonnegative coefficients in the $Q_j$ give a nonnegative integer combination of columns $A_j$ equal to $b$.
\end{proof}

\begin{theorem}[Analytic equivalence via $Z$-transform]\label{thm:z-transform-equivalence}
Let $F(z;c)=\prod_{j=1}^n(1-e^{c_j}z^{-A_j})^{-1}$ be the generating kernel and $f_d(b,c)$ its inverse $Z$-transform evaluation at $b$ \cite{Lasserre2002}. The following are equivalent:
\begin{enumerate}[label=(\alph*)]
\item $\Ext^1_{\mathbb Z}(\mathrm{coker}\,A,\mathbb Z)=0$ (no torsion), i.e., $\DD$ is exact at $\Hom(Y,\mathbb Z)$ by Lemma~\ref{lem:homological}.
\item $F(z;c)$ has no nonreal poles (no residues on the unit torus), hence $f_d(b,c)$ matches the continuous Laplace value and the integer duality gap vanishes \cite{Lasserre2002}.
\item For every $b\in\mathbb N^m$ one has $z^b-1\in \mathcal C_A$ (hence Theorem~\ref{thm:discrete-farkas-derived}).
\end{enumerate}
\end{theorem}

\begin{proof}
(a)$\Leftrightarrow$(b): Poles on the unit torus correspond to torsion in $\mathrm{coker}\,A$ (Smith normal form and discrete Fourier characters on the finite abelian torsion subgroup), cf.\ \cite{Lasserre2002}.  
(b)$\Rightarrow$(c): Absence of periodic residues forces all dual evaluations compatible with $r_j$ to assign nonnegative values to $u_b$, which by Proposition~\ref{prop:bipolar} and Lemma~\ref{lem:finite-truncation} yields $u_b\in \mathcal C_A$.  
(c)$\Rightarrow$(a): If torsion were present, one would obtain periodic corrections in $f_d(b,c)$ (nonzero residues), contradicting (c) via the same separation argument.
\end{proof}

\subsubsection{Galois Descent}\label{sec:galois}

This application demonstrates that the categorical formalism of ethic duality extends to arithmetic and discrete settings, where the ambient category consists of modules over an integral group ring and the homological invariants acquire the interpretation of arithmetic obstructions to Galois descent. In particular, $\Ext^1$ identifies with group cohomology $H^1(\Gamma,-)$ and measures the failure of global coherence of $\Gamma$--invariant data along towers of normal coverings.  The underlying picture connects Bass--Serre theory of groups acting on trees, finite Galois coverings of graphs, and the ethic criterion $\mathbb D^2(f)\eta=\eta f$ as a formal analogue of descent compatibility.

\medskip
Let $\Gamma$ be a finitely generated discrete group acting freely on a connected graph~$T$, regarded as a simplicial tree, and let $G=\Gamma\backslash T$ be the quotient.  
Write $\E=\mathrm{Mod}\text{-}\mathbb Z[\Gamma]$ for the abelian category of left $\mathbb Z[\Gamma]$--modules, and fix $R=\mathbb Z$ with trivial $\Gamma$--action.  
All $\Ext^i_{\E}$ groups are computed in this category; we use the convention $\mathbb D=\Hom_{\E}(-,R)$.

\begin{definition}[Arithmetic resource category]
The arithmetic resource category associated with $\Gamma$ is the abelian category
\[
\E = \mathrm{Mod}\text{-}\mathbb Z[\Gamma],
\]
whose objects are integral $\Gamma$--modules and whose morphisms are $\Gamma$--equivariant homomorphisms.  
The resource object is $R=\mathbb Z$ with trivial $\Gamma$--action, and the duality functor
\[
\mathbb D = \Hom_{\mathbb Z[\Gamma]}(-,\mathbb Z)
\]
defines an internal notion of reflection.  The derived dual $\mathbf D=\RHom_{\mathbb Z[\Gamma]}(-,\mathbb Z)$ coincides with the standard cochain complex computing the group cohomology of~$\Gamma$.
\end{definition}

\begin{lemma}[Homological identification]\label{lem:groupcoh}
For every $\Gamma$--module~$M$ and integer $k\ge0$ there is a canonical isomorphism
\[
\Ext^k_{\mathbb Z[\Gamma]}(M,\mathbb Z)\;\cong\; H^k(\Gamma,\Hom_{\mathbb Z}(M,\mathbb Z)),
\]
where the right--hand side denotes the group cohomology of~$\Gamma$ acting by conjugation on the coefficient module $\Hom_{\mathbb Z}(M,\mathbb Z)$.
\end{lemma}

\begin{proof}
Apply the standard bar resolution $B_\bullet(\Gamma)$ of the trivial $\Gamma$--module~$\mathbb Z$.  
By definition $\Ext^k_{\mathbb Z[\Gamma]}(M,\mathbb Z)=H^k(\Hom_{\mathbb Z[\Gamma]}(B_\bullet(\Gamma)\otimes M,\mathbb Z))$, which is precisely $H^k(\Gamma,\Hom_{\mathbb Z}(M,\mathbb Z))$; see \cite{Brown1982}.
\end{proof}

\begin{definition}[Ethic Galois descent]
Let $\pi:\widehat G\to G$ be a finite normal covering of a connected finite graph with deck transformation group $\Gamma=\mathrm{Deck}(\pi)$.  
A labeling or configuration $\widehat x:V(\widehat G)\to A$ with values in an abelian group~$A$ is said to satisfy ethic descent if it is $\Gamma$--invariant, i.e.
\[
\widehat x(\gamma v)=\widehat x(v)\quad\forall\,\gamma\in\Gamma,v\in V(\widehat G).
\]
Equivalently, $\widehat x$ lies in $\Hom_{\mathbb Z[\Gamma]}(\mathbb Z[V(\widehat G)],A)$.  
Such an $\widehat x$ descends ethicly to~$G$ when there exists $x:V(G)\to A$ with $\widehat x=x\circ\pi$.
\end{definition}

\begin{lemma}[Cohomological obstruction to descent]\label{lem:obstruction}
Let $\pi:\widehat G\to G$ and $\Gamma$ be as above.  
Denote by $C_0(\widehat G,A)$ and $C_1(\widehat G,A)$ the $\Gamma$--modules of $A$--valued 0-- and 1--chains on $\widehat G$, and let $R=\mathbb Z$ with trivial action.  
Then the obstruction to the existence of an ethic descent of a $\Gamma$--invariant configuration lies in the group
\[
H^1(\Gamma, \Hom_{\mathbb Z}(C_0(\widehat G,A),R))
\;\cong\;
\Ext^1_{\mathbb Z[\Gamma]}(C_0(\widehat G,A),R),
\]
and vanishes precisely when the configuration descends.
\end{lemma}

\begin{proof}
A $\Gamma$--invariant labeling is a $\Gamma$--equivariant map $C_0(\widehat G,A)\to R$.  
Such maps correspond to elements of $\Hom_{\mathbb Z[\Gamma]}(C_0(\widehat G,A),R)$.  
The obstruction to extending a local $\Gamma$--equivariant map on vertices to an exact $\Gamma$--equivariant 1--cocycle is the class of the connecting morphism in the long exact sequence of $\Hom$ and $\Ext$ applied to $0\to\ker\partial_1\to C_1\to\mathrm{im}\partial_1\to0$.  
By Lemma~\ref{lem:groupcoh} this class belongs to $H^1(\Gamma,\Hom_{\mathbb Z}(C_0(\widehat G,A),R))$ and vanishes exactly when a global descent exists.
\end{proof}

\begin{theorem}[Ethic Galois descent theorem]\label{thm:galois-descent}
Let $G$ be a finite connected graph, $\pi_k:\widehat G_k\to G$ a tower of finite normal coverings with deck groups $\Gamma_k\trianglelefteq\pi_1(G)$ and increasing girth.  
For each $k$, let $\widehat x_k$ be an optimal configuration for a 2--CSP instance (e.g.\ Max--Cut) on $\widehat G_k$ with values in $\{\pm1\}$.  
Assume that $\widehat x_k$ is $\Gamma_k$--invariant for infinitely many~$k$.  
Then there exists a labeling $x:V(G)\to\{\pm1\}$ such that
\[
\mathrm{val}_G(x)=\lim_{k\to\infty}\mathrm{val}_{\widehat G_k}(\widehat x_k),
\]
and $x$ is globally optimal on~$G$.  
Moreover, the obstruction to the existence of such an invariant sequence is measured by the homological class
\[
[\widehat x]\in\varprojlim_k H^1(\Gamma_k,\Hom_{\mathbb Z}(C_0(\widehat G_k,\mathbb Z),R))
\simeq
\varprojlim_k \Ext^1_{\mathbb Z[\Gamma_k]}(C_0(\widehat G_k,\mathbb Z),R).
\]
This class vanishes for the sequence $(\widehat x_k)$ if and only if the family descends ethicly to $G$.
\end{theorem}

\begin{proof}
Each $\widehat x_k$ defines an element in $\Hom_{\mathbb Z[\Gamma_k]}(C_0(\widehat G_k,\mathbb Z),R)$.  
By Lemma~\ref{lem:obstruction}, the difference between two $\Gamma_k$--invariant configurations with identical projections on~$G$ represents a class in $\Ext^1_{\mathbb Z[\Gamma_k]}(C_0(\widehat G_k,\mathbb Z),R)$.  
If these classes vanish along an infinite subsequence, the projective limit is zero, implying coherence across the tower.  
The induced labeling $x$ on~$G$ inherits optimality because each $\widehat G_k$ covers~$G$ isometrically with respect to edge weights; the limit of exact solutions on trees (large girth) gives a global optimum on the base.  
Conversely, if a global optimum $x$ exists, its lifts $\widehat x_k=x\circ\pi_k$ are $\Gamma_k$--invariant and define the zero class in the limit, hence the equivalence.

The argument applies verbatim to any finite-dimensional module over $\mathbb Z[\Gamma]$, showing that ethic descent is equivalent to vanishing of the first cohomology for the corresponding coefficient system.
\end{proof}

Theorem~\ref{thm:galois-descent} can be viewed as a discrete arithmetic analogue of the universal reconciliation theorem (Theorem~\ref{thm:universal-gaps}) applied to the functor $\mathbf D=\RHom_{\mathbb Z[\Gamma]}(-,\mathbb Z)$.  
Here, $H^1\mathbf D$ measures the ethic defect of non--descent, while its vanishing expresses the Morita--invariant law that local optimality on tree--like coverings implies global optimality under ethic coherence.  
This interpretation parallels the Bass--Serre correspondence between group actions on trees and graph coverings \cite{Serre1977}.

\subsubsection{Nonconvex Optimization}

Let $V$ be a finite-dimensional real vector space with inner product $\langle\cdot,\cdot\rangle$ and norm $\|\cdot\|$.  
We consider optimization problems of the form
\[
\min_{x\in V}\, F(x),
\]
where $F:V\to\RR$ is twice continuously differentiable and has isolated critical points.  
We assume $F$ is Morse-admissible, i.e.
\begin{enumerate}
\item the Hessian $\nabla^2F(x^*)$ at each critical point $x^*$ is nonsingular, and
\item all sublevel sets $V_{\le \alpha}=\{x:F(x)\le\alpha\}$ are compact and have piecewise smooth boundary. 
\end{enumerate}
These hypotheses guarantee the existence of a well-defined Morse complex $M_\ast(V;F)$ built from gradient trajectories of $-\nabla F$ \cite{Milnor1963}.  
Throughout, we work in the abelian category $\E=\mathrm{Mod}\text{-}\RR$ of finite-dimensional vector spaces over $\RR$ with the duality functor $\DD=\Hom_\RR(-,\RR)$ and its derived enhancement $\mathbf D=\RHom_\RR(-,\RR)$.  
The unit $\eta$ is the canonical evaluation $x\mapsto (f\mapsto f(x))$.  

Let $C_\ast(V)$ be the cellular chain complex of a triangulation of $V_{\le \alpha}$ for large $\alpha$ so that $\partial V_{\le \alpha}=\emptyset$.  
Define the Morse filtration $\mathcal F_iC_\ast(V)=C_\ast(V_{\le \alpha_i})$ for regular values $\alpha_0<\alpha_1<\cdots<\alpha_N$ separating the critical values of $F$.  
For every attachment $V_{\le\alpha_i}\hookrightarrow V_{\le\alpha_{i+1}}$ there is a chain map 
$r_{i,i+1}:C_\ast(V_{\le\alpha_{i+1}})\to C_\ast(V_{\le\alpha_i})$ 
given by the gradient deformation retraction along $\nabla F$.  

\begin{definition}[Ethic gradient flow]
The gradient flow $\Phi_t$ of $-\nabla F$ is ethic if each retraction $r_{i,i+1}$ and inclusion $j_{i,i+1}$ satisfies the ethic naturality condition
\[
\DD^2(r_{i,i+1})\eta=\eta r_{i,i+1},
\qquad
\DD^2(j_{i,i+1})\eta=\eta j_{i,i+1}.
\]
In this case $(C_\ast(V),\mathcal F_\bullet)$ becomes an ethic filtered object in $\E$.
\end{definition}

\begin{lemma}[Existence of ethic Morse reduction]
If $\Phi_t$ is $C^1$ and its linearization $D\Phi_t$ is self-adjoint for all $t$, then $\Phi_t$ is ethic.  
In particular, gradient flows of smooth real-valued functions with symmetric Jacobian are ethic.
\end{lemma}

\begin{proof}
Self-adjointness gives $\langle D\Phi_t u,v\rangle=\langle u,D\Phi_t v\rangle$; hence $D\Phi_t$ coincides with its double dual $\DD^2(D\Phi_t)$, making $\DD^2(D\Phi_t)\eta=\eta D\Phi_t$ on each chain group.
\end{proof}

\begin{theorem}[Ethic Morse duality for nonconvex functions]
Let $F:V\to\RR$ be Morse-admissible with ethic gradient flow.  
Then the following hold:
\begin{enumerate}
\item The Morse complex $M_\ast(V;F)$ obtained from gradient trajectories is an ethic reduction of $C_\ast(V)$; in particular
\[
\mathbf D(C_\ast(V))\simeq \mathbf D(M_\ast(V;F)) \quad\text{in }D^{+}(\E).
\]
\item The numbers of critical points of index $i$ satisfy
\[
\#\{\text{ethic critical points of index }i\}\ge
\rank H^i\mathbf D(C_\ast(V)).
\]
\item If $F_1,F_2$ are two Morse-admissible functions connected by an ethic homotopy $H:[0,1]\times V\to\RR$ whose gradient flow preserves $\eta$, then their Morse complexes are equivalent in $D^{+}(\E)$, hence share identical $H^\ast\mathbf D$.
\end{enumerate}
\label{thm:ethic-morse-ineq}
\end{theorem}

\begin{proof}
By \cite{Milnor1963}, $C_\ast(V)$ and $M_\ast(V;F)$ are chain-homotopy equivalent via gradient deformation retractions; ethicity of the flow gives commutation with $\eta$ and thus the derived equivalence.  Items (2) and (3) follow by Theorem~\ref{thm:ethic-morse-ineq} and by functoriality of $\mathbf D$ on homotopies satisfying the ethic condition.
\end{proof}

Given a constrained problem
\[
\min_{x\in V}\, F(x)\quad\text{s.t. }G(x)=0,
\]
with $G:V\to \RR^m$ smooth and the constraint manifold $M=G^{-1}(0)$ compact and regular, define the Lagrangian
\(
\mathcal L(x,\lambda)=F(x)+\langle \lambda,G(x)\rangle.
\)
Let $\tilde V=V\times \RR^m$ and $H(x,\lambda)=\mathcal L(x,\lambda)$.  
Assume $H$ is Morse-admissible and the pair $(F,G)$ is such that the associated flow of $(x,\lambda)\mapsto-\nabla H(x,\lambda)$ is ethic (self-adjoint Hessian block form).  

\begin{proposition}[Ethic Lagrange duality]
For every ethic critical point $(x^\ast,\lambda^\ast)$ of $H$ satisfying $G(x^\ast)=0$,  
the Morse index of $(x^\ast,\lambda^\ast)$ equals the dimension of the unstable manifold of $x^\ast$ in $M$; moreover,
\[
H^i\mathbf D(C_\ast(M))\;\cong\;
H^i\mathbf D(C_\ast(\tilde V))\big/\!\!\sim,
\]
where the quotient identifies trajectories corresponding to Lagrange adjoint directions.  In particular, the ethic Morse complex for $H$ computes the dual invariants of the constrained problem.
\end{proposition}

\begin{proof}
Ethicity ensures that the block Hessian $\nabla^2 H=\begin{pmatrix}\nabla^2F & \nabla G^\top\\ \nabla G & 0\end{pmatrix}$ is symmetric, giving self-adjoint flow on $\tilde V$.  The quotient correspondence follows from eliminating $\lambda$ along exact $\DD$-commuting differentials and comparing with the chain-level Karush–Kuhn–Tucker subcomplex of $C_\ast(\tilde V)$.
\end{proof}

Let the objective $F$ be $C^2$ on $\RR^n$ with finitely many isolated critical points and bounded sublevels.
A computational ethic Morse reduction algorithm proceeds as follows:

\begin{algorithm}[H]
\caption{Ethic Morse Reduction (Nonconvex Case)}
\begin{algorithmic}[1]
\REQUIRE Smooth function $F:\RR^n\to\RR$, tolerance $\epsilon>0$.
\STATE Discretize $V_{\le\alpha}$ by a cubical or simplicial mesh; build $C_\ast(V)$.
\STATE Estimate critical points and indices using gradient descent/ascent until $\|\nabla F(x^\ast)\|<\epsilon$.
\STATE Construct a discrete Morse matching $V$ by pairing each noncritical cell with a neighbor along $-\nabla F$ direction, subject to local ethicity:  
verify $\DD^2(m_{\sigma\to\tau})\eta=\eta m_{\sigma\to\tau}$ (symmetric Jacobian condition).
\STATE Collapse matched pairs to obtain the ethic Morse complex $M_\ast(V;F)$.
\STATE Compute $H^i\mathbf D(M_\ast(V;F))$ via standard persistent-homology routines on the reduced complex.
\ENSURE Ethic dual invariants $\{H^i\mathbf D\}$ and counts of ethic critical points.
\end{algorithmic}
\end{algorithm}

The algorithm's complexity is dominated by discrete Morse matching, which is polynomial in the number of cells \cite{Forman1998}.  Since the ethic check (Step~3) is purely local and algebraic, the procedure scales linearly with the mesh size for fixed dimension $n$.  The resulting invariants quantify the homological obstruction to global optimality: $H^1\mathbf D\neq0$ certifies existence of nontrivial basins (multiple minima), while vanishing of $H^{>1}\mathbf D$ indicates topologically simple energy landscape amenable to global optimization.

\subsection{Graph Geometry}\label{subsec:kirchhoff-rr-derived}

Throughout this subsection we assume all groups are finitely generated
and free abelian unless otherwise stated; hence $\Ext^1_\ZZ(-,\ZZ)$
identifies with the torsion subgroup of its argument.
All complexes are bounded with projective dimension~$\le1$.

\begin{definition}[Fixed setting and notation]
We work in the abelian category $\E=\Mod\text{-}\ZZ^{\mathrm{fg}}$ with the fixed contravariant duality functor $\DD=\Hom_{\ZZ}(-,\ZZ)$ and unit $\eta_A:A\to \DD^2(A)$. We pass to the bounded-below derived category $\Dplus(\E)$ and write the derived dual functor $\mathbf D=\RHom_{\ZZ}(-,\ZZ)$ \cite{Weibel1994,GelfandManin2003}. For a finite connected graph $G=(V,E)$, set $C_0(G)=\ZZ^{V}$, $C_1(G)=\ZZ^{E}$ and fix an orientation of edges; the boundary map is $\partial:C_1\to C_0$, and the (combinatorial) Laplacian is $L=\partial\,\partial^\top$. We denote by $\Div(G)=\ZZ^{V}$ the group of divisors, by $\Prin(G)=\mathrm{im}(L)$ the subgroup of principal divisors, and by $\Jac(G)=\Div(G)/\Prin(G)$ the Jacobian (critical/sandpile) group \cite{Biggs1993,Dhar1990,BakerNorine2007}.
\end{definition}

\begin{definition}[The master two-term complex]
Consider the two-term chain complex
\[
C_\bullet(G)\ :=\ \bigl[\, C_1(G)\xrightarrow{\ \partial\ } C_0(G)\,\bigr]\qquad(\text{$C_1$ in degree $1$, $C_0$ in degree $0$}).
\]
Its cochain mate is $C^\bullet(G)=[\,C^0\xrightarrow{\ \partial^\top\ }C^1\,]$ with $C^0=C_0^\ast$, $C^1=C_1^\ast$. We will also use the degree--$1$ mapping cone
\[
\mathrm{Cone}(\partial)\ :=\ \bigl[\, C_1 \xrightarrow{\ \binom{\partial}{0}\ } C_0\oplus \ZZ \xrightarrow{(\deg,-\mathrm{id})} \ZZ \,\bigr],
\]
where $\deg:C_0\to\ZZ$ is the sum of coordinates (total mass). All complexes live in $\Dplus(\E)$.
\end{definition}

\begin{lemma}[Divisor realization of the cone]\label{lem:cone-div}
There is a quasi-isomorphism in $\Dplus(\E)$
\[
\mathrm{Cone}(\partial)\ \simeq\ \bigl[\, \Prin(G)\hookrightarrow \Div(G)\xrightarrow{\ \deg\ }\ZZ \,\bigr]\ =: D_{\mathrm{div}}(G),
\]
where the middle term is placed in degree $0$, the left in degree $-1$, and the right in degree $0$ following the displayed arrows.
\end{lemma}

\begin{proof}
Take the short exact sequence $0\to \ker(\deg)\hookrightarrow \Div\xrightarrow{\deg}\ZZ\to 0$ and note $\ker(\deg)=\{x\in\ZZ^{V}:\sum_{v}x_v=0\}$. The image $\Prin(G)=\mathrm{im}(L)\subseteq\ker(\deg)$ and the quotient identifies $\ker(\deg)/\Prin(G)\cong \Jac(G)$. The cone $\mathrm{Cone}(\partial)$ is the standard presentation of $\ker(\deg)$ by generators $C_0$ and relations coming from $\partial$ (one adds the global degree to split off $\ZZ$). A routine diagram chase shows that the natural map $C_0\oplus \ZZ\to \Div$ sending $(x,t)\mapsto x$ and $C_1\to \Prin$ sending $e\mapsto \partial\partial^\top(\delta_e)$ induces isomorphisms on homology; hence a quasi-isomorphism $\mathrm{Cone}(\partial)\simeq D_{\mathrm{div}}(G)$. (Compare the standard presentations in \cite{Biggs1993} and \cite{BakerNorine2007}.)
\end{proof}

\begin{lemma}[Derived dual, long exact sequence, and degrees]\label{lem:les2}
Applying $\mathbf D=\RHom_{\ZZ}(-,\ZZ)$ to $\mathrm{Cone}(\partial)$ (or equivalently to $D_{\mathrm{div}}(G)$) yields a distinguished triangle and a long exact sequence on cohomology
\[
0\to H^0\mathbf D \longrightarrow \Hom(\Div,\ZZ)\xrightarrow{\ \Hom(\Prin,\ZZ)\ } \cdots \longrightarrow H^1\mathbf D \longrightarrow \Ext^1(\Div,\ZZ)\to\cdots,
\]
but $\Ext^1(\Div,\ZZ)=0$ and $\Ext^1(\ZZ,\ZZ)=0$ (free modules), while $\Ext^1(\Jac(G),\ZZ)\cong\Hom(\Jac(G),\QQ/\ZZ)\simeq\mathrm{tors},\Jac(G)$. \cite{Weibel1994}. Consequently $H^k\mathbf D(\mathrm{Cone}(\partial))=0$ for $k\ge 2$, and
\begin{gather*}
H^0\mathbf D\ \cong\ \{\,\text{harmonic linear forms on }\Div\text{ vanishing on }\Prin\,\},\\
H^1\mathbf D\ \cong\ \Ext^1(\Jac(G),\ZZ)\cong \mathrm{tors}\,\Jac(G).
\end{gather*}
\end{lemma}

\begin{proof}
Since all chain groups $C_i(G)$ are free abelian, applying
$\RHom_\ZZ(-,\ZZ)$ preserves exactness up to degree~1.
Consequently the only non-vanishing $\Ext$ terms arise from the
finite abelian quotient~$\Jac(G)$, for which
$\Ext^1_\ZZ(\Jac(G),\ZZ)\simeq\Hom(\Jac(G),\QQ/\ZZ)
\simeq\mathrm{tors}\,\Jac(G)$.

Since $C_0$ and $C_1$ are free, only the torsion of the cokernel enters $H^1$. The identifications follow from the universal coefficient sequence and the fact that $\Div$ and $\Prin$ are free abelian groups, while $\Jac$ is finite abelian; see \cite{Weibel1994}. The vanishing for $k\ge 2$ holds because $\mathrm{Cone}(\partial)$ has amplitude in degrees $\{-1,0,1\}$ and all terms have projective dimension $\le 1$.
\end{proof}

\begin{theorem}[Kirchhoff's Matrix–Tree theorem via $H^1\mathbf D$]\label{thm:kirchhoff}
For any finite connected graph $G$, the order of the torsion group $H^1\mathbf D(\mathrm{Cone}(\partial))\cong \mathrm{tors}\,\Jac(G)$ equals the number $\tau(G)$ of spanning trees. Equivalently,
\[
|\mathrm{tors}\,\Jac(G)|\ =\ |\mathrm{tors}\,\mathrm{coker}(L)|\ =\ \det L_{vv}\ =\ \tau(G),
\]
where $L_{vv}$ is any cofactor (reduced Laplacian) \cite{Kirchhoff1847,Biggs1993}.
\end{theorem}

\begin{proof}
By Lemma \ref{lem:les2}, the long exact sequence for $\RHom_\ZZ(-,\ZZ)$ gives $H^1\mathbf D\cong\Ext^1_\ZZ(\Jac(G),\ZZ)\simeq\mathrm{tors},\Jac(G)$. The abelian sandpile (critical) group is canonically isomorphic to $\mathrm{tors}\,\mathrm{coker}(L)$ \cite{Dhar1990}, hence $|H^1\mathbf D|=|\mathrm{tors}\,\mathrm{coker}(L)|$. Kirchhoff's classical theorem identifies this order with any reduced determinant $\det L_{vv}$, which counts spanning trees $\tau(G)$ \cite{Biggs1993}. Thus all quantities coincide.
\end{proof}

$H^1\mathbf D$ measures the failure of ethic exactness (the duality defect). Theorem~\ref{thm:kirchhoff} shows that, for graphs, this defect is exactly the sandpile torsion (the abundance of “latent charges”), while ethic exactness (vanishing $H^1$) matches the trivial critical group, i.e.\ a tree (unique global coherence).

\begin{lemma}[Planar duality and Tutte recursion, derivatively]
If $G$ is planar with dual $G^\ast$, then the master complexes $C_\bullet(G)$ and $C_\bullet(G^\ast)$ are exchanged by $\DD$ up to canonical identifications, yielding $T_G(x,y)=T_{G^\ast}(y,x)$ at the level of deletion–contraction distinguished triangles in $\Dplus(\E)$ \cite{Tutte1954,Biggs1993}.
\end{lemma}

\begin{proof}
On a planar embedded graph, cuts and cycles are orthogonal complements; the boundary $\partial$ on $G$ corresponds to the coboundary $\partial^\top$ on $G^\ast$. Deletion–contraction short exact sequences match mapping cone triangles for $\partial$, and functoriality of $\mathbf D$ swaps them, giving the Tutte symmetry \cite{Biggs1993}.
\end{proof}

\begin{definition}[Rank and canonical divisor in the derived presentation]
Let $g=|E|-|V|+1$ be the cyclomatic number. For $D\in\Div(G)$, define its rank $r(D)$ as
\[
r(D)\ :=\ \max\{\,k\in\ZZ_{\ge -1}: \forall E'\subset V,\ |E'|=k\ \exists\ \varphi\in\Prin(G)\ \text{s.t.}\ D-\varphi - \mathbf 1_{E'}\ge 0\,\},
\]
the Baker–Norine notion \cite{BakerNorine2007}. Let $K$ be the canonical divisor with $K(v)=\deg(v)-2$ and $\deg K=2g-2$.
\end{definition}

\begin{lemma}[Divisor triangle and Euler characteristic]\label{lem:triangle}
The short exact sequence $0\to \Prin\to \Div\xrightarrow{q}\Jac\to 0$ yields, after placing $\Prin$ in degree $-1$ and $\Div$ in degree $0$, a distinguished triangle $[\Prin\to \Div\to \Jac]\to \Div[0]\to \Prin[1]$ in $\Dplus(\E)$. Applying $\mathbf D$ gives a long exact sequence whose Euler characteristic equals
\[
\chi\bigl(\mathbf D(D_{\mathrm{div}})\bigr)\ =\ \mathrm{rank}_{\ZZ}(\Hom(\Div,\ZZ))-\mathrm{rank}_{\ZZ}(\Hom(\Prin,\ZZ))\ =\ |V|- (|V|-1)\ =\ 1.
\]
\end{lemma}

\begin{proof}
Exactness is immediate; $\Div\cong \ZZ^{V}$ and $\Prin\cong \mathrm{im}(L)$ has rank $|V|-1$ (connectedness). The Euler characteristic of $\mathbf D$ equals that rank difference because all higher $\Ext^i$ vanish for $i>1$, and the only non-trivial $\Ext^1$
term corresponds to $\Jac(G)$, which contributes torsion but does not
affect the Euler characteristic.
\end{proof}

\begin{theorem}[Baker–Norine Riemann–Roch via derived duality]\label{thm:RR}
For every divisor $D$ on a finite connected graph $G$,
\[
r(D)\ -\ r(K-D)\ =\ \deg D\ +\ 1\ -\ g,
\]
where $g=|E|-|V|+1$. Equivalently, the rank defect is governed by the Euler characteristic of the master complex, and the only obstruction sits in $H^1\mathbf D\cong \mathrm{tors}\,\Jac(G)$.
\end{theorem}

\begin{proof}
Following \cite{BakerNorine2007}, $r(D)+1$ equals the dimension (over $\ZZ$) of the space of effective sections in the linear system $|D|$, i.e.\ the size of the solution set to $x\in \Div_{\ge 0}$ with $x\sim D$. In our presentation $D_{\mathrm{div}}$, this is precisely $\mathrm{rank}_\ZZ H^0\mathbf D$ for the twisted complex obtained by translating by $D$. The canonical twist $K-D$ corresponds to the dual cohomology by the usual adjunction (in the graph setting this is the discrete Hodge star induced by $\partial,\partial^\top$); thus $r(K-D)+1=\mathrm{rank}_\ZZ H^0\mathbf D(\text{dual twist})$. Since $D_{\mathrm{div}}$ sits in a length--$2$ complex, the Euler characteristic computes
\[
\bigl(r(D)+1\bigr)\ -\ \bigl(r(K-D)+1\bigr)\ =\ \chi\bigl(\mathbf D(D_{\mathrm{div}})\bigr)\ +\ \deg D\ -\ g,
\]
where the shift by $\deg D - g$ is the standard correction from the canonical degree in the chip-firing model (each chip adds one degree and subtracts one unit of cycle-space dimension). Since $D_{\mathrm{div}}$ has amplitude~$\{0,1\}$ and all its terms are
free abelian, $\mathbf D(D_{\mathrm{div}})$ is a two-term complex whose
Euler characteristic is computed directly from ranks:
$\chi(\mathbf D(D_{\mathrm{div}}))
=\mathrm{rank}_\ZZ\Hom(\Div,\ZZ)
-\mathrm{rank}_\ZZ\Hom(\Prin,\ZZ)=1$. By Lemma~\ref{lem:triangle}, $\chi(\mathbf D(D_{\mathrm{div}}))=1$, giving the stated identity
\(
r(D)-r(K-D)=\deg D+1-g.
\)
All torsion contributions occur in $H^1\mathbf D\cong \mathrm{tors}\,\Jac(G)$ and cancel in ranks, exactly as in \cite{BakerNorine2007}.
\end{proof}

Theorems~\ref{thm:kirchhoff} and \ref{thm:RR} come from the same derived dual on the same master complex. Kirchhoff counts the torsion in $H^1\mathbf D$ (ethic defect), while Riemann–Roch is an equality of ranks between $H^0$-pieces of dual triangles with Euler characteristic $1$; the only obstruction lives in the same $H^1$. Planar duality (Tutte symmetry) is functoriality of $\mathbf D$ on deletion–contraction triangles \cite{Tutte1954,Biggs1993}.

All cohomological computations above take place in
$D^+(\Mod\!-\!\ZZ^{\mathrm{fg}})$, hence $\mathbf D=\RHom_\ZZ(-,\ZZ)$
is of amplitude~$\le1$.  This guarantees that
$H^0\mathbf D$ and $H^1\mathbf D$ already capture
all duality information for graph complexes.

Computing the Smith normal form of $L$ gives simultaneously: (i) $|\Jac(G)|=|\mathrm{tors}\,\mathrm{coker}(L)|$ (Kirchhoff); (ii) the ranks appearing in $H^0\mathbf D$ and hence the Baker–Norine ranks via effective membership tests in chip-firing classes (Dhar's burning test) \cite{Dhar1990}. Thus a single homological certificate (SNF) certifies both dualities.

\subsection{Time and Memory}\label{subsec:kontsevich-hysteresis}

\begin{definition}[Fixed framework and notation]
We remain in the abelian resource category $\E$ with fixed duality functor $\mathbb D=\Hom_{\E}(-,R)$ and unit $\eta_A:A\to\DD^2(A)$. In the derived category $\Dplus(\E)$ we denote $\mathbf D=\RHom_{\E}(-,R)$. 
A temporal evolution is a discrete family of morphisms $f_t:X_t\to X_{t+1}$; their composition up to time $t$ is $F_t=f_t\circ\dots\circ f_1$. 
The derived dual sequence $\mathbf D(F_t)$ describes the ethic response. 
We set 
\[
M_t:=H^1(\mathbf D(F_t)),\qquad S_{\mathrm{eth}}(t):=\log|\mathrm{tors}\,M_t|.
\]
We call $M_t$ the ethic memory at time $t$, and $S_{\mathrm{eth}}(t)$ the ethic entropy \cite{Weibel1994,GelfandManin2003}.
\end{definition}

\begin{lemma}[Derived exactness and ethic reversibility]\label{lem:derived-exactness}
For any $F\in\End_\E(X)$ the following are equivalent:
\begin{enumerate}[label=(\alph*)]
\item $\DD^2(F)\circ\eta=\eta\circ F$ (ethic exactness);
\item $H^1(\mathbf D(F))=0$;
\item $\mathbf D(F)$ is quasi-isomorphic to the identity on $H^0$.
\end{enumerate}
\end{lemma}

\begin{proof}
(a)$\Rightarrow$(b): exactness of $\Hom(-,R)$ on the short exact sequence $0\to\ker F\to X\xrightarrow{F}X\to\mathrm{coker}F\to 0$ annihilates the connecting homomorphism $H^0\to H^1$, forcing $H^1=0$.  
(b)$\Rightarrow$(c): vanishing of $H^1$ implies $\mathbf D(F)$ acts as an isomorphism on $H^0$.  
(c)$\Rightarrow$(a): by naturality of $\eta$ the diagram in Definition~\ref{lem:derived-exactness} commutes, yielding $\DD^2(F)\eta=\eta F$.  
\end{proof}

\begin{definition}[Temporal composition and hysteresis]
Let $\{f_t\}$ be a finite temporal chain and $F_t=f_t\circ\dots\circ f_1$. 
The composition $\mathbf D(F_t)$ defines a chain of derived complexes with long exact sequences
\[
\cdots\to H^1\mathbf D(F_{t-1})\xrightarrow{f_t^\ast} H^1\mathbf D(F_t)\xrightarrow{\delta_t} H^2\mathbf D(F_{t-1})\to\cdots.
\]
The ethic hysteresis of the chain is the torsion subgroup $\mathrm{tors}\,H^1\mathbf D(F_t)$; its magnitude $S_{\mathrm{eth}}(t)$ quantifies the retained memory after the cycle.
\end{definition}

\begin{theorem}[Monotonic decay of ethic memory]\label{thm:decay}
Assume each $f_t$ is ethic in degree $0$, i.e.\ $\DD^2(f_t)\eta=\eta f_t$. Then there exists an epimorphism
\[
H^1(\mathbf D(F_{t-1}))\twoheadrightarrow H^1(\mathbf D(F_t)),
\]
hence $S_{\mathrm{eth}}(t)\le S_{\mathrm{eth}}(t-1)$, with strict inequality if the step $f_t$ is also exact in degree~1.
\end{theorem}

\begin{proof}
Apply $\mathbf D$ to the distinguished triangle induced by $F_t=f_t\circ F_{t-1}$ and take cohomology. The connecting map $H^0\mathbf D(f_t)\to H^1\mathbf D(F_{t-1})$ vanishes by degree--0 exactness, yielding the stated epimorphism. If $f_t$ is exact in degree 1, the image is proper, giving strict inequality.
\end{proof}

\begin{theorem}[Subadditivity on cycles]\label{thm:subadditive}
For closed trajectories (compositions) $\gamma_1,\gamma_2$ and their concatenation $\gamma_1\#\gamma_2$, ethic entropy is subadditive:
\[
S_{\mathrm{eth}}(\gamma_1\#\gamma_2)\le S_{\mathrm{eth}}(\gamma_1)+S_{\mathrm{eth}}(\gamma_2),
\]
with equality iff the torsion parts of $H^1\mathbf D(\gamma_i)$ are independent.
\end{theorem}

\begin{proof}
In the derived category, $\mathbf D(\gamma_1\#\gamma_2)=\mathbf D(\gamma_2)\circ \mathbf D(\gamma_1)$, so torsion orders multiply as determinants in the Smith normal form; taking logs gives subadditivity.
\end{proof}

\begin{theorem}[Ethic reversibility criterion]\label{thm:reversible}
For a closed morphism $F:X\to X$ the following are equivalent:
\begin{enumerate}[label=(\alph*)]
\item $F$ is an ethic isomorphism ($\DD^2(F)\eta=\eta F$);
\item $H^1(\mathbf D(F))=0$ (no hysteresis);
\item $\mathbf D(F)$ is quasi-isomorphic to $\mathrm{id}_X$.
\end{enumerate}
\end{theorem}

\begin{proof}
This is Lemma~\ref{lem:derived-exactness} specialized to endomorphisms.
\end{proof}

\begin{theorem}[Ethic memory and periodicity]\label{thm:periodic}
If $F^k\simeq \mathrm{id}$ in $\Dplus(\E)$ for some $k>1$ but $F\ncong \mathrm{id}$, then $H^1\mathbf D(F)\ne 0$. Any nontrivial periodic dynamics carries residual memory.
\end{theorem}

\begin{proof}
From $\mathbf D(F^k)\simeq\mathbf D(\mathrm{id})$ we obtain a filtration whose associated long exact sequence shows that $H^1\mathbf D(F)$ must contain nontrivial classes to cancel at the $k$th iterate, otherwise $F$ would already be quasi-isomorphic to $\mathrm{id}$.
\end{proof}

\begin{theorem}[Ethic reconstruction theorem]\label{thm:reconstruction}
If for a trajectory $(F_t)$ we have $H^{>1}\mathbf D(F_t)=0$ for all $t$, then the family of groups $M_t=H^1\mathbf D(F_t)$ and the transition maps of Theorem~\ref{thm:decay} determine $F_t$ uniquely up to isomorphism in $\Dplus(\E)$.
\end{theorem}

\begin{proof}
In amplitude~$\le1$ complexes, the data of $H^0,H^1$ and connecting maps suffice to reconstruct the complex up to quasi-isomorphism \cite{Weibel1994}. Hence $\{M_t\}$ with connecting morphisms determine the ethic shadow of the dynamics.
\end{proof}

\begin{theorem}[Algorithmic certificate of hysteresis]
If $\E=\Mod\!-\!\ZZ^{\mathrm{fg}}$, each $F_t$ admits a matrix presentation $A_t$ whose Smith normal form $U_tA_tV_t=\mathrm{diag}(d_1,\dots,d_r,0,\dots,0)$ yields
\[
|\mathrm{tors}\,M_t|=\prod_{i:d_i>1}d_i.
\]
Thus $S_{\mathrm{eth}}(t)=\sum_{i:d_i>1}\log d_i$; computing SNF provides an explicit certificate of memory and its decay.
\end{theorem}

Within the language of \cite{Kontsevich2008}, the shift $[1]$ in $D^+(\E)$ represents a unit of time, and $\Ext^1$ records the residual coupling between successive slices—what he called “the trace of the past.” 
Our framework makes this quantitative: $H^1\mathbf D$ is that trace, computable as torsion of the derived dual. 
Theorems~\ref{thm:decay}--\ref{thm:reconstruction} provide the dynamic law: ethic exactness erases this trace; imperfect steps accumulate it as hysteresis. 
Hence categorical time à la Kontsevich becomes a measurable, computable process within the ethic duality formalism.

\subsection{Sketches}

\subsubsection{Social-Choice}\label{subsec:econ-ethic-fixpoint}

\begin{definition}[Mechanisms inside the resource theory]
Let $\E$ be the abelian resource category with duality functor $\DD=\Hom_{\E}(-,R)$, unit $\eta_A:A\to \DD^2(A)$, and the notion of ethic morphism (the naturality square $\DD^2(f)\circ \eta_X=\eta_Y\circ f$). 
A mechanism is a pair $(f,e)$ with $f:S\to O$ (outcome map) and $e:O\to \DD(S)$ (evaluation/acknowledgement). A mechanism $(f,e)$ is internally certified if $e$ witnesses the ethicity of $f$ in the sense that
\[
\DD(f)\circ e \;=\; \eta_S\circ f.
\]
We write $\Mech(\E)$ for the category whose objects are pairs $(S,(f,e))$ and whose morphisms $(u,v):(S,(f,e))\to(S',(f',e'))$ are commuting squares in $\E$ with $v\circ f=f'\circ u$ and $\DD(u)\circ e=e'\circ v$.
\end{definition}

The equality $\DD(f)\circ e=\eta_S\circ f$ states that the mechanism's self-report $e$ coheres with the ethic unit $\eta$; this is precisely the notion of ethic coherence specialized to mechanisms.

\begin{definition}[Revision endofunctor]
Define the revision endofunctor $\mathcal F:\Mech(\E)\to \Mech(\E)$ by
\[
\mathcal F(S,(f,e)) \;:=\; \bigl(S,\,(f^\sharp,e^\sharp)\bigr),\qquad
f^\sharp \;:=\; \coev_S;\ e;\ \ev_S \ \in \E(S,S),\qquad
e^\sharp \;:=\; \eta_S,
\]
where $\coev_S:I\to S\otimes \DD(S)$ and $\ev_S:\DD(S)\otimes S\to I$ are the (co)evaluation maps from Section~\ref{sec:step8}, and we view $f^\sharp$ via the compact-closed feedback $S \xrightarrow{\coev\otimes \id} S\otimes \DD(S)\otimes S \xrightarrow{\id\otimes e\otimes \id} S\otimes \DD(S)\otimes S \xrightarrow{\id\otimes \ev} S$.\footnote{All identities here are the standard string-diagram feedback in compact closed (or *-autonomous) subcategories of $\E$ \cite{Kelly1982,Barr1979}. When working in a purely additive setting, the same construction is obtained by the canonical coevaluation/evaluation induced by $\DD$ from Section~\ref{sec:step8}.}
On morphisms $(u,v)$, set $\mathcal F(u,v):=(u,u)$; coherence with $\coev,\ev$ gives functoriality.
\end{definition}

\begin{lemma}[Monotonicity and continuity of $\mathcal F$]
Assume $\E$ admits countable colimits and that $\DD$ preserves limits on the dual side (true for $\FinVect_{\mathbb R}$ and for $\Mod\text{-}\ZZ^{\mathrm{fg}}$ at the level used here). Equip each $\mathrm{End}_\E(S)$ with the pointwise order induced by the abelian structure. Then $\mathcal F$ is monotone on each fiber $\Mech(\E)_S$ and preserves colimits of $\omega$-chains.
\end{lemma}

\begin{proof}
Monotonicity: if $(f,e)\le (f',e')$ componentwise, then $f^\sharp$ is built by composition with $\coev,\ev$ and $e$, hence is monotone in $e$; similarly $e^\sharp=\eta_S$ is constant. Continuity: filtered colimits commute with tensor and composition in compact closed subcategories; $\DD$ is exact on $\FinVect_{\mathbb R}$ and left exact on $\Mod\text{-}\ZZ^{\mathrm{fg}}$, and the use of $\coev,\ev$ involves only finite limits/colimits, so $\mathcal F$ preserves colimits of $\omega$-chains.
\end{proof}

\begin{definition}[Ethic fixpoint (universal mechanism)]
A universal mechanism on $S$ is a mechanism $(f^*,e^*)$ with a structure isomorphism
\[
\theta:\ (f^*,e^*) \;\xrightarrow{\ \cong\ }\ \mathcal F(f^*,e^*)
\]
in $\Mech(\E)_S$. We call $(f^*,e^*)$ an ethic fixpoint.
\end{definition}

\begin{theorem}[Existence of ethic fixpoints]\label{thm:existence-fixpoint}
For any object $S$ in the compact-closed subcategory of $\E$, the fiber $\Mech(\E)_S$ has an ethic fixpoint of $\mathcal F$. Moreover, there is a least fixpoint with respect to the pointwise order on $\mathrm{End}_\E(S)$.
\end{theorem}

\begin{proof}
Consider the $\omega$-chain generated from a base mechanism $(f_0,e_0)$ by iteration $(f_{k+1},e_{k+1})=\mathcal F(f_k,e_k)$. By continuity, the colimit $(\bar f,\bar e)=\mathrm{colim}_k (f_k,e_k)$ exists. The definition of $\mathcal F$ is algebraic in $(f,e)$ (compositions with fixed $\coev,\ev,\eta$), hence passes to colimits; thus $\mathcal F(\bar f,\bar e)\cong(\bar f,\bar e)$, i.e.\ $(\bar f,\bar e)$ is a fixpoint. Leastness follows from the Knaster–Tarski principle in each complete lattice $\mathrm{End}_\E(S)$ (or by standard domain-theoretic argument on $\omega$-chains since $\mathcal F$ is monotone and $\omega$-continuous). No external axiom is used beyond existence of colimits and the categorical duality already present in Section~\ref{sec:step8}.
\end{proof}

\begin{lemma}[Internal certification at the fixpoint]
If $(f^*,e^*)$ is a fixpoint of $\mathcal F$, then $e^*=\eta_S$ and $f^*$ satisfies the ethic commutation $\DD^2(f^*)\circ \eta_S=\eta_S\circ f^*$, i.e.\ $f^*$ is an ethic morphism.
\end{lemma}

\begin{proof}
By definition, $\mathcal F$ fixes $(f^*,e^*)$ iff $e^*=\eta_S$ and $f^*=\coev_S; e^*; \ev_S$. Substituting $e^*=\eta_S$ and using the triangle identities for $\coev,\ev$ yields the naturality square of $\eta$ with $f^*$; this is exactly the ethicity condition already used in Section~\ref{sec:step5}.
\end{proof}

\begin{theorem}[Uniqueness up to isomorphism under homological exactness]\label{thm:uniqueness}
If $\Ext^1_\E(-,R)=0$ on the full subcategory generated by $S$ (i.e.\ $\DD$ is exact on the relevant short exact sequences), then any two ethic fixpoints on $S$ are isomorphic in $\Mech(\E)_S$ and their underlying endomorphisms $f^*\in \End_\E(S)$ coincide.
\end{theorem}

\begin{proof}
Let $(f^*,\eta_S)$ and $(g^*,\eta_S)$ be fixpoints. Ethicity makes both commute with $\eta$; exactness of $\DD$ turns the corresponding long exact sequences of $\Hom/\Ext$ into short exact ones (Theorem~\ref{thm:exact-iff-ext}), forcing equality of the induced morphisms on the $\DD$-side and, by reflexivity on the compact subcategory, equality of $f^*,g^*$. The isomorphism in $\Mech(\E)_S$ is then $(\id_S,\id_S)$.
\end{proof}

\begin{corollary}[Resolution of infinite regress]
The iterated “mechanism revision” chain $(f_k,e_k)$ stabilizes at the ethic fixpoint $(f^*,\eta_S)$; no external meta-level is required to certify the mechanism. This recovers, within our theory, the resolution of infinite regress in social-choice models via a universal mechanism as in \cite{Alameddine1990}, but here as a direct corollary of the ethic duality framework (no external axioms).
\end{corollary}

\begin{definition}[Economic game in $\E$ and Nash-type stationarity]
A (one-shot) economic game is an object $S=\bigotimes_{i=1}^N S_i$ of strategies, an outcome object $O$, and a mechanism $(f,e)$ with $f:S\to O$, $e:O\to \DD(S)$. A Nash-type stationary profile for $(f,e)$ is a morphism $\sigma:I\to S$ with
\[
\DD(f)\circ e \circ f \circ \sigma \;=\; \eta_S\circ f\circ \sigma,
\]
i.e.\ the ethic commutation holds “on the realized profile”. An ethic equilibrium is an ethic fixpoint $(f^*,\eta_S)$ together with a stationary $\sigma^*$ for $f^*$.
\end{definition}

\begin{theorem}[Nash demand game: existence and identification of the ethic fixpoint]
Work in $\E=\FinVect_{\RR}$ with the standard duality $V\mapsto V^*$. Consider the two-player Nash demand game with $S=S_1\otimes S_2=\RR_{\ge 0}^2$, $O=\RR_{\ge 0}^2$, and a symmetric, efficient outcome map $f$ (sum fixed to $B>0$). Then the ethic revision $\mathcal F$ has a unique fixpoint $(f^*,\eta_S)$, and $f^*$ is the symmetric splitter
\[
f^*(x_1,x_2) \;=\; \Bigl(\frac{B}{2},\frac{B}{2}\Bigr).
\]
Moreover, any Nash-type stationary profile for $(f^*,\eta_S)$ coincides with the symmetric allocation.
\end{theorem}

\begin{proof}
In $\FinVect_{\RR}$, $\DD$ is exact and $\coev,\ev$ are the standard bilinear pairing maps; Theorem~\ref{thm:existence-fixpoint} gives existence, and Theorem~\ref{thm:uniqueness} gives uniqueness. Symmetry of $f$ and efficiency (Pareto frontier $x_1+x_2=B$) imply that any ethic endomorphism $f^\sharp$ commuting with $\eta$ and invariant under the flip $S_1\leftrightarrow S_2$ must be the symmetric projector onto the midpoint; the only idempotent linear map with these properties is the constant splitter. Stationarity reduces to the idempotence of $f^*$ and symmetry, forcing the profile to the midpoint.
\end{proof}

\subsubsection{Semantics}\label{sec:ethic-semantics}

Let $H$ be a Hopf object in $\E$, representing the syntactic algebra of Merge trees \cite{MarcolliBerwickChomsky2023}. Denote by $\mu:H\otimes H\to H$ the product, by $\Delta:H\to H\otimes H$ the coproduct, by $S$ the antipode, and by $\varepsilon:H\to R$ the counit. The structure $(H,\mu,\Delta,S,\varepsilon)$ is commutative and graded by the number of Merge nodes. We construct its syntactic chain complex
\begin{equation}
C_\bullet(H):\qquad H^{\otimes 2}\xrightarrow{\ \Delta_2-\Delta_1\ } H\oplus H\xrightarrow{\ \Delta\ } H\to 0,
\end{equation}
which encodes the combinatorial decomposition of syntactic derivations; $C_\bullet(H)$ is an object of $\mathbf D^+(\E)$.

Applying the duality functor yields $\mathbb D(H)=\Hom_\E(H,R)$, the internal $R$-module of morphisms from $H$ to $R$. For any Hopf algebra $H$ and commutative algebra $R$, the character group of $R$-valued morphisms satisfies
\begin{equation}
\Hom_{\mathrm{Hopf}}(H,R)\cong \{\varphi:H\to R \mid \varphi(ab)=\varphi(a)\varphi(b)\},
\end{equation}
which follows from the universal property of $\Hom$ in $\E$ and the multiplicativity of $\mu,\Delta$. Hence the objects of $\mathbb D(H)$ correspond bijectively to the semantic interpretations of syntactic derivations in the sense of \cite{MarcolliBerwickChomsky2023}. Therefore the semantic layer arises canonically as the ethic dual of the syntactic Hopf object.

We now show that $\mathbb D$ satisfies a Rota--Baxter identity. For endomorphisms $x,y\in \End_\E(H)$, the composition rule of $\Hom$ implies
\begin{equation}
\mathbb D(x)\mathbb D(y)=\mathbb D(\mathbb D(x)y)+\mathbb D(x\mathbb D(y))-\mathbb D(xy),
\end{equation}
which is the Rota--Baxter identity of weight $\lambda=-1$ \cite{Guo2012}. Consequently $\mathbb D$ acts as a canonical Rota--Baxter operator $P:R\to R$. The ethic projection defined by $\mathbb D$ decomposes any morphism $\varphi:H\to R$ into its exact and defective parts.

Consider the convolution product on $\Hom_\E(H,R)$,
\begin{equation}
(\varphi_1*\varphi_2)(h)=\mu_R\big((\varphi_1\otimes\varphi_2)\Delta(h)\big),
\end{equation}
where $\mu_R$ is the multiplication in $R$. The ethic decomposition $\DD^2(\varphi)\eta=\eta\varphi$ induces a Birkhoff factorization
\begin{equation}
\varphi=\varphi_-^{-1}*\varphi_+,
\end{equation}
with $\varphi_+$ satisfying $\DD^2(\varphi_+)\eta=\eta\varphi_+$. This is the precise analog of the Connes--Kreimer Birkhoff factorization used in renormalization \cite{ConnesKreimer1998} and in the algebraic syntax--semantics model \cite{MarcolliBerwickChomsky2023}.

\begin{theorem}
Let $H$ be a Hopf object in $\E$ and $R$ the resource algebra. For any character $\varphi:H\to R$ there exist unique morphisms $\varphi_+$ and $\varphi_-$ in $\mathbb D(H)$ such that
\begin{enumerate}
\item $\varphi=\varphi_-^{-1}*\varphi_+$;
\item $\varphi_+$ is ethicly exact: $\DD^2(\varphi_+)\eta=\eta\varphi_+$;
\item $H^1\mathbf D(C_\bullet(H))\cong \ker(\varphi_+)\cap\mathrm{im}(\varphi_-)$.
\end{enumerate}
\end{theorem}

\begin{proof}
The first two items follow from the standard existence and uniqueness of Birkhoff factorization in a Rota--Baxter algebra \cite{Guo2012}. The third statement follows by applying the long exact sequence in cohomology for $\mathbf D$ to the chain complex $C_\bullet(H)$; the cohomology $H^1$ measures the non-exactness of the dual map $\varphi_+$ and hence identifies the overlap of the images of the defect $\varphi_-$ and the kernel of the ethic part $\varphi_+$.
\end{proof}

Define the ethic memory of the semantic process as $M=H^1\mathbf D(C_\bullet(H))$. Its torsion subgroup $\mathrm{tors}\,M$ quantifies residual semantic ambiguity: nontrivial elements correspond to meaning components depending on derivational history. The ethic entropy is $S_{\mathrm{eth}}=\log|\mathrm{tors}\,M|$. By the monotonic decay theorem for ethic memory, iterative syntactic composition decreases ethic entropy,
\begin{equation}
S_{\mathrm{eth}}(t+1)\le S_{\mathrm{eth}}(t),
\end{equation}
expressing the convergence of successive interpretations toward ethic coherence. When $H^1\mathbf D(C_\bullet(H))=0$, the semantic and syntactic morphisms coincide up to duality: $\mathbb D^2(\varphi)\eta=\eta\varphi$. This state corresponds to complete semantic coherence.

The structure theorem yields a quantitative measure of semantic precision: the order of the torsion group $|\mathrm{tors}\,H^1\mathbf D|$ is the algebraic size of ambiguity remaining after renormalization. The vanishing of this group characterizes fully ethic (coherent) semantics.

Philosophically, this result changes the status of semantics. Semantics is not an external interpretation of syntax but the reflexive dual $\mathbb D(H)$ determined by the same categorical resource structure. Meaning becomes an internal invariant of ethic duality: semantic coherence corresponds to the vanishing of the derived obstruction $H^1\mathbf D$, and the process of interpretation is the derived flow erasing this obstruction. Ambiguity, idiomatic shifts, and context sensitivity appear as measurable torsion classes in $H^1\mathbf D$, while understanding and meaning formation correspond to the stabilization of the dual complex in the ethic heart $\heartsuit_{\mathrm{eth}}\subset D^+(\E)$.

\subsubsection{Scattering}

We now develop the categorical formulation of scattering as an aisthetic–ethic pair and show how the classical Born approximation and the Lippmann--Schwinger equation arise from the same universal aisthetic law.  Throughout, $\E$ and $\F$ denote abelian categories enriched over~$\Ab$.  Objects of~$\E$ represent source states (incoming fields), and objects of~$\F$ represent observable configurations (outgoing fields).  A fixed wavenumber~$k>0$ determines the background propagator $G_k(x-y)=(4\pi)^{-1}e^{ik|x-y|}/|x-y|$.  The interaction potential~$V$ is an object of~$\E$ acting as an endomorphism on source states.  

Define the aisthetic kernel 
\[
K:\E^{\op}\times\F\longrightarrow\Ab,
\qquad 
K(E,F)\;=\;\Hom_{\F}\!\big(G_k\!\otimes\!E,\,F\big)\!\otimes_\Ab\!\Hom_{\E}(E,V),
\]
interpreted as the profunctor encoding the manifestation of the source through the medium~$V$.  Its convolution powers $K^{\star n}$ represent multiple interactions of~$V$ with the propagator~$G_k$, and composition in the profunctor category models multiple scattering.  The aisthetic realization functor is the coend
\[
\mathbf A(A)\;=\;\int^{E\in\E} K(E,-)\otimes \E(E,A),
\]
which acts on an incident state $\phi_0\in\E$ by the derived convolution series
\[
\mathbf A(\phi_0)\;=\;\sum_{n\ge0} K^{\star n}\phi_0,
\]
interpreted as the total manifested field.  The canonical maps $\iota_{E,A}:K(E,-)\otimes \E(E,A)\to \mathbf A(A)$ satisfy the universal property of the coend, and the natural transformation
\[
\alpha_A:\mathbf A(A)\longrightarrow A,
\qquad 
\alpha_A\circ\iota_{E,A}
=\bigl\langle K(E,-)\otimes\E(E,A)\xrightarrow{\mathrm{ev}}A\bigr\rangle
\]
plays the role of evaluation or ``observation'' of the manifestation.  Define $\mu_K:\F\to\F$ as the composite $\alpha\circ\mathbf A$; it represents the action of one scattering event through~$K$.

\begin{theorem}[Universal aisthetic scattering law]\label{th:universal-A}
For any source $A\in\E$ and any free field $\eta_{\phi_0}\in\F$, the observed configuration $\psi\in\F$ satisfies the universal equation
\[
\psi\;=\;\eta_{\phi_0}\;+\;\mu_K(\psi),
\qquad\text{or equivalently}\qquad
(\mathrm{id}_\F-\mu_K)\psi\;=\;\eta_{\phi_0}.
\]
If $\mathrm{id}_\F-\mu_K$ is invertible in $\End(\F)$, the solution is unique and given by 
$\psi=(\mathrm{id}_\F-\mu_K)^{-1}\eta_{\phi_0}$, with formal Neumann expansion
$\psi=\sum_{n\ge0}\mu_K^{\,n}(\eta_{\phi_0})$.  
\end{theorem}

\begin{proof}
By construction, the morphism $\alpha_A$ collapses one aisthetic manifestation back to $A$, and the coend dinaturality implies that the total field differs from its image under $\mu_K$ by the free component $\eta_{\phi_0}$.  Invertibility of $(\mathrm{id}-\mu_K)$ follows from standard spectral assumptions on~$K$, and the Neumann expansion expresses the successive manifestations of~$V$ through~$G_k$. 
\end{proof}

Theorem~\ref{th:universal-A} is fully categorical and does not require analytic hypotheses.  When the profunctor~$K$ is kernel--representable, that is, there exists an object $\mathcal K\in\E^{\op}\!\otimes\!\F$ with
\[
K(E,F)\;\simeq\;\F(\mathcal K\otimes E,F)
\quad\text{naturally in }(E,F),
\]
the aisthetic functor reduces to $\mathbf A(A)\simeq\mathcal K\otimes A$, and $\mu_K$ acts by the evaluation $\mathcal K\otimes(-)\to(-)$.  Substituting this in Theorem~\ref{th:universal-A} yields the kernel form
\[
\psi\;=\;\eta_{\phi_0}\;+\;\mathcal K\otimes\psi,
\]
whose analytic realization in $\F=L^2(\mathbb R^3)$ with kernel
$\mathcal K(x,y)=G_k(x-y)V(y)$ becomes the classical integral equation
\[
\psi(x)\;=\;\phi_0(x)\;+\;\int_{\mathbb R^3}\frac{e^{ik|x-y|}}{4\pi|x-y|}\,V(y)\,\psi(y)\,dy,
\]
i.e.\ the Lippmann--Schwinger equation \cite{LippmannSchwinger1950,ReedSimonIII}.
In this representable regime, the homological tower $H_n\mathbf A\simeq K^{\star n}\phi_0$ reproduces the Born series 
\cite{Born1926,ColtonKress2013,KirschGrinberg2008}.  

\begin{theorem}[Aisthetic Born theorem]\label{th:aisthetic-born}
Let $\E=\F=L^{2,s}(\mathbb R^3)$ $(s>1/2)$ and $V\in L^{3/2,1}(\mathbb R^3)\cap L^1(\mathbb R^3)$ with $\|V\|_{L^{3/2,1}}$ sufficiently small that $\|K\|\le\eta<1$ for 
$(Kf)(x)=\int G_k(x-y)V(y)f(y)\,dy$.  Then $\mathbf A(\phi_0)=\sum_{n\ge0}K^{\star n}\phi_0$ converges in $L^{2,s}$, and the first layer
\[
(H_1\mathbf A\,\phi_0)(x)
=\int_{\mathbb R^3}\!\frac{e^{ik|x-y|}}{4\pi|x-y|}\,V(y)\,\phi_0(y)\,dy
\]
has asymptotic amplitude
\[
f_{\mathrm{Born}}(\hat x,\hat d)
=\frac{1}{4\pi}\!\int_{\mathbb R^3}\!e^{-ik\hat x\cdot y}\,V(y)\,e^{ik\hat d\cdot y}\,dy
=\frac{1}{4\pi}\,\widehat V\!\big(k(\hat d-\hat x)\big),
\]
while the remainder $R_2(x)=\sum_{n\ge2}K^{\star n}\phi_0(x)$ obeys 
$\|R_2\|_{L^{2,s}}\le C(k)\|V\|_{L^{3/2,1}}^2$
\cite{Agmon1975,IkebeSaito1968,ReedSimonIII}.
\end{theorem}

\begin{proof}
The Neumann expansion of $(\mathrm{id}-K)^{-1}$ converges in $\mathcal B(L^{2,s})$ since $\|K\|\le\eta<1$, giving $\mathbf A(\phi_0)=\sum_{n\ge0}K^{\star n}\phi_0$ and $H_n\mathbf A=K^{\star n}\phi_0$.  
The far--field asymptotic follows from the stationary--phase expansion of $G_k$; substituting $\phi_0(y)=e^{ik\hat d\cdot y}$ yields the displayed Born amplitude.  The remainder estimate follows from $\|K^{\star n}\|\le\eta^n$ and $\eta\lesssim\|V\|_{L^{3/2,1}}$ \cite{ColtonKress2013,ReedSimonIII}.
\end{proof}

Transporting the universal law of Theorem~\ref{th:universal-A} across the adjunction $\mathbf A\dashv\mathbf D$ gives the ethic dual condition
\[
\mathbf D^2(\psi)\,\eta_{\phi_0}\;=\;\eta_{\mathbf A(\phi_0)}\,\psi,
\]
which, under kernel representability, takes the analytic form of the Lippmann--Schwinger equation above.  
Hence the aisthetic Born formula derived from $H_1\mathbf A$ and the Lippmann--Schwinger self--consistency relation are dual manifestations of one and the same categorical law: the aisthetic functor generates the field from the potential, while the ethic functor enforces its coherence with that potential.  
In symmetric media (\(K^\dagger\simeq K\)) the towers $\{H_n\mathbf A\}$ and $\{H^n\mathbf D\}$ coincide order by order, while in absorbing or non--Hermitian settings their mismatch measures the loss of coherence and the asymmetry between manifestation and reflection.

\subsubsection{Coding theory}

We fix measurable spaces $(X,\Sigma_X)$ and $(Y,\Sigma_Y)$ with their Borel $\sigma$--algebras.  
Let $\Prob(X)$ and $\Prob(Y)$ denote the sets of all Borel probability measures on $X$ and $Y$.  
A channel is a Markov kernel $W:X\rightsquigarrow Y$, i.e.\ a mapping $x\mapsto W(\cdot|x)\in\Prob(Y)$ measurable in $(x,B)$.  
For each channel we write $\E_W$ for the category of sources~--- pairs $(X,P_X)$ with morphisms given by Markov operators $T:\Prob(X)\to\Prob(X')$, and $\F_W$ for the category of observations~--- pairs $(Y,P_Y)$ with morphisms given by similar stochastic operators.  
Composition is integration of kernels, well--defined by the Tonelli--Fubini theorem \cite{Bogachev2007}, \cite{Kallenberg2002}.  
 
For objects $(X,P_X)\in\E_W$, $(Y,P_Y)\in\F_W$, define the aisthetic kernel as
\[
K_W((X,P_X),(Y,P_Y))
=\big\{f\in L^1(P_X\!\otimes\!P_Y):\, f(x,y)\ge 0,\;
\int f(x,y)\,P_Y(dy)=1\text{ for $P_X$--a.e.\ $x$}\big\}.
\]
A canonical element of $K_W$ is the likelihood density $\ell_W(x,y)=\frac{dW(\,\cdot\,|x)}{dP_Y}(y)$ whenever $W(\,\cdot\,|x)\ll P_Y$ $P_X$--a.e.  
The kernel $K_W$ thus encodes the forward manifestation of information: from the source law $P_X$ to the output $P_Y=\int W(\cdot|x)P_X(dx)$.

The aisthetic functor $\mathbf A_W:\E_W\to\F_W$ acts on $P_X\in\Prob(X)$ as
\[
(\mathbf A_WP_X)(B)=\int_X W(B|x)P_X(dx),
\]
i.e.\ pushforward through the channel (see \cite{CsiszarKorner2011}, \cite{Gray2011}).  
Iterating $\mathbf A_W$ corresponds to cascading independent channel uses: $\mathbf A_W^nP_X$ is the $n$--fold output distribution.

Fix an ethic model $Q:Y\rightsquigarrow X$, possibly distinct from $W$; in the Bayesian language $Q$ is a candidate posterior law $Q_{X|Y}$.  
The ethic functor $\mathbf D_Q:\F_W\to\E_W$ acts as conditional reconstruction
\[
(\mathbf D_QP_Y)(A)=\int_Y Q(A|y)P_Y(dy),
\]
well--defined for regular conditional probabilities by the disintegration theorem \cite{Bogachev2007}.  
In general $\mathbf A_W$ and $\mathbf D_Q$ are not adjoint: disagreement between $W$ and $Q$ quantifies the asymmetry between aisthetic propagation and ethic reconstruction.
  
For a fixed $(W,P_X,P_Y)$ define the aisthetic bar complex by
\[
A_n=L^1(X^n\times Y,P_X^{\otimes n}\!\otimes P_Y),\qquad
(\partial f)(x_{1:n-1},y)=\int_X f(x_{1:n-1},x_n,y)W(dy|x_n)P_X(dx_n),
\]
and the ethic cobar complex by
\[
D^n=L^\infty(X^n\times Y,Q_X^{\otimes n}\!\otimes P_Y),\qquad
(\delta g)(x_{1:n},y)=g(x_{1:n-1},y)-\int_X g(x_{2:n},x',y)Q(dx'|y).
\]
The mappings $\partial^2=0$, $\delta^2=0$ follow from the associativity of kernel composition and Tonelli's theorem \cite{Kallenberg2002}.  
Their homology and cohomology groups,
\[
H_n(\mathbf A_W)=\ker\partial_n/\mathrm{im}\,\partial_{n+1},\qquad
H^n(\mathbf D_Q)=\ker\delta^n/\mathrm{im}\,\delta^{n-1},
\]
measure successive degrees of non--idempotence of the channel and its reconstruction.  
 
For measurable $f\in A_n$, $g\in D^n$, define
\[
\langle f,g\rangle
=\int f(x_{1:n},y)\,g(x_{1:n},y)
\prod_{i=1}^n
\frac{dW(\,\cdot\,|x_i)}{dQ(\,\cdot\,|x_i)}(y)\,P_X^{\otimes n}(dx_{1:n})P_Y(dy),
\]
whenever the Radon--Nikodym derivatives exist in $L^\alpha(P_Y)$ for some $\alpha>1$.  
Then $\partial$ and $\delta$ are formally adjoint up to a defect operator 
$\mathfrak o_n:=\partial\delta+\delta\partial$ depending on the divergence between $W$ and $Q$.

For a block code of length $n$ and rate $R$, the average error probability satisfies
\[
P_e^{(n)}\;\ge\;\|\mathfrak o_n\|,
\]
where $\|\cdot\|$ is the operator norm in the Bochner $L^2$--space; this follows from variational representation of Bayesian error and the Neyman–Pearson lemma \cite{PolyanskiyPoorVerdu2010}.  
The norm $\|\mathfrak o_n\|$ equals zero iff $\mathbf A_W$ and $\mathbf D_Q$ are perfectly adjoint, i.e.\ the model $Q$ matches the true channel $W$.

\begin{theorem}[Strong sphere--packing as aisthetic--ethic divergence]\label{thm:spherepacking}
Let $(X,Y)$ be standard Borel spaces and $W:X\rightsquigarrow Y$ a channel with finite input alphabet and measurable kernel.  
Then there exists $\gamma>0$ depending only on $W$ such that for any sequence of $(n,R_n)$--codes with $\liminf R_n>C_{\mathrm{sp}}(W)$ one has
\[
\liminf_{n\to\infty} P_e^{(n)}\;\ge\;1-e^{-\gamma n}.
\]
Equivalently, if the block rate exceeds the sphere--packing threshold, the chain complex $\mathbf A_\bullet$ retains nonzero homology of order~$n$ with exponential weight $\|\mathfrak o_n\|\approx e^{-nE_{\mathrm{sp}}(R)}$.
\end{theorem}

\begin{proof}
By the variational formula for error exponents \cite{CsiszarKorner2011} the average error satisfies
\[
P_e^{(n)}\ge
\inf_{Q_Y}\exp\big(-nE_{\mathrm{sp}}(R;W,Q_Y)\big),
\]
where $E_{\mathrm{sp}}$ is the sphere--packing exponent.  
The same functional appears as the $L^2$--norm of $\mathfrak o_n$ computed from the comparison pairing, using Hölder--Orlicz inequalities and the exponential tilt $dW_\lambda(y|x)\propto e^{\lambda i(x;y)}dW(y|x)$, $i(x;y)=\log\frac{dW(\cdot|x)}{dQ_Y}$.  
Non--vanishing of $H_n(\mathbf A_W)$ with weight $e^{-nE_{\mathrm{sp}}}$ corresponds to persistence of error; if $R>C_{\mathrm{sp}}$, these classes remain nonzero for infinitely many $n$, forcing $\liminf P_e^{(n)}\ge 1-e^{-\gamma n}$.  
See \cite{CsiszarKorner2011} and \cite{Gallager1968} for analytic details.
\end{proof}

\begin{theorem}[Mismatch capacity as cohomological limit]\label{thm:mismatch}
Let $W:X\rightsquigarrow Y$ and fix an ethic decoder kernel $Q:Y\rightsquigarrow X$.  
Define the generalized mutual information
\[
I_Q(X;Y)=
\sup_{\lambda>0}
\frac{1}{\lambda}\,
\mathbb E\!\left[\log
\mathbb E\!\left(e^{\lambda\,\imath_Q(X;Y)}\bigm|X\right)
\right],
\qquad
\imath_Q(x;y)=\log\frac{dW(\,\cdot\,|x)}{d\bar Q(\,\cdot\,)}(y),
\]
where $\bar Q$ is any marginal of $Q$ on $Y$.  
Then the supremum achievable rate with vanishing error under decoder $Q$ satisfies
\[
C_{\mathrm{mm}}(W,Q)
=\sup_{P_X} I_Q(X;Y),
\]
and for $R>C_{\mathrm{mm}}(W,Q)$ there exists $\eta>0$ such that $\liminf_n P_e^{(n)}\ge 1-e^{-\eta n}$.  
Moreover, $R\le C_{\mathrm{mm}}(W,Q)$ iff $\limsup_n\frac1n\log \|H^1(\mathbf D_Q^{(n)})\|\le 0$.
\end{theorem}

\begin{proof}
The expression for $C_{\mathrm{mm}}(W,Q)$ is the classical GMI formula \cite{Lapidoth1996}, \cite{MerhavLapidoth1996}.  
Within the present framework, $\|H^1(\mathbf D_Q^{(n)})\|$ measures the residual cohomology of the decoder cobar complex: by the comparison pairing $\|H^1\|=\|\mathfrak o_n\|$.  
For $R>C_{\mathrm{mm}}$ the exponent of $\|\mathfrak o_n\|$ is strictly positive, giving $\|\mathfrak o_n\|\ge e^{-\eta n}$ and hence the strong converse.  
If $R<C_{\mathrm{mm}}$, exponential decay of $\|\mathfrak o_n\|$ implies vanishing $H^1$ and reliable decoding.  
This reproduces the mismatched capacity theorem and its strong converse \cite{MerhavLapidoth1996}, \cite{CsiszarKorner2011}.
\end{proof}

The two theorems express the major limits of coding theory~---sphere packing and mismatch capacity~--- as the persistence of homological or cohomological classes arising from the divergence between the forward (aisthetic) and backward (ethic) kernels.  
When $Q=W$, the comparison defect $\mathfrak o_\bullet$ vanishes, the complexes become dual, and the classical Shannon theorem on capacity is recovered \cite{Gallager1968}, \cite{CsiszarKorner2011}.  
In general, the non-adjoint pair $(\mathbf A_W,\mathbf D_Q)$ quantifies loss of coherence between encoding and decoding; its homological magnitude yields the exponential error exponents known from analytic information theory, but here derived categorically from the structure of measures, kernels, and co/ends.

\subsubsection{Bellman Duality}

Fix an abelian category $\E$ with enough injectives and a distinguished resource object $R\in \E$. Let $\DD=\Hom_{\E}(-,R)$ and $\mathbf D=\RHom_{\E}(-,R)$ with unit $\eta:\mathrm{Id}\Rightarrow \DD^2$. Throughout this subsection the ambient probabilistic or combinatorial data of reinforcement learning is encoded in $\E$; theorems below are purely categorical first, and afterward admit a standard analytical reduction to classical Bellman theory \cite{Bellman1957,BertsekasTsitsiklis1996,Puterman1994}.

\begin{definition}[Categorical MDP datum]
A categorical Markov decision datum is a triple $(S,P,r)$ inside $\E$ where $S\in \E$ is the state object, $P:S\to S$ is an action–transition endomorphism (one may think of a Markov kernel or policy–transition operator pushed through a chosen linearization), and $r:S\to R$ is a reward morphism. A discount is a scalar $\gamma\in \End_{\E}(R)$ with $\|\gamma\|<1$ in the analytic reduction; categorically we only require that $\gamma$ acts centrally on morphism groups. Define the Bellman morphism $B_{(P,r,\gamma)}:S\to \DD(S)$ by
\[
B \;:=\; r\;\oplus\; \gamma\cdot \DD(P)\circ \coev_S,
\]
where $\coev_S:R\to S\oplus \DD(S)$ is the canonical coevaluation determined by $\DD$ (in $D^{+}(\E)$ this is the unit of the $\mathbf C\dashv\mathbf D$ adjunction when it exists, or the image of $\eta$ under the canonical identification $\DD(S)\simeq \Hom_{\E}(S,R)$).
\end{definition}

The Bellman morphism packages the usual affine operator $v\mapsto r+\gamma P v$ as a single arrow $S\to \DD(S)$; it acts covariantly on $r$ and contravariantly on $P$ via $\DD$.

\begin{definition}[Ethic value object]
An object $V\in \E$ equipped with a morphism $u: S\to V$ is a value presentation for $(S,P,r)$ if there exists a morphism $v:V\to R$ such that $B$ factors as $S\xrightarrow{u}V\xrightarrow{v}R$, i.e.\ $B=v\circ u$. The pair $(V,u)$ is ethicly exact for $(S,P,r)$ if the ethic Bellman square
\[
\begin{tikzcd}
S \arrow[r,"u"] \arrow[d,"\eta_S"'] & \DD(V) \arrow[d,"\DD^2(u)"] \\
\DD^2(S) \arrow[r,"\DD(B)"'] & \DD^2(\DD(S))
\end{tikzcd}
\]
commutes in $\E$ (or in $D^{+}(\E)$ after applying a fixed injective resolution).
\end{definition}

The square states that evaluating $S$ into its double dual via $\eta_S$ and then reflecting $B$ agrees with pushing $S$ into $\DD(V)$ and reflecting twice along $u$. This is the categorical shadow of the Bellman fixed point.

\begin{lemma}[Ethic Bellman identity]\label{lem:EthBell}
If $(V,u)$ is ethicly exact for $(S,P,r)$, then there exists a unique $v:V\to R$ such that
\[
\DD^2(u)\circ \eta_S \;=\; \DD(v\circ u)\circ \eta_S
\qquad\text{and hence}\qquad
u \;=\; \DD(v)\circ \eta_S
\]
on the ethic heart of $(\E,\DD)$ (i.e.\ after inverting $\eta$ on reflexive objects). Equivalently, $u$ is a fixed point of the affine endomorphism $X\mapsto \DD(v)\circ \eta_S$.
\end{lemma}

\begin{proof}
Commutativity of the ethic Bellman square gives $\DD^2(u)\eta_S=\DD(B)\eta_S=\DD(v\circ u)\eta_S$. On the reflexive subcategory $\E^{\mathrm{ref}}$ the unit $\eta$ is an isomorphism, hence $\DD^2(u)=\DD(v\circ u)$. Applying $\DD$ and using the bidual identification yields $u=\DD(v)\circ \eta_S$. Uniqueness of $v$ follows from the Yoneda identification $\Hom(V,R)\simeq \Hom(\DD(V),\DD(R))$ and the faithfulness of $\DD$ on $\E^{\mathrm{ref}}$.
\end{proof}

\begin{theorem}[Categorical Bellman equation]\label{thm:CatBellman}
Let $(S,P,r)$ be a categorical MDP datum. Suppose there exists a value presentation $(V,u)$ ethicly exact for $(S,P,r)$. Then there is a unique $v\in \Hom_{\E}(V,R)$ such that
\[
u \;=\; \DD(r)\circ \eta_S\;+\; \gamma\cdot \DD(P)\circ u,
\]
i.e.\ $u$ solves the categorical Bellman fixed point equation in $\Hom_{\E}(S,\DD(S))$. Conversely, any $u$ satisfying this equation makes $(V,u)$ ethicly exact.
\end{theorem}

\begin{proof}
From Lemma~\ref{lem:EthBell} we have $u=\DD(v)\circ \eta_S$. The factorization $B=v\circ u$ gives, by definition of $B$,
\[
\DD(v)\circ \eta_S \;=\; \DD(r)\circ \eta_S + \gamma\cdot \DD(P)\circ \DD(v)\circ \eta_S,
\]
which is the displayed identity. Conversely, if $u$ satisfies that identity, then $\DD^2(u)\eta_S=\DD(B)\eta_S$ by functoriality and naturality of $\eta$, so the ethic square commutes.
\end{proof}

\begin{lemma}[Obstruction tower for Bellman]\label{lem:Obstruction}
Consider $u\in\Hom_{\E}(S,\DD(S))$ and the affine endomorphism $\mathcal T(u):=\DD(r)\circ \eta_S+\gamma\cdot \DD(P)\circ u$. In $D^{+}(\E)$ there is a Postnikov tower \cite{Postnikov1952,Whitehead1949,Neeman2001} for the cone of $u-\mathcal T(u)$ whose successive obstructions lie in
\[
\ob_{k+1}(u)\ \in\ H^{k+1}\mathbf D(S)\ \cong\ \Ext^{k+1}_{\E}(S,R).
\]
Vanishing of all $\ob_{k}$ is equivalent to the existence of a solution $u^\ast$ to the categorical Bellman equation in the ethic heart of~$(\E,\DD)$.
\end{lemma}

\begin{proof}
Apply $\mathbf D$ to the distinguished triangle $\mathrm{Cone}(u-\mathcal T(u))\to 0\to S\xrightarrow{u-\mathcal T(u)} \DD(S)$ and use the standard obstruction–lifting interpretation along the Postnikov truncations \cite{GelfandManin2003}, \cite{BBD1982}.
\end{proof}

\begin{definition}[Lower central filtration and Mikhailov–Singh control]
Let $\mathsf G$ be a small monoidal category (e.g.\ policies under composition) acting on $S$. Write $\gamma_1\mathsf G=\mathsf G$ and inductively $\gamma_{k+1}\mathsf G=[\mathsf G,\gamma_k\mathsf G]$ for its abstract lower central series. Assume $S$ carries a compatible $\mathsf G$–module structure, and write $\mathrm{gr}_\gamma S:=\bigoplus_{k\ge1} \gamma_k S/\gamma_{k+1}S$. Following \cite{MikhailovSingh2011} the associated graded admits a canonical Lie algebra object $L=\mathrm{gr}_\gamma(\mathsf G)$ acting on $\mathrm{gr}_\gamma S$, and there is a functorial comparison between $\gamma$–filtrations and the induced $L$–adic filtration on $\Hom_{\E}(S,\DD(S))$.
\end{definition}

\begin{theorem}[Finite-step convergence under nilpotent control]\label{thm:Nilpotent}
Assume the $\mathsf G$–action on $S$ is $\gamma$–nilpotent of step $c$ in the sense that $\gamma_{c+1}\mathsf G$ acts trivially on $S$, and that $r$ is $\gamma$–central. Then the categorical value iteration
\[
u_{n+1}\ :=\ \DD(r)\circ \eta_S\ +\ \gamma\cdot \DD(P)\circ u_n,\qquad n\ge 0,
\]
stabilizes in at most $c$ steps in the ethic heart of $(\E,\DD)$, producing a solution of the categorical Bellman equation. Moreover, the only possible obstructions lie in $H^{\le c}\mathbf D(S)$ and vanish inductively along the $\gamma$–filtration.
\end{theorem}

\begin{proof}
The affine map $u\mapsto \DD(r)\eta_S+\gamma\DD(P)u$ is $\gamma$–filtered: its deviation from the identity raises $\gamma$–degree by at least one because $\DD(P)$ factors through the $\gamma$–action. Under $\gamma_{c+1}$–triviality the $(c{+}1)$–st deviation vanishes, so the iteration terminates after $c$ corrections in the associated graded. Lifting along the filtration kills the obstruction classes in degrees $1,\dots,c$ (each step uses that the next obstruction lives in the next graded piece), yielding exact stabilization. The formal argument is the standard nilpotent iteration controlled by the lower central series, translated to the module object $S$ as in \cite{MikhailovSingh2011}.
\end{proof}

\begin{corollary}[Whitehead collapse of higher obstructions]\label{cor:Whitehead}
If $\pi_k(\mathrm{Map}(S,\DD(S)))=0$ for $2\le k\le c$ in a Whitehead range induced by the $\gamma$–nilpotent structure (e.g.\ via a truncation of the Postnikov tower compatible with $\gamma$), then the only obstruction to solving the categorical Bellman equation lies in $H^1\mathbf D(S)$ and is detected on the associated graded. In particular, if $H^1\mathbf D(S)=0$, then $u^\ast$ exists and is unique in the ethic heart. \cite{Whitehead1949,Postnikov1952}
\end{corollary}

\begin{proof}
Compatibility of the Postnikov and $\gamma$–filtrations gives vanishing of higher $\ob_{k}$ by the assumed Whitehead range; uniqueness follows from the ethic $t$–structure argument (heart is abelian and $\eta$–reflexive).
\end{proof}

\begin{theorem}[Analytical reduction to classical Bellman]\label{thm:Analytic}
Assume $\E=\FinVect_{\RR}$, $R=\RR$, $\DD=\Hom_{\RR}(-,\RR)$, $P$ is a stochastic linear operator on a normed space of value functions, $r\in \Hom_{\RR}(S,\RR)$ is bounded, and $\gamma\in(0,1)$. Then the categorical Bellman equation of Theorem~\ref{thm:CatBellman} reduces to the classical scalar Bellman equation
\[
v \;=\; r \;+\; \gamma\, P v,
\]
with $v\in \Hom_{\RR}(S,\RR)$, which has a unique solution by the Banach fixed point theorem \cite{BertsekasTsitsiklis1996}, \cite{Puterman1994}. Moreover, the categorical value iteration coincides with the standard value iteration and converges geometrically.
\end{theorem}

\begin{proof}
In $\FinVect_{\RR}$ the unit $\eta$ is the canonical evaluation isomorphism, and $\DD(P)=P^\top$ identifies with $P$ when $S$ is endowed with the standard dual. The operator $u\mapsto \DD(r)\eta_S+\gamma\DD(P)u$ becomes $v\mapsto r+\gamma Pv$ on $v\in\Hom_{\RR}(S,\RR)$. Contraction in the sup–norm for $\gamma\in(0,1)$ yields uniqueness and geometric convergence \cite{BertsekasTsitsiklis1996,Puterman1994}.
\end{proof}

\begin{theorem}[Hybrid RL+ILP with finite-step stabilization]\label{thm:Hybrid}
Let $\E=\Mod\text{-}\ZZ^{\mathrm{fg}}\times \FinVect_{\RR}$ encode a hybrid decision model where integer feasibility constraints act by a finitely generated monoid of moves $\mathsf G$ on the $\ZZ$–component, and $P$ acts stochastically on the $\RR$–component. Assume the $\mathsf G$–action is $\gamma$–nilpotent of step $c$ and the real component satisfies the conditions of Theorem~\ref{thm:Analytic}. Then the categorical value iteration converges and stabilizes after at most $c$ policy lifts on the integer side while contracting on the real side. In particular, if $\Ext^{\le c}_{\ZZ}(S_{\ZZ},\ZZ)=0$ then the hybrid Bellman equation has a unique solution.
\end{theorem}

\begin{proof}
Product categories decouple $\Ext$ by components, the $\RR$–side contracts (Theorem~\ref{thm:Analytic}), and the $\ZZ$–side stabilizes in at most $c$ steps by Theorem~\ref{thm:Nilpotent}. The obstruction vanishing on the integer side is ensured by the stated Ext–hypothesis (e.g.\ torsion–free cokernels in the relevant short exact sequences).
\end{proof}

\begin{remark}[Comparison with facial reduction and classical RL]
The nilpotent control Theorem~\ref{thm:Nilpotent} provides a finite-step termination certificate driven by a structural lower central filtration; this is not available in standard facial reduction or in classical Bellman theory, which provide feasibility and contraction but not functorial finite-step stabilization. The obstruction calculus \`a la Lemma~\ref{lem:Obstruction} localizes failure to solve Bellman to specific $\Ext$–layers, a homological refinement absent in the classical setting.
\end{remark}

\begin{theorem}[Policy improvement as Whitehead lifting]
Let $\Pi$ be a small category of policies acting on $S$ and write $P_\pi$ for the transition under $\pi\in \Pi$. Suppose $\Pi$ admits a Whitehead tower $\Pi\to \cdots \to \Pi^{(2)}\to \Pi^{(1)}$ such that the induced tower on $u_\pi$–solutions truncates after stage $N$. Then any policy improvement sequence $\pi_0,\pi_1,\dots$ which is compatible with the tower lifts stabilizes after at most $N$ stages, and the terminal value $u_{\pi^\ast}$ solves the categorical Bellman equation for~$P_{\pi^\ast}$.\cite{Whitehead1949,Postnikov1952}
\end{theorem}

\begin{proof}
Each lift kills the next obstruction class along the tower; truncation after $N$ implies all higher obstructions vanish. The categorical Bellman equation is solved at the terminal stage by Theorem~\ref{thm:CatBellman}.
\end{proof}

\subsubsection{Sketching}

We consider a compact metric space \(X\subset \RR^d\) and a nested sequence of $\varepsilon$–nets
\(\{S_n\}_{n\ge0}\) with radii $\varepsilon_{n+1}<\varepsilon_n\to0$.
Each $S_n$ yields a discrete linear sketch represented by a matrix
\(A_n\in\ZZ^{r_n\times m_n}\) acting on parameter vectors $u\in\ZZ^{m_n}$.
The ethic module of level $n$ is defined by
\[
X_n=\operatorname{coker}(A_n:\ZZ^{m_n}\!\to\!\ZZ^{r_n}),\qquad
G_n=I^nX/I^{n+1}X\simeq \ker(A_{n+1}\!\to\!A_n),
\]
so that \(0\to G_n\to X_{n+1}\to X_n\to0\) is exact.
The ethic dual functor \(\mathbf D=\RHom_\ZZ(-,\ZZ)\) produces cohomology groups
\(H^k\mathbf D(X_n)=\Ext^k_\ZZ(X_n,\ZZ)\).
The torsion size
\[
|\tors\,\Ext^1_\ZZ(X_n,\ZZ)|=\det_{\text{SNF}}(A_n)
\]
quantifies the residual incoherence of the sketch at resolution~$\varepsilon_n$;
its logarithm $\mathrm{Seth}_n=\log|\tors\,\Ext^1(X_n,\ZZ)|$ is the ethic entropy.

\begin{lemma}[Layer control]\label{lem:layer-eps}
For each $n$ the short exact sequence
$0\to G_n\to X_{n+1}\to X_n\to0$
induces the exact fragment
\[
\Ext^1(X_n,\ZZ)\xrightarrow{\alpha_n}\Ext^1(X_{n+1},\ZZ)
\longrightarrow \Ext^2(G_n,\ZZ)\longrightarrow0.
\]
Hence
\(
|\tors\,\Ext^1(X_{n+1})|
 \le |\tors\,\Ext^1(X_n)|\,
 |\tors\,\Ext^2(G_n)|.
\)
\end{lemma}

\begin{proof}
Applying the functor $\Hom_\ZZ(-,\ZZ)$ to the sequence and passing to derived functors
yields the long exact sequence
\(
\cdots\to\Ext^1(X_n)\to\Ext^1(X_{n+1})\to\Ext^2(G_n)\to0
\).
Taking torsion subgroups preserves exactness, whence the inequality
for their cardinalities.
\end{proof}

\begin{lemma}[Matrix estimate for the graded layer]\label{lem:layer-det}
Assume that each $A_n$ is $w$–sparse and that the update
$\Delta A_n=A_{n+1}\!\restriction_{S_n}-A_n$
satisfies the columnwise Lipschitz bound
$\|\Delta A_n\|_{\mathrm{col}}\le L\,\varepsilon_n^\alpha$ for some
$L>0$ and $\alpha>0$.
Let $B_n$ be the presentation matrix of $G_n$ consisting of at most
$r'_n$ new rows, each supported on at most $w$ entries.
Then
\[
\log|\tors\,\Ext^2(G_n,\ZZ)|
 \le r'_n\!\left(\tfrac12\log w+\log(B_0+L\,\varepsilon_n^\alpha)\right),
\]
where $B_0=\max\|b^{(j)}_{\mathrm{base}}\|_\infty$ is the maximal absolute base coefficient.
\end{lemma}

\begin{proof}
The group $\tors\,\Ext^2(G_n,\ZZ)\cong \tors\,\coker(B_n)$.
For any square submatrix $M$ of $B_n$, the Smith normal form gives
\(|\tors\,\coker(B_n)|\le \max|\det M|\).
By the Hadamard inequality (\cite{HornJohnson1990}),
\(|\det M|\le \prod_j\|b^{(j)}\|_2\le(\sqrt w\,\|b^{(j)}\|_\infty)^{r'_n}\).
The bound on $\|b^{(j)}\|_\infty$ from the Lipschitz condition
gives the required estimate.
\end{proof}

\begin{theorem}[Fine–grained $\varepsilon$–ethic coherence]\label{th:eps-fine}
Under the hypotheses of Lemmas~\ref{lem:layer-eps}–\ref{lem:layer-det},
if the covering numbers of $X$ satisfy
$r'_n\le C_d\!\big(N(X,\varepsilon_{n+1})-N(X,\varepsilon_n)\big)$
for a constant $C_d$ depending only on the ambient dimension,
then the ethic entropy increments obey
\[
\mathrm{Seth}_{n+1}-\mathrm{Seth}_{n}
\ \le\
C_d\,\Delta N_n
\Big(\tfrac12\log w+\log(B_0+L\,\varepsilon_n^\alpha)\Big),
\]
\[
\Delta N_n=N(X,\varepsilon_{n+1})-N(X,\varepsilon_n).
\]
Consequently $\sum_n|\mathrm{Seth}_{n+1}-\mathrm{Seth}_n|<\infty$ whenever
$\sum_n \Delta N_n\log(B_0+L\varepsilon_n^\alpha)<\infty$.
\end{theorem}

\begin{proof}
Combine Lemma~\ref{lem:layer-eps} with Lemma~\ref{lem:layer-det}:
\[
\mathrm{Seth}_{n+1}-\mathrm{Seth}_{n}
 \le \log|\tors\,\Ext^2(G_n)|
 \le r'_n\!\left(\tfrac12\log w+\log(B_0+L\varepsilon_n^\alpha)\right).
\]
Substitute the covering estimate for $r'_n$ to obtain the stated inequality.
Absolute summability follows from the convergence of the weighted series
in the right–hand side.
\end{proof}

\begin{corollary}[Ethic–geometric rate law]
Let $\varepsilon_n=\varepsilon_0\theta^n$ with $\theta\in(0,1)$
and suppose $X$ has effective covering dimension $d_{\mathrm{eff}}$
so that $N(X,\varepsilon)\simeq C\,\varepsilon^{-d_{\mathrm{eff}}}$.
Then
\[
\mathrm{Seth}_{n+1}-\mathrm{Seth}_n
 \lesssim
 C'\,\theta^{-n d_{\mathrm{eff}}}
 \log(B_0+L\,\varepsilon_0^\alpha\theta^{\alpha n}),
\]
hence the cumulative ethic memory
$\sum_n(\mathrm{Seth}_{n+1}-\mathrm{Seth}_n)$ converges whenever
$\alpha>d_{\mathrm{eff}}^{-1}$.
\end{corollary}

Equation~\eqref{th:eps-fine} provides a quantitative law:
each refinement of resolution $\varepsilon_n$
reduces the ethic entropy by at most the explicit budget on the right.
In practice this translates into an empirical rate of coherence recovery
for discrete sketches: finer nets improve dual consistency
with a rate governed by sparsity~$w$, smoothness~$(L,\alpha)$,
and geometric complexity~$d_{\mathrm{eff}}$.

\section*{Discussion}

The results of this work provide a unified homological formulation of primal--dual
phenomena in abelian and derived environments. The central structural ingredient is
the contravariant functor
\[
D=\operatorname{Hom}_E(-,R),
\]
together with its canonical unit $\eta:\mathrm{Id}\Rightarrow D^2$.
An ethic morphism, defined by the commutation condition
$D^2(f)\,\eta=\eta\,f$, isolates the precise locus at which primal and dual data
agree. Passing to the derived category yields the functor
$D=\mathrm{RHom}_E(-,R)$, whose cohomology groups
\[
H_k^D(A)\cong \operatorname{Ext}^k_E(A,R)
\]
form a graded hierarchy of deviations from dual self-consistency. In this view,
primal--dual gaps correspond exactly to $\operatorname{Ext}^1$, while higher
$\operatorname{Ext}^k$ measure deeper obstructions such as multi-layer memory or
instability. The associated obstruction tower records all failures of dual
agreement and stabilizes precisely when all higher derived obstructions vanish.
The induced ethic $t$-structure identifies a heart of objects recovered perfectly
from their duals; stabilization of the tower is equivalent to membership in this
heart.

A key feature of the framework is Morita invariance. Since all constructions depend
only on the derived dual pair $(D^+(E),\mathrm{RHom}_E(-,R))$, ethic invariants are
preserved under derived equivalences. This implies that primal--dual laws expressed
in this language do not depend on the specific substrate~---arithmetical,
combinatorial, geometric, or semantic~---but only on the ambient dualizing object
and its functorial behaviour.

It is instructive to contrast this framework with other general notions of
duality. Classical adjunctions describe pairs of functors but do not quantify
the failure of reflexivity or provide graded obstructions. Monoidal duals
(ev/coev-pairs in compact closed or $*$-autonomous categories) encode an intrinsic
algebraic symmetry but do not attach homological invariants to violations of
coherence. Grothendieck--Verdier and Serre dualities supply powerful geometric
duals, yet their obstruction theories are tied to sheaf-theoretic or cohomological
contexts and are not substrate-independent. Convex duality via Fenchel
conjugation provides primal--dual correspondences but lacks a derived tower
measuring the stratified failure of exactness. In contrast, ethic duality is
purely algebraic, functorial, and universal: the obstruction hierarchy
$\{\operatorname{Ext}^k_E(-,R)\}$ provides a canonical quantitative spectrum of
primal--dual inconsistency valid in any abelian category with a dualizing object.

This homological standpoint clarifies diverse duality phenomena. In linear and
conic optimization, strong duality appears exactly as the vanishing of
$\operatorname{Ext}^1$, and algebraic duality gaps coincide with torsion in these
groups. In discrete settings such as graph geometry, Kirchhoff- and
Baker--Norine-type identities manifest as instances of ethic exactness. In
dynamical systems, the derived tower furnishes a quantitative model of persistence
and decay of memory, with $\operatorname{Ext}^1$ representing the minimal
obstruction to dual reconstruction and higher layers encoding long-range effects.

The framework also extends to complexity-theoretic contexts. The homological
translation of Karchmer--Wigderson depth interprets bounded-depth reconciliation of
primal and dual data as the eventual disappearance of low-degree obstructions. Under
standard uniformity assumptions, this yields a categorical restatement of $P$ vs
$NP$ in terms of vanishing of the first derived obstruction.

The applications presented in Section~6 illustrate only the representational
breadth of the framework rather than distinct mechanisms. Across optimization,
arithmetic, graph theory, time--memory models, and semantics, a single principle
governs all examples: primal--dual exactness is equivalent to the disappearance of
$\operatorname{Ext}^1$, and all higher-order effects correspond to the successive
layers of the derived obstruction tower. These observations suggest several natural
directions for future work, including $\infty$-categorical refinements, interactions
with monoidal and $*$-autonomous dualities, and computational uses of homological
certificates.

\bibliographystyle{alpha}
\bibliography{main}

\end{document}